\documentclass[11pt]{article}
\usepackage{amsmath}
\usepackage{amssymb,amsbsy,amsthm,dsfont}
\usepackage{graphicx}
\usepackage[dvipsnames,hyperref]{xcolor}
\usepackage{caption}
\usepackage{subcaption}
\usepackage{enumerate}
\usepackage{cases}
\usepackage[margin=1in]{geometry}
\usepackage{hyperref}
\hypersetup{
	colorlinks,
	citecolor=blue,
	filecolor=red,
	linkcolor=blue,
	urlcolor=blue
}

\allowdisplaybreaks[1]

\numberwithin{equation}{section}

\DeclareMathOperator{\R}{\mathbb{R}} % Real numbers
\newcommand{\p}{\mathbb{P}} % Probability P
\newcommand{\E}{\mathbb{E}} % Expectation E
\DeclareMathOperator{\Var}{Var} % Variance Var
\DeclareMathOperator{\Cov}{Cov} % Covariance Cov 
\DeclareMathOperator{\Bin}{Bin} % Binomial Bin
\DeclareMathOperator*{\argmax}{arg\,max} % argmax
\DeclareMathOperator*{\argmin}{arg\,min} % argmin

\newcommand{\wh}{\widehat}  % widehat wh
\newcommand{\wt}{\widetilde} % widetilde wt 

\newcommand{\SBM}{\mathrm{SBM}} % Stochastic Block Model SBM
\newcommand{\MAP}{\mathrm{MAP}} % maximum a posteriori MAP

\newcommand{\overlap}{\mathsf{ov}}

\newcommand{\cut}[1]{}

\def\cA{{\mathcal A}}

\def\cC{{\mathcal C}}

\def\cE{{\mathcal E}}
\def\cF{{\mathcal F}}

\def\cS{{\mathcal S}}
\def\cT{{\mathcal T}}

\newtheorem{theorem}{Theorem}[section]
\newtheorem{lemma}[theorem]{Lemma}
\newtheorem{proposition}[theorem]{Proposition}

\newtheorem{conjecture}[theorem]{Conjecture}

\newtheorem{definition}[theorem]{Definition}

\newtheorem{remark}[theorem]{Remark}

\newtheorem{objective}{Objective}

\begin{document}

\title{Correlated Stochastic Block Models: Exact Graph Matching\\ with Applications to Recovering Communities}
\author{
	Mikl\'os Z.\ R\'acz
	\thanks{Princeton University; \url{mracz@princeton.edu}. Research supported in part by NSF grant DMS 1811724.} 
	\and
	Anirudh Sridhar
	\thanks{Princeton University; \url{anirudhs@princeton.edu}. Research supported in part by NSF grant DMS 1811724.}
}
\date{\today}

\maketitle

%%%%%%%%%%%%%%%%
%%% Abstract %%%
%%%%%%%%%%%%%%%%

\begin{abstract}
We consider the task of learning latent community structure from multiple correlated networks. First, we study the problem of learning the latent vertex correspondence between two edge-correlated stochastic block models, focusing on the regime where the average degree is logarithmic in the number of vertices. We derive the precise information-theoretic threshold for exact recovery: above the threshold there exists an estimator that outputs the true correspondence with probability close to 1, while below it no estimator can recover the true correspondence with probability bounded away from~0. As an application of our results, we show how one can exactly recover the latent communities using \emph{multiple} correlated graphs in parameter regimes where it is information-theoretically impossible to do so using just a single graph. 
\end{abstract}

%%%%%%%%%%%%%%%%
%%% Document %%%
%%%%%%%%%%%%%%%%

%%%%%%%%%%%%%%%%%%%%%%%%%%%%%%%%%%%%%%%%%%%%
\section{Introduction} \label{sec:intro} %%%
%%%%%%%%%%%%%%%%%%%%%%%%%%%%%%%%%%%%%%%%%%%%

Learning community structure in networks is a ubiquitous inference task in several domains, including biology \cite{chen2006detecting, Marcotte751}, sociology \cite{fortunato2010community}, and machine learning \cite{sahebi2011community, linden2003amazon, wu2015clustering}. Recent decades have therefore seen an explosion of work on the topic, leading to 
determining the fundamental information-theoretic limits 
for learning communities in probabilistic generative models~\cite{Abbe_survey,abbe2016exact,abbe2015community,mossel2016consistency,mossel2014reconstruction,massoulie2014community,mossel2018proof}, as well as algorithms that work well in practice~\cite{ruan2013efficient, Karypis98afast, dhillon2004kernelkm}. Typically, such algorithms only leverage the structure of the network (i.e., the configuration of node-node links). 
%Depending on the context, 
Increasingly, 
one often has access to {\it side information} that can greatly improve the performance of inference algorithms. 

There is a vast literature on designing algorithms that incorporate various types of side information to aid in recovering communities in networks. The works \cite{deshpande2018contextual,lu2020contextual,mossel2016local, kanade2016global,binkiewicz2017covariate,zhang2016community,bothorel2015clustering,saad2020sideinfo, ma2021community} leverage {\it node-level} information that is correlated with community memberships; 
here the 
%fundamental information-theoretic limits 
sharp limits 
for community detection were conjectured by  Deshpande~et~al.~\cite{deshpande2018contextual} 
and recently proven by Lu and Sen~\cite{lu2020contextual}. Another line of work \cite{han2015consistent,arroyo2020inference, paul2020null, paul2020spectral, lei2019consistent, ali2019latent, bhattacharya2020consistent, mayya2019mutual, ma2021community} 
recovers communities from a \emph{multi-layer} network, 
where the different layers are conditionally independent given the same community structure. 
%where the distribution of edges in each layer is influenced by the same underlying community structure. 
Recently Ma and Nandy~\cite{ma2021community} synthesized these two strands of literature. 

In contrast to prior work, we explore scenarios where the side information comes in the form of {\it multiple correlated networks}, which is natural in several domains including social networks~\cite{narayanan2009anonymizing, pedarsani2011privacy, korula2014efficient}, computational biology~\cite{singh2008global}, and machine learning~\cite{cour2007balanced,conte2004thirty}. In the context of social networks, for instance, many datasets are anonymized to protect the identity of users. Nevertheless, one may be able to infer additional information about users from additional networks by noting that the interaction patterns of the same set of users are likely to be similar across networks~\cite{narayanan2009anonymizing, pedarsani2011privacy,korula2014efficient}. In computational biology, an important goal is to study the functional properties of protein groups through a protein-protein interaction (PPI) network. Using the insight that functionally similar protein groups will have similar interaction structures, one can compare PPIs across species to infer protein functions \cite{singh2008global}. In all of these examples, an important task, commonly known as {\it graph matching}, is to synthesize the information from multiple correlated networks in a sensible manner.

To the best of our knowledge, we are the first to consider the use of multiple correlated networks 
%as side information 
for recovering communities. Specifically, we quantify, in an information-theoretic sense, how much information we can gain from correlated networks in order to infer community structure. To this end, we focus on correlated graphs $G_1$ and $G_2$ drawn marginally according to the stochastic block model (SBM), which is widely recognized as the canonical probabilistic generative model for networks with community structure~\cite{HLL83,Abbe_survey}. The reason for studying this probabilistic model is twofold. For one, it serves as a prototypical model for networks with community structure found in practice, hence the algorithms we develop will serve as a starting point for applications. Moreover, the SBM has well-defined ground-truth communities, so we can concretely study the correctness of algorithms in terms of whether the communities they output align with the ground truth.

\subsection{Models and Questions}
\label{sec:models}

\textbf{The stochastic block model (SBM).} 
The SBM is perhaps the simplest and most well-known probabilistic generative model for networks with community structure. It was initially proposed by Holland, Laskey, and Leinhardt~\cite{HLL83} and subsequently used as a theoretical testbed for evaluating clustering algorithms on average-case networks (see, e.g.,~\cite{dyer1989solution, bui1984graph, bopanna1987eigenvalues}). 
A striking fact about the SBM is that it exhibits sharp information-theoretic phase transitions for various inference tasks, leading to a \emph{precise} understanding of when community information can be extracted from network data. 
Such phase transitions were first conjectured by Decelle~et~al.~\cite{DKMZ11} and were subsequently proven rigorously by several authors~\cite{mossel2014reconstruction,massoulie2014community,mossel2018proof, abbe2016exact,mossel2016consistency,abbe2015community,bordenave2015nonbacktracking, Abbe_survey}. In summary, the SBM is a well-motivated and mathematically rich setting for studying inference tasks. 

In this work we focus on the simplest setting, a SBM with two symmetric communities. 
For a positive integer $n$ and $p,q \in [0,1]$, we construct $G \sim \mathrm{SBM}(n,p,q)$ as follows. 
The graph $G$ has $n$ vertices, 
labeled by the elements of $[n] : = \{1, \ldots, n \}$. 
Each vertex $i \in [n]$ has a community label $\sigma_{i} \in \{+1,-1\}$; 
these are drawn i.i.d.\ uniformly at random across all $i \in [n]$. 
Let $\boldsymbol{\sigma} : = \{ \sigma_i \}_{i = 1}^n$ be the vector of community labels, 
with the two communities given by the sets 
$V_+ : = \{ i \in [n]: \sigma_i = + 1 \}$ and 
$V_- : = \{ i \in [n] : \sigma_i = -1 \}$. 
Then, given the community labels $\boldsymbol{\sigma}$, 
the edges of $G$ are drawn independently across vertex pairs as follows. 
For distinct $i,j \in [n]$, 
if $\sigma_i \sigma_j = 1$, 
then the edge $(i,j)$ is in $G$ with probability $p$;  
else $(i,j)$ is in $G$ with probability $q$. 

\textbf{Community recovery.} 
Generally speaking, a community recovery algorithm takes as input $G$ (without knowledge of the community labels $\boldsymbol{\sigma}$) and outputs a community labeling $\wh{\boldsymbol{\sigma}}$. The {\it overlap} between the estimated labeling and the ground truth is given by 
\[
\overlap(\wh{\boldsymbol{\sigma}}, \boldsymbol{\sigma}) : = \frac{1}{n} \left|  \sum\limits_{i = 1}^n \wh{\boldsymbol{\sigma}}_i \boldsymbol{\sigma}_i  \right|.
\]
In the formula for the overlap, we take an absolute value since the labelings $\boldsymbol{\sigma}$ and $- \boldsymbol{\sigma}$ specify the same community partition (and it is only possible to recover $\boldsymbol{\sigma}$ up to its sign). Moreover, notice that $\overlap(\wh{\boldsymbol{\sigma}}, \boldsymbol{\sigma})$ is always between 0 and 1, where a larger value corresponds to a better match between the estimated communities and the ground truth. Indeed, the algorithm succeeds in exactly recovering the communities (i.e., $\wh{\boldsymbol{\sigma}} = \boldsymbol{\sigma}$ or $\wh{\boldsymbol{\sigma}} = - \boldsymbol{\sigma}$) if and only if $\overlap(\wh{\boldsymbol{\sigma}}, \boldsymbol{\sigma}) = 1$.

In the logarithmic degree regime---that is, when $p = \alpha \log (n) / n$ and $q = \beta \log (n) / n$ for some fixed constants $\alpha , \beta \ge 0$---it is well-known that there is a sharp information-theoretic threshold for exactly recovering communities in the SBM~\cite{abbe2016exact, mossel2016consistency,abbe2015community,Abbe_survey}. Specifically, if 
\begin{equation}
\label{eq:community_achievability_threshold}
\left|\sqrt{\alpha} - \sqrt{\beta} \right| > \sqrt{2},
\end{equation}
then exact recovery is possible: there is a polynomial-time algorithm which outputs an estimator~$\wh{\boldsymbol{\sigma}}$ satisfying $\lim_{n \to \infty} \p( \overlap ( \wh{\boldsymbol{\sigma}}, \boldsymbol{\sigma}) = 1) = 1$. On the other hand, if 
\begin{equation}
\label{eq:community_impossibility_threshold}
\left| \sqrt{\alpha} - \sqrt{\beta} \right| < \sqrt{2},
\end{equation}
then exact recovery is impossible: for \emph{any} estimator $\widetilde{\boldsymbol{\sigma}}$, 
we have that 
$\lim_{n \to \infty} \p (\overlap(\widetilde{\boldsymbol{\sigma}}, \boldsymbol{\sigma}) = 1) = 0$.

\textbf{Correlated SBMs.} 
The goal of our work is to understand how side information in the form of multiple {\it correlated} SBMs affects the threshold given by~\eqref{eq:community_achievability_threshold} and~\eqref{eq:community_impossibility_threshold}. 
To construct a pair of {correlated} SBMs, we define an additional parameter $s \in [0,1]$ which controls the level of correlation between the two graphs. Formally, we construct $(G_1, G_2) \sim \mathrm{CSBM}(n,p,q,s)$ as follows. 
First, generate a parent graph $G \sim \SBM(n,p,q)$, 
and let $\boldsymbol{\sigma}$ denote the community labels. 
Next, given~$G$, we construct $G_{1}$ and $G_{2}'$ by independent subsampling: 
each edge of $G$ is included in $G_{1}$ with probability $s$, independently of everything else, and non-edges of $G$ remain non-edges in $G_{1}$; 
we obtain $G_{2}'$ independently in the same fashion. 
Note that $G_{1}$ and $G_{2}'$ inherit the vertex labels from the parent graph $G$, 
and the community labels are given by $\boldsymbol{\sigma}$ in both graphs. 
Finally, we let $\pi_{*}$ be a uniformly random permutation of $[n]$, independently of everything else, 
and generate $G_{2}$ by relabeling the vertices of $G_{2}'$ according to $\pi_{*}$ 
(e.g., vertex $i$ in $G_{2}'$ is relabeled to $\pi_{*}(i)$ in $G_{2}$). 
This last step in the construction of $G_2$ reflects the observation that in applications, node labels are often obscured. This construction is visualized in Figure~\ref{fig:correlated_sbm}.

\begin{figure}[t]
    \centering
    \includegraphics[width=0.95\textwidth]{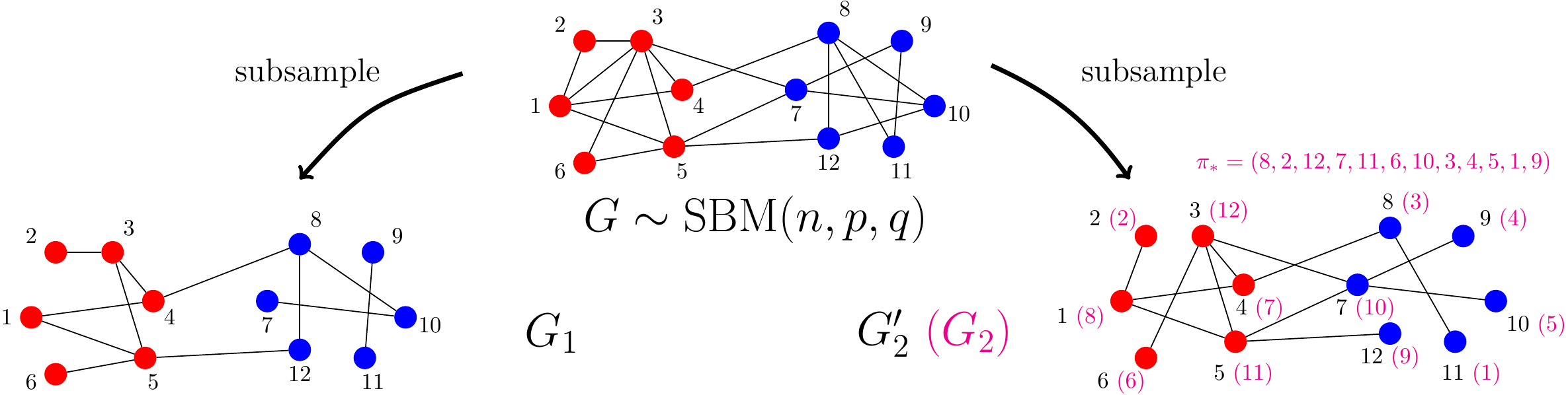}
    \caption{Schematic showing the construction of correlated SBMs (see text for details).}
    \label{fig:correlated_sbm}
\end{figure}

This model of correlated SBMs was first studied by Onaran, Erkip, and Garg~\cite{onaran2016optimal}. This process of generating correlated graphs (i.e., by first generating a parent graph, independently subsampling it, and randomly permuting the labels) is a natural and common approach for inducing correlation in the formation of edges, and has been employed to study correlated graphs from the Erd\H{o}s-R\'{e}nyi model (see, e.g., \cite{pedarsani2011privacy}, as well as further references in Section \ref{sec:related_work}), the Chung-Lu model~\cite{yu2021power}, and the preferential attachment model \cite{korula2014efficient}.

An important property of the construction is that marginally $G_{1}$ and $G_{2}$ are both SBMs. Specifically, since the subsampling probability is $s$, we have that 
$G_{1} \sim \SBM(n,ps,qs)$. 
% and $G_{2} \sim \SBM(n,ps,qs)$. 
In the logarithmic degree regime, where $p = \alpha \log(n)/n$ and $q = \beta \log(n) / n$, \eqref{eq:community_achievability_threshold} implies that the communities can be exactly recovered from $G_1$ 
%(or $G_2$) 
{\it alone} if 
\begin{equation}
\label{eq:community_recovery_achievability_single}
\left| \sqrt{\alpha} - \sqrt{\beta} \right| > \sqrt{\frac{2}{s}}.
\end{equation}
A central question of our work is how one can utilize the side information in $G_2$ to go {\it beyond} the single-graph threshold \eqref{eq:community_recovery_achievability_single}. This is formalized as follows. 

\begin{objective}[Exact community recovery]
\label{q:community_recovery}
Given $(G_1, G_2) \sim \mathrm{CSBM}\left( n, \frac{\alpha \log n}{n},\frac{\beta \log n }{n}, s \right)$, 
determine conditions on $\alpha$, $\beta$, and $s$ so that there exists an estimator $\wh{\boldsymbol{\sigma}} = \wh{\boldsymbol{\sigma}}(G_1, G_2)$ satisfying 
\[
\lim_{n \to \infty} \p( \overlap(\wh{\boldsymbol{\sigma}}, \boldsymbol{\sigma}) = 1) = 1.
\]
\end{objective}

A key observation is that if the latent correspondence $\pi_*$ is known, then one can readily improve the achievability region in~\eqref{eq:community_recovery_achievability_single}. Indeed, if $\pi_*$ is known, then one can reconstruct $G_2'$ from $G_2$. We can then construct a new graph $H_*$ by ``overlaying" $G_1$ and $G_2'$ (i.e., taking their union). Formally, $(i,j)$ is an edge in $H_*$ if and only if $(i,j)$ is an edge in $G_1$ or $G_2'$. An equivalent interpretation is that $(i,j)$ is an edge in the parent graph $G$ and it is included in either $G_1$ or $G_2'$ in the subsampling process. The probability that the edge is {\it not} included in either $G_1$ or $G_2'$ is $(1 - s)^2$, so 
it follows that 
$H_* \sim \mathrm{SBM} \left( n, \alpha (1 - (1 - s)^2) \log(n)/n, \beta (1 - (1 - s)^2 ) \log(n)/n \right)$. 
% \[
% H_* \sim \mathrm{SBM} \left( n, \alpha (1 - (1 - s)^2) \frac{\log n}{n}, \beta (1 - (1 - s)^2 ) \frac{\log n}{n} \right). 
% \]
By~\eqref{eq:community_achievability_threshold} it thus follows that exact community recovery is possible if
\begin{equation}
\label{eq:threshold_union_graph}
\left| \sqrt{\alpha } - \sqrt{\beta} \right| > \sqrt{ \frac{2}{1 - (1 - s)^2} }.
\end{equation}
Since $1 - (1- s)^2 > s$ for $s \in (0,1)$, \eqref{eq:threshold_union_graph} is a strict improvement over \eqref{eq:community_recovery_achievability_single}. Remarkably, this implies that if $\pi_*$ is known and if
\[
\sqrt{\frac{2}{s}} 
> \left| \sqrt{\alpha} - \sqrt{\beta} \right| 
> \sqrt{ \frac{2}{1 - (1 - s)^2}},
\]
then it is information-theoretically impossible to exactly recover $\boldsymbol{\sigma}$ from $G_1$ (or $G_2$) alone, but one can recover $\boldsymbol{\sigma}$ exactly by combining information from $G_1$ {\it and} $G_2$. To make this 
%analysis 
rigorous, we study when it is possible to exactly recover $\pi_*$ from $G_1$ and $G_2$. This task is known as {\it graph matching}. 

\begin{objective}[Exact graph matching]
\label{q:graph_matching}
Given 
$(G_1, G_2) \sim \mathrm{CSBM} \left( n, \frac{\alpha \log n}{n}, \frac{\beta \log n}{n}, s \right)$, 
determine conditions on $\alpha$, $\beta$, and $s$ so that there exists an estimator $\wh{\pi} = \wh{\pi}(G_1, G_2)$ satisfying 
\[
\lim_{n \to \infty} \p( \wh{\pi} = \pi_*) = 1.
\]
\end{objective}

While we have motivated graph matching as an intermediate step in recovering communities, it is an important problem in its own right, with applications to data privacy in social networks~\cite{narayanan2009anonymizing, pedarsani2011privacy}, protein-protein interaction networks~\cite{singh2008global}, and machine learning~\cite{cour2007balanced, conte2004thirty}, among others. 
In particular, 
%in terms of potential negative societal impacts, 
it is well known that graph matching algorithms can be used to de-anonymize social networks~\cite{narayanan2009anonymizing}, showing that anonymity is not the same as privacy. 
Studying the fundamental limits of when graph matching is possible can serve to highlight the precise conditions when anonymity can indeed guarantee privacy, and when additional safeguards are necessary. 

Although Objective \ref{q:graph_matching} has not been studied previously, 
%in the literature, 
there is strong evidence 
%that there is 
of 
a phase transition for exact recovery of $\pi_*$ in the logarithmic degree regime. 
In the special case of correlated Erd\H{o}s-R\'{e}nyi graphs---that is, when $\alpha = \beta$---Cullina and Kiyavash~\cite{cullina2016improved,cullina2018exact} showed 
that the maximum likelihood estimate exactly recovers $\pi_*$ with probability tending to 1 if 
$s^2 \alpha > 1$. 
When $\alpha \neq \beta$, and 
assuming that the community labels are \emph{known} in both graphs, 
Onaran, Garg, and Erkip~\cite{onaran2016optimal} showed that exact recovery of $\pi_*$ is possible if 
$s ( 1 - \sqrt{1 - s^2} ) 
\left( \alpha + \beta \right) / 2
> 3$. 
Cullina~et~al.~\cite{cullina2016simultaneous}, 
also assuming that community labels are known in both graphs, 
stated (without proof) that exact recovery is possible if 
$s^{2} (\alpha + \beta) / 2 > 2$. 
Since these works assume knowledge of community labels, 
it is unclear if these conditions allow to recover $\pi_{*}$ based on knowledge of only $G_{1}$ and $G_{2}$. 
Nevertheless, they suggest that exact graph matching may be possible in the logarithmic degree regime. 

Turning to impossibility results, 
in correlated Erd\H{o}s-R\'{e}nyi graphs, if $s^2 \alpha < 1$, then there is no estimator which exactly recovers $\pi_*$ with probability bounded away from zero~\cite{cullina2016improved,cullina2018exact,wu2021settling}.~For correlated SBMs, Cullina~et~al.~\cite{cullina2016simultaneous} showed that one cannot exactly recover $\pi_{*}$ when $s^{2} (\alpha + \beta)/2 < 1$.

In particular, for correlated Erd\H{o}s-R\'{e}nyi graphs the information-theoretic threshold $s^2 \alpha = 1$ is the {\it connectivity threshold} for the intersection graph %\footnote{Given two graphs $H_1, H_2$, the edge $(i,j)$ is in the intersection graph of $H_1$ and $H_2$ if and only if it is an edge in both $H_1$ and $H_2$.} 
of $G_1$ and $G_2'$. 
(Given two graphs $H_1$ and~$H_2$, the edge $(i,j)$ is in the intersection graph of $H_1$ and $H_2$ if and only if it is an edge in both $H_1$ and~$H_2$.)
%For $G_1$ and $G_2'$ generated through the correlated SBM distribution, 
For correlated SBMs 
the connectivity threshold for the intersection graph is 
\begin{equation}
\label{eq:sbm_connectivity}
s^2 
\left( 
\frac{\alpha + \beta}{2} 
\right) 
= 1. 
\end{equation}
%Generalizing ideas from the literature on the correlated Erd\H{o}s-R\'{e}nyi model 
This 
suggests that~\eqref{eq:sbm_connectivity} may be the information-theoretic threshold for exact recovery of $\pi_*$ for correlated~SBMs. 
Our main result, Theorem~\ref{thm:graph_matching}, shows that this is indeed the case.

\subsection{Results}

We now describe our results, which 
%provide answers to the questions posed by 
address 
Objectives~\ref{q:community_recovery} and~\ref{q:graph_matching}. 
In Section~\ref{subsec:graph_matching}, we precisely characterize the fundamental information-theoretic limits for exact graph matching, thereby fully achieving Objective \ref{q:graph_matching}. In Section \ref{sub:community_recovery}, we provide partial answers to Objective~\ref{q:community_recovery}; in particular, these provide the  information-theoretic threshold for exact community recovery in the regime where $s^2 (\alpha + \beta)/2 > 1$. Finally, in Section \ref{sub:multiple_community_recovery}, we extend the ideas of Section \ref{sub:community_recovery} to establish achievability and impossibility results for exact community recovery with $K$ correlated SBMs.

\subsubsection{Exact Graph Matching}
\label{subsec:graph_matching}

%In our first set of results, we derive the precise information-theoretic threshold, above which exact graph matching is possible and below which no estimator can output the correct vertex correspondence with probability bounded away from zero. 
%In the following result, we describe an estimator that correctly recovers the true correspondence above the information-theoretic threshold. 

We start with our main result, 
which determines the achievability region for exact graph matching in correlated SBMs, 
providing an estimator that correctly recovers the latent vertex correspondence above the information-theoretic threshold. 

\begin{theorem}
\label{thm:graph_matching} 
Fix constants $\alpha, \beta > 0$ and $s \in [0,1]$. 
Let $(G_1, G_2) \sim \mathrm{CSBM}\left( n,\frac{\alpha \log n}{n} ,\frac{\beta \log n}{n}, s \right)$.
Let $\wh{\pi}(G_1, G_2)$ be a vertex mapping that maximizes the number of agreeing edges between $G_1$ and~$G_2$ (that is, the number of matched pairs of vertices for which an edge exists between them in both graphs). If 
\begin{equation}
\label{eq:recovery_condition}
s^2 
\left( 
\frac{\alpha + \beta}{2} 
\right) 
> 1,
\end{equation}
then
\[
\lim\limits_{n \to \infty} \p \left( \wh{\pi}(G_1, G_2) = \pi_{*} \right) = 1.
\]
\end{theorem}

We remark that the estimator $\wh{\pi}$ used in Theorem \ref{thm:graph_matching} is a natural and well-motivated estimator for the latent mapping $\pi_*$. It was first considered by Pedarsani and Grossglauser \cite{pedarsani2011privacy} in the context of the correlated Erd\H{o}s-R\'{e}nyi model, where it is the maximum a posteriori (MAP) estimate~\cite{cullina2016improved, cullina2018exact, onaran2016optimal}. As a result, it achieves the information-theoretic threshold for exact recovery of $\pi_{*}$ in the correlated Erd\H{o}s-R\'{e}nyi model \cite{cullina2018exact, wu2021settling}. 
This estimator has also been studied in the context of correlated SBMs by Onaran, Erkip, and Garg \cite{onaran2016optimal}; they show that if the commmunity labels of all vertices in $G_1$ and $G_2$ are known, then the permutation which maximizes the number of agreeing edges {\it and} is consistent with the community labels (i.e., does not map a vertex with label $+1$ to a vertex of label $-1$) succeeds in recovering $\pi_*$ exactly, provided that the (suboptimal) condition 
$s ( 1 - \sqrt{1 - s^2} ) 
\left( \alpha + \beta \right) / 2
> 3$ 
holds. Theorem~\ref{thm:graph_matching} improves on this result using a more refined analysis, 
obtaining the optimal condition~\eqref{eq:recovery_condition}, and not assuming any  knowledge of community~labels.

The next result establishes a converse to Theorem~\ref{thm:graph_matching}. This was previously proven in~\cite{cullina2016simultaneous}.

\begin{theorem}
\label{thm:graph_matching_impossibility} 
Fix constants $\alpha, \beta > 0$ and $s \in [0,1]$. 
Let $(G_1, G_2) \sim \mathrm{CSBM}\left( n, \frac{\alpha \log n}{n}, \frac{\beta \log n}{n}, s \right)$ and suppose that 
\begin{equation}
\label{eq:impossibility_condition}
s^2 \left( \frac{\alpha + \beta}{2} \right) < 1.
\end{equation}
Then for any estimator $\widetilde{\pi}(G_1, G_2)$, we have that
$\lim\limits_{n \to \infty} \p ( \widetilde{\pi}(G_1, G_2) = \pi_*) = 0$. 
\end{theorem}

Together, Theorems \ref{thm:graph_matching} and \ref{thm:graph_matching_impossibility} establish the fundamental information-theoretic limits for exact recovery of $\pi_*$. This is the natural generalization of the corresponding results for correlated Erd\H{o}s-R\'{e}nyi graphs: 
when $\alpha = \beta$, the same estimator $\wh{\pi}$ succeeds if $s^2 \alpha > 1$, else if $s^2 \alpha < 1$, then no estimator can exactly recover $\pi_*$ with probability bounded away from zero~\cite{cullina2018exact, wu2021settling}.  

An overview of the proofs of Theorems \ref{thm:graph_matching} and \ref{thm:graph_matching_impossibility} is given in Section~\ref{sec:proofs_overview}. 

\subsubsection{Exact Community Recovery} 
\label{sub:community_recovery}

We now turn to exact community recovery with two correlated SBMs, 
formalizing the arguments of Section~\ref{sec:models}. 
The strategy is to first perform exact graph matching, 
then to combine the two graphs by taking their union with respect to the matching, 
and finally to run an exact community recovery algorithm on this new graph. 

\begin{theorem}
\label{thm:community_recovery} 
Fix constants $\alpha, \beta > 0$ and $s \in [0,1]$. 
Let $(G_1, G_2) \sim \mathrm{CSBM}\left( n,\frac{\alpha \log n}{n} ,\frac{\beta \log n}{n}, s \right)$. 
Suppose that $s^{2} \left( \alpha + \beta \right) / 2 > 1$ and 
\begin{equation}
\label{eq:community_achievability}
\left| \sqrt{\alpha} - \sqrt{\beta} \right| > \sqrt{ \frac{2}{1 - (1 - s)^2} }.
\end{equation}
Then there is an estimator $\wh{\boldsymbol{\sigma}} = \wh{\boldsymbol{\sigma}}(G_1, G_2)$ such that 
\[
\lim\limits_{n \to \infty} \p \left( \overlap \left(\wh{\boldsymbol{\sigma}}, \boldsymbol{\sigma} \right) = 1 \right) = 1.
\]
\end{theorem}

The proof readily follows from Theorem~\ref{thm:graph_matching} and existing results on exact community recovery in the SBM~\cite{abbe2016exact, mossel2016consistency,abbe2015community,Abbe_survey}. 
\begin{proof}
%We first define some notation. 
Given a permutation $\pi$ mapping $[n]$ to $[n]$, we let $G_1 \lor_{\pi} G_2$ be the {\it union graph with respect to $\pi$}, so that $(i,j)$ is an edge in $G_1 \lor_{\pi} G_2$ if and only if $(i,j)$ is an edge in $G_1$ or $(\pi(i), \pi(j))$ is an edge in $G_2$. In the special case where $\pi = \pi_*$, $H_{*} := G_1 \lor_{\pi_{*}} G_2$ is the subgraph of the parent graph $G$ consisting of edges that are in either $G_1$ or $G_2'$. It is readily seen that
\begin{equation}
\label{eq:H_star_distribution}
H_{*} \sim \mathrm{SBM} \left( n, \alpha(1 - (1 - s)^2) \frac{\log n}{n} , \beta (1 - (1 - s)^2) \frac{\log n}{n} \right).
\end{equation}
The algorithm we study first computes $\wh{\pi}(G_1, G_2)$ according to Theorem~\ref{thm:graph_matching}. 
We then pick any community recovery algorithm that is known to succeed until the 
% fundamental 
information-theoretic limit, 
and run it on $\wh{H} : = G_1 \lor_{\wh{\pi}} G_2$; we denote the result of this algorithm by $\wh{\boldsymbol{\sigma}}( \wh{H})$. We can then write
\begin{align*}
\p( \overlap(\wh{\boldsymbol{\sigma}}(\wh{H}), \boldsymbol{\sigma}) \neq 1 ) & \le \p ( \{ \overlap(\wh{\boldsymbol{\sigma}}(\wh{H}), \boldsymbol{\sigma}) \neq 1 \} \cap \{ \wh{H} = H_* \} ) + \p(\wh{H} \neq H_* ) \\
& \le \p( \overlap(\wh{\boldsymbol{\sigma}}(H_*), \boldsymbol{\sigma}) \neq 1) + \p( \wh{\pi} \neq \pi_\star ),
\end{align*}
where, to obtain the inequality in the second line, we have used that $\wh{\boldsymbol{\sigma}}(\wh{H}) = \wh{\boldsymbol{\sigma}}(H_*)$ on the event $\{\wh{H} = H_* \}$, and that $\wh{H} \neq H_*$ implies $\wh{\pi} \neq \pi_*$. 
Since exact community recovery on $H_{*}$ is possible when~\eqref{eq:community_achievability} holds~\cite{abbe2016exact, mossel2016consistency,abbe2015community,Abbe_survey}, 
we know that 
$\p(\overlap(\wh{\boldsymbol{\sigma}}(H_*), \boldsymbol{\sigma}) \neq 1) \to 0$ as $n \to \infty$. 
In light of Theorem~\ref{thm:graph_matching} we also have that $\p( \wh{\pi} \neq \pi_*) \to 0$ when $s^{2}(\alpha + \beta)/2 > 1$, concluding the proof. 
\end{proof}

\begin{figure}[t]
    \centering
    \begin{subfigure}{0.3\textwidth}
        \centering
        \includegraphics[width=\textwidth]{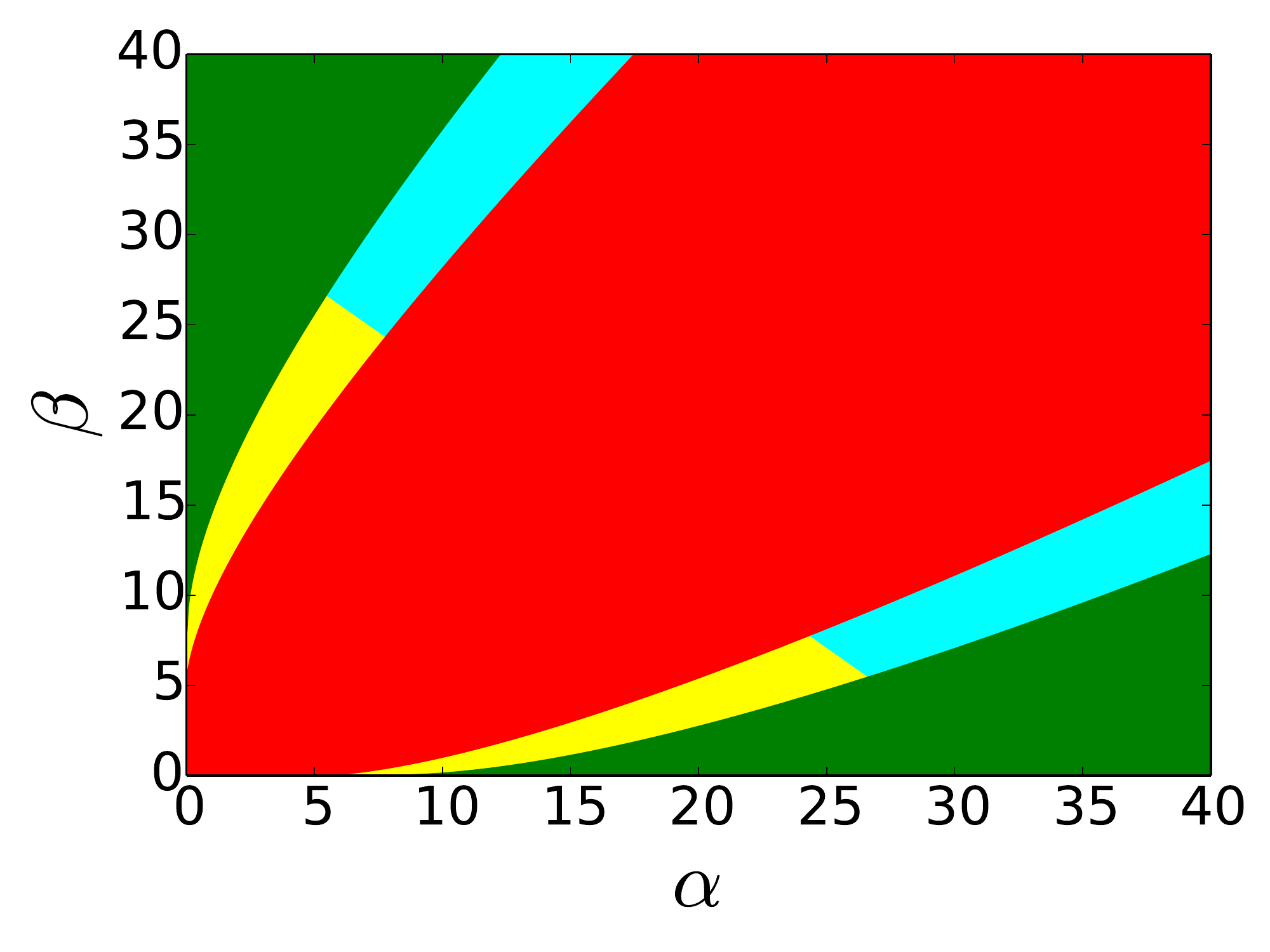}
        \caption{Fixed $s=0.25$.}
    \end{subfigure}
    \quad 
    \begin{subfigure}{0.3\textwidth}
        \centering
        \includegraphics[width=\textwidth]{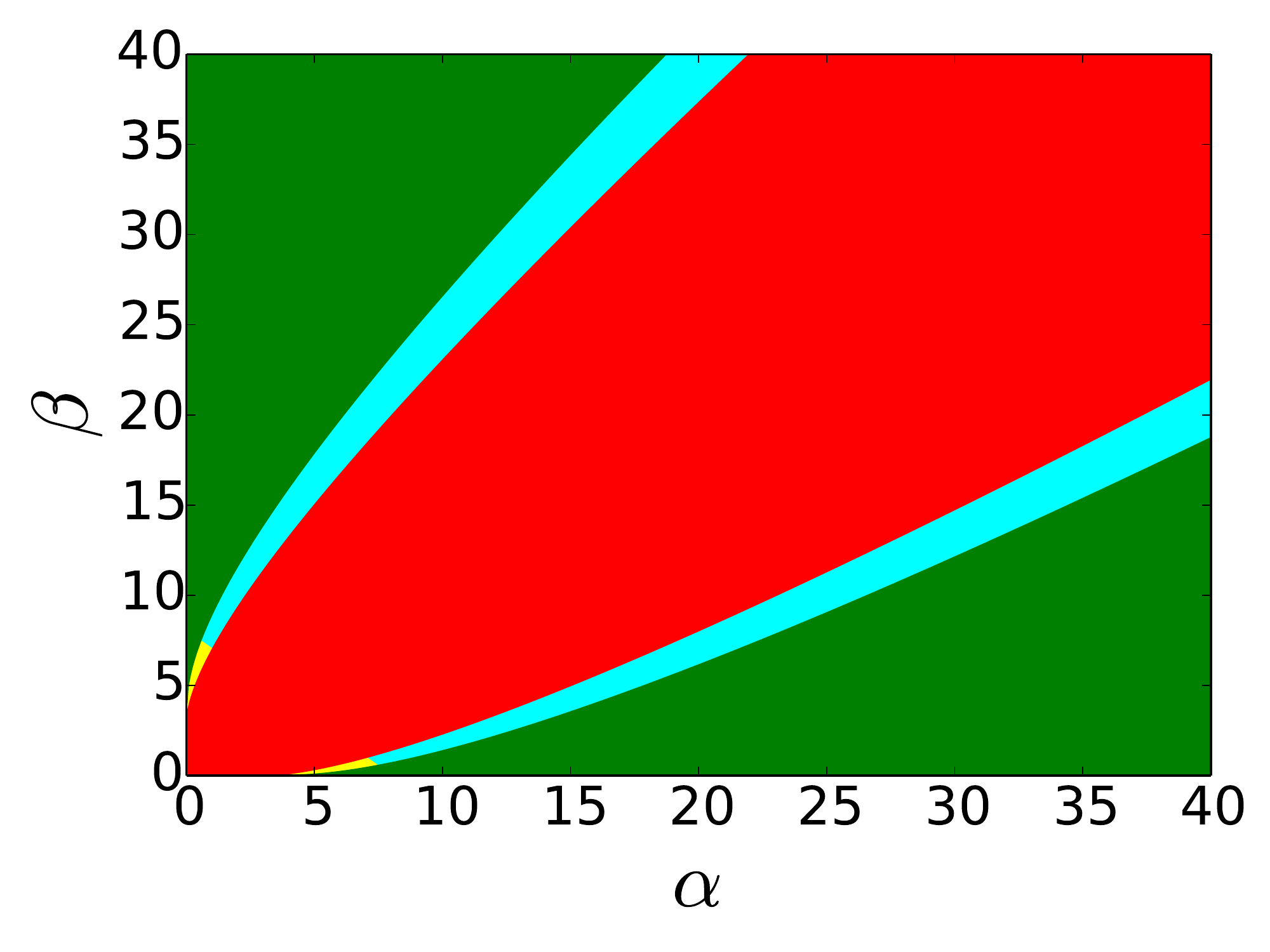}
        \caption{Fixed $s=0.5$.}
    \end{subfigure}
    \quad 
    \begin{subfigure}{0.3\textwidth}
        \centering
        \includegraphics[width=\textwidth]{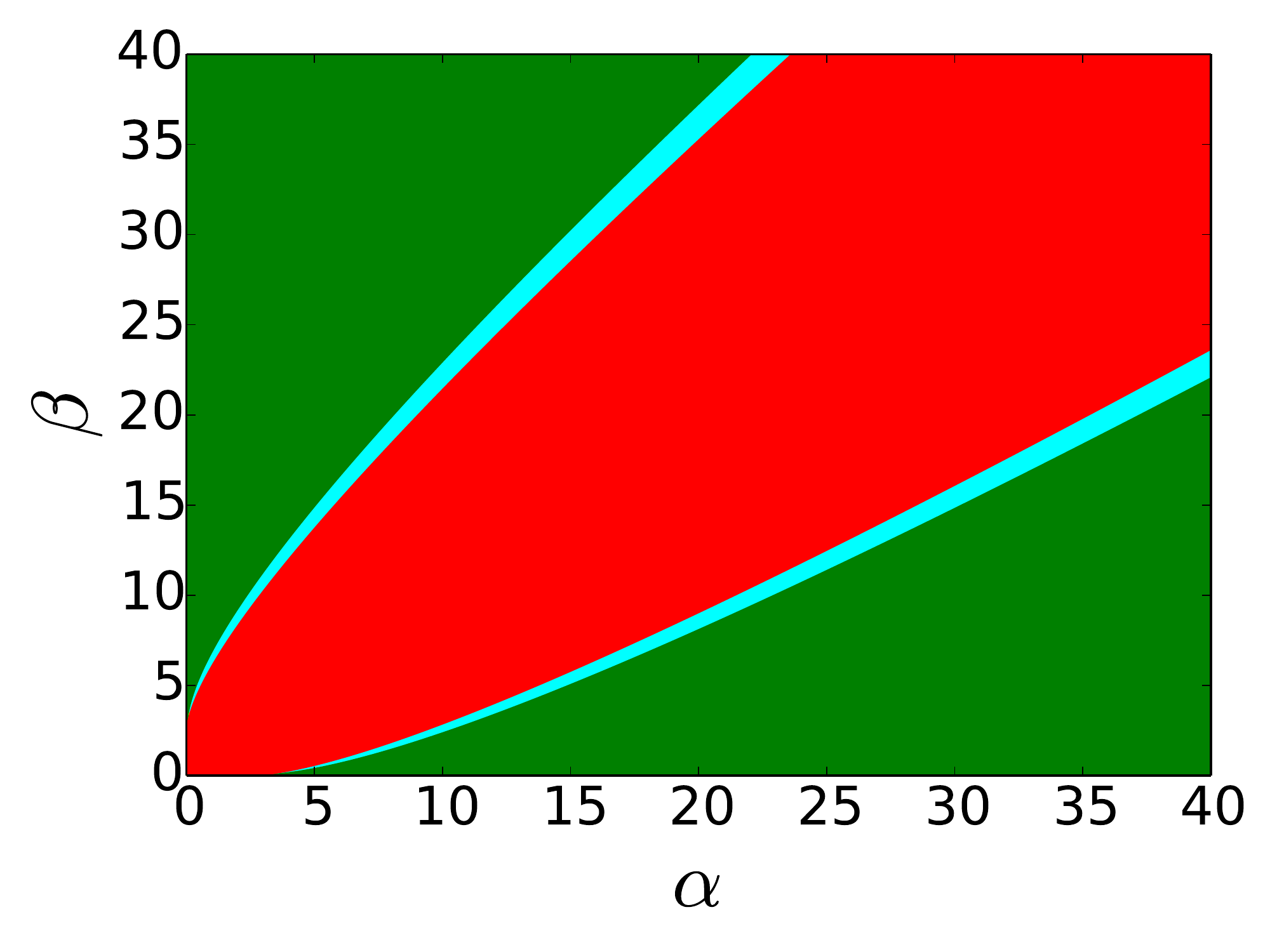}
        \caption{Fixed $s=0.75$.}
    \end{subfigure}
    \caption{Phase diagrams for exact community recovery for fixed $s$, with $\alpha \in [0,40]$ and $\beta \in [0,40]$ on the axes. 
    \emph{Green region:} exact community recovery is possible from $G_{1}$ alone; 
    \emph{Cyan region:} exact community recovery is impossible from $G_{1}$ alone, but it is possible from $(G_{1},G_{2})$; 
    \emph{Yellow region:} exact community recovery is impossible from $G_{1}$ alone, unknown if it is possible from $(G_{1},G_{2})$; 
    \emph{Red region:} exact community recovery is impossible from $(G_{1},G_{2})$. 
    }
    \label{fig:phase_diagram_fixed_s}
\end{figure}

\begin{figure}[t]
    \centering
    \begin{subfigure}{0.3\textwidth}
        \centering
        \includegraphics[width=\textwidth]{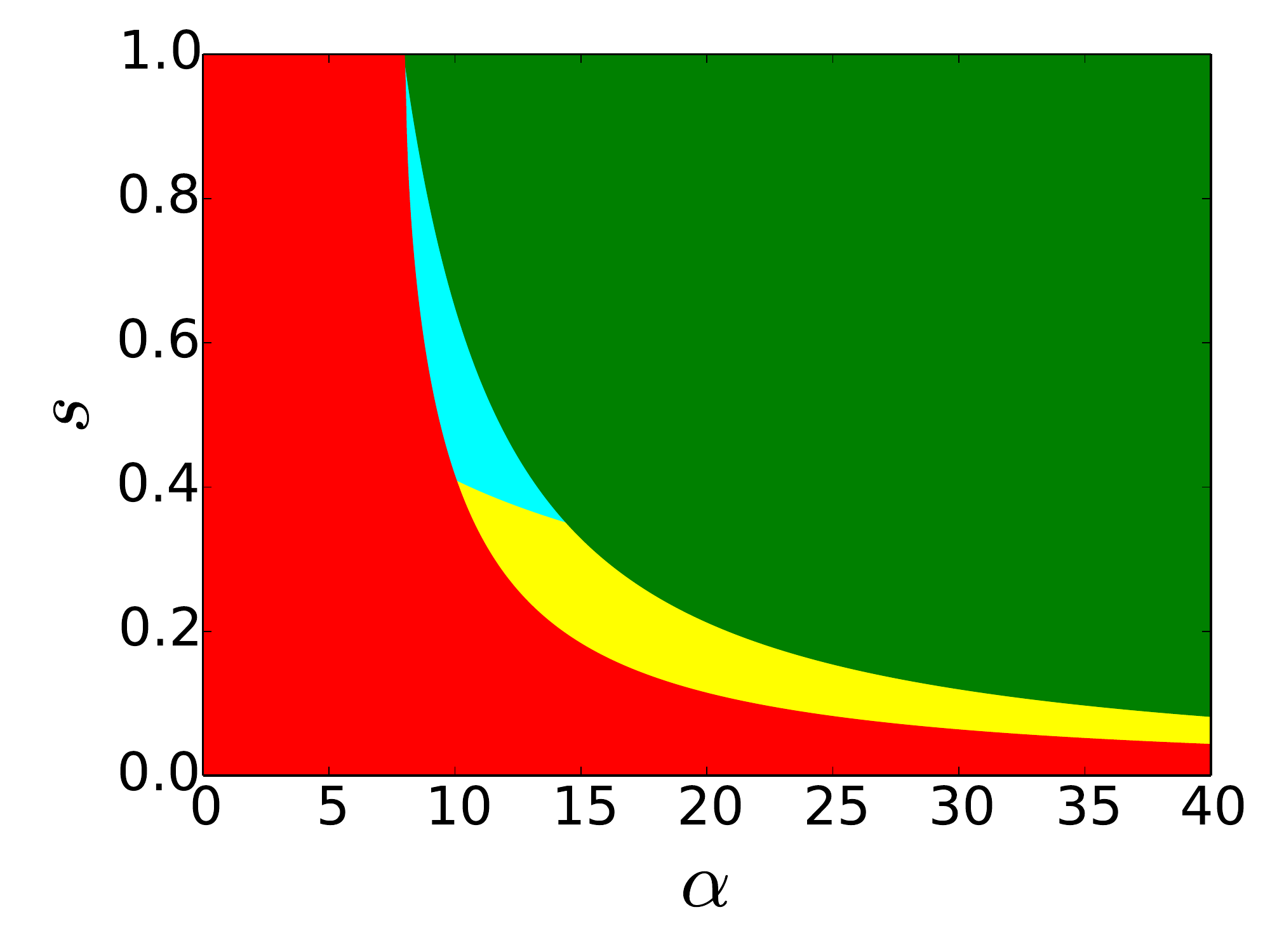}
        \caption{Fixed $\beta=2$.}
    \end{subfigure}
    \quad 
    \begin{subfigure}{0.3\textwidth}
        \centering
        \includegraphics[width=\textwidth]{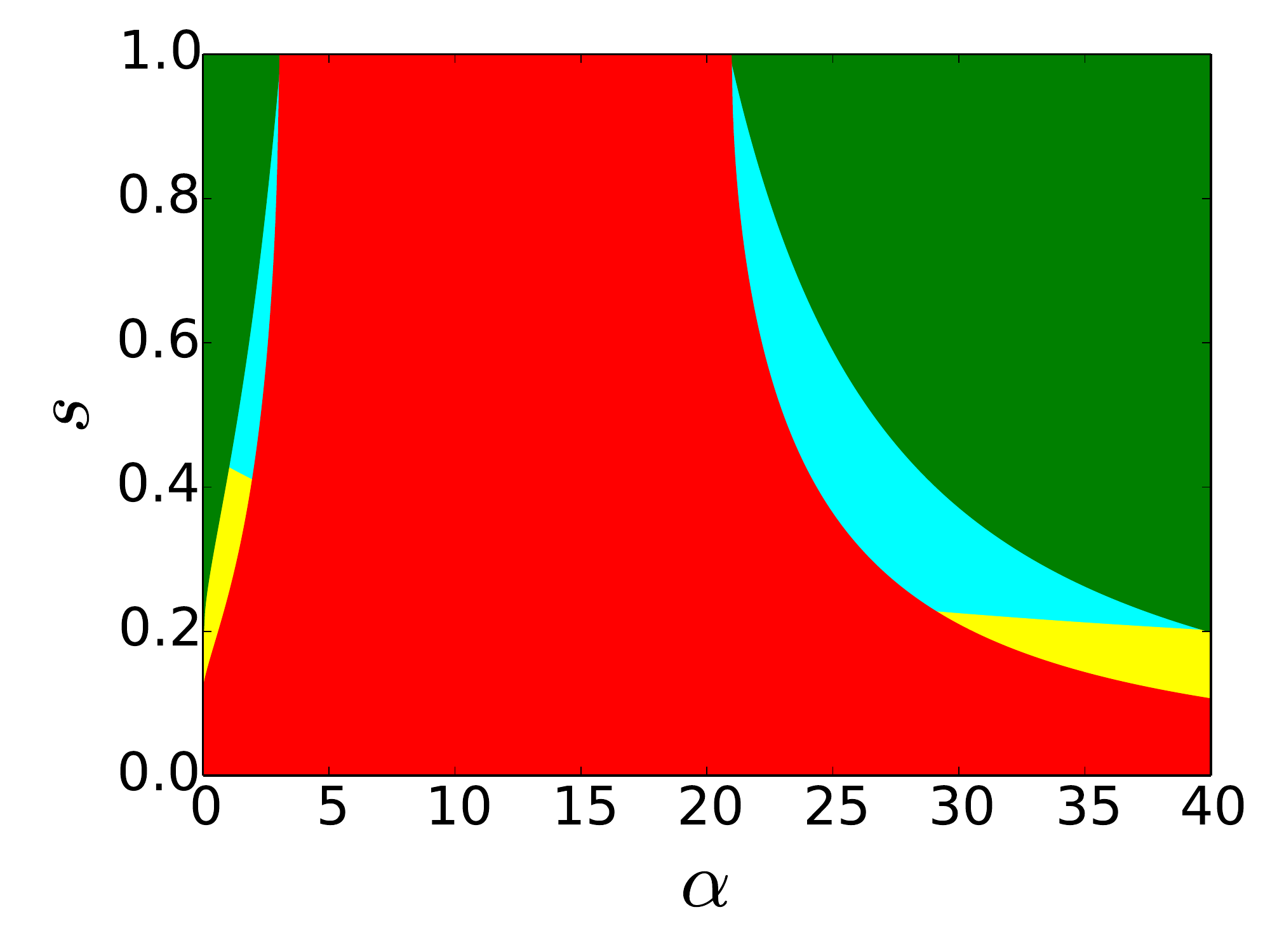}
        \caption{Fixed $\beta = 10$.}
    \end{subfigure}
    \quad 
    \begin{subfigure}{0.3\textwidth}
        \centering
        \includegraphics[width=\textwidth]{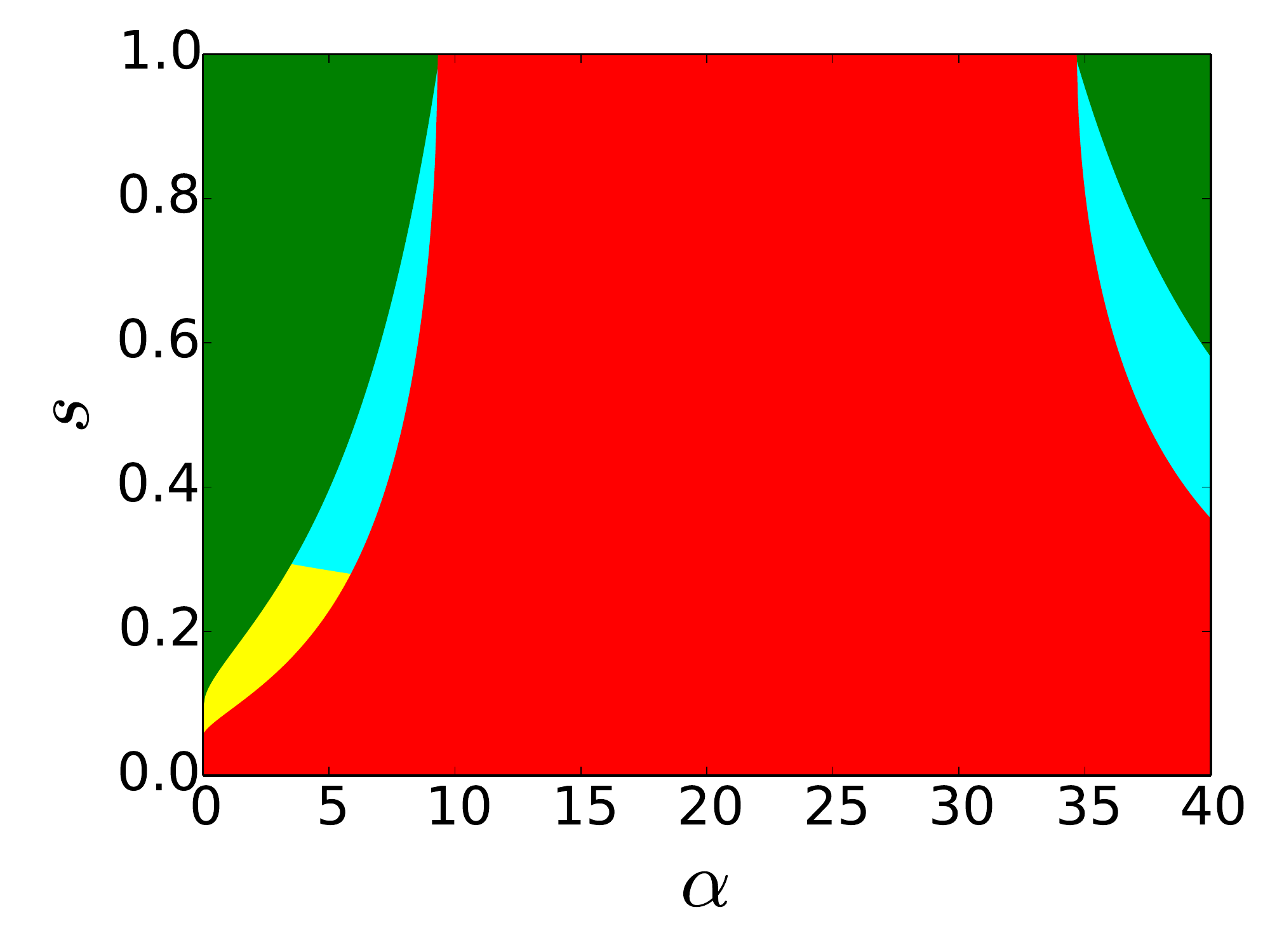}
        \caption{Fixed $\beta = 20$.}
    \end{subfigure}
    \caption{Phase diagrams for exact community recovery for fixed $\beta$, with $\alpha \in [0,40]$ and $s \in [0,1]$ on the axes. (Colors as in  Fig.~\ref{fig:phase_diagram_fixed_s}.)}
    \label{fig:phase_diagram_fixed_beta}
\end{figure}

By the discussion in Section~\ref{sec:models}, 
Theorem~\ref{thm:community_recovery} establishes the existence of a region of the parameter space 
where 
(i) there exists an algorithm that can exactly recover the communities using \emph{both} $G_{1}$ and $G_{2}$, 
but 
(ii) it is information-theoretically impossible to do so using $G_{1}$ (or $G_{2}$) alone. 
Figures~\ref{fig:phase_diagram_fixed_s},~\ref{fig:phase_diagram_fixed_beta}, and~\ref{fig:phase_diagram_fixed_ratio} illustrate phase diagrams of the parameter space, 
where this region is highlighted in 
cyan.

To complement the achievability result of Theorem~\ref{thm:community_recovery}, our next result provides a condition under which exact community recovery is information-theoretically impossible. 

\begin{theorem}
\label{thm:community_recovery_impossibility} 
Fix constants $\alpha, \beta > 0$ and $s \in [0,1]$. 
Let $(G_1, G_2) \sim \mathrm{CSBM} \left( n, \alpha \frac{\log n}{n}, \beta \frac{\log n}{n}, s \right)$ and suppose that 
\begin{equation}
\label{eq:union_graph_impossibility}
\left| \sqrt{\alpha} - \sqrt{\beta} \right| < \sqrt{ \frac{2}{1 - (1 - s)^2} }.
\end{equation}
Then for any estimator $\widetilde{\boldsymbol{\sigma}} = \widetilde{\boldsymbol{\sigma}}(G_1, G_2)$, we have that
$
\lim\limits_{n \to \infty} \p ( \overlap( \widetilde{\boldsymbol{\sigma}} ,\boldsymbol{\sigma}) = 1) = 0.
$
\end{theorem}

The idea behind the proof is a simulation argument. %, as follows. 
Recall $H_{*} := G_1 \lor_{\pi_{*}} G_2$ 
% and its distribution 
from the proof of Theorem~\ref{thm:community_recovery}, 
and note that 
$H_{*} \sim \mathrm{SBM} \left( n, \alpha (1 - (1 - s)^2) \log(n)/n, \beta (1 - (1 - s)^2 ) \log(n)/n \right)$. 
From $H_{*}$ it is possible to simulate $(G_{1}, G_{2})$, 
and so if exact community recovery is possible given $(G_{1}, G_{2})$, 
then it is also possible given $H_{*}$. 
However, it is known~\cite{abbe2016exact, mossel2016consistency,abbe2015community,Abbe_survey} that exact community recovery is not possible from $H_{*}$ if~\eqref{eq:union_graph_impossibility} holds. 
See Section~\ref{sec:impossibility_community_recovery} for the full proof.

\begin{figure}[t!]
    \centering
    \begin{subfigure}{0.3\textwidth}
        \centering
        \includegraphics[width=\textwidth]{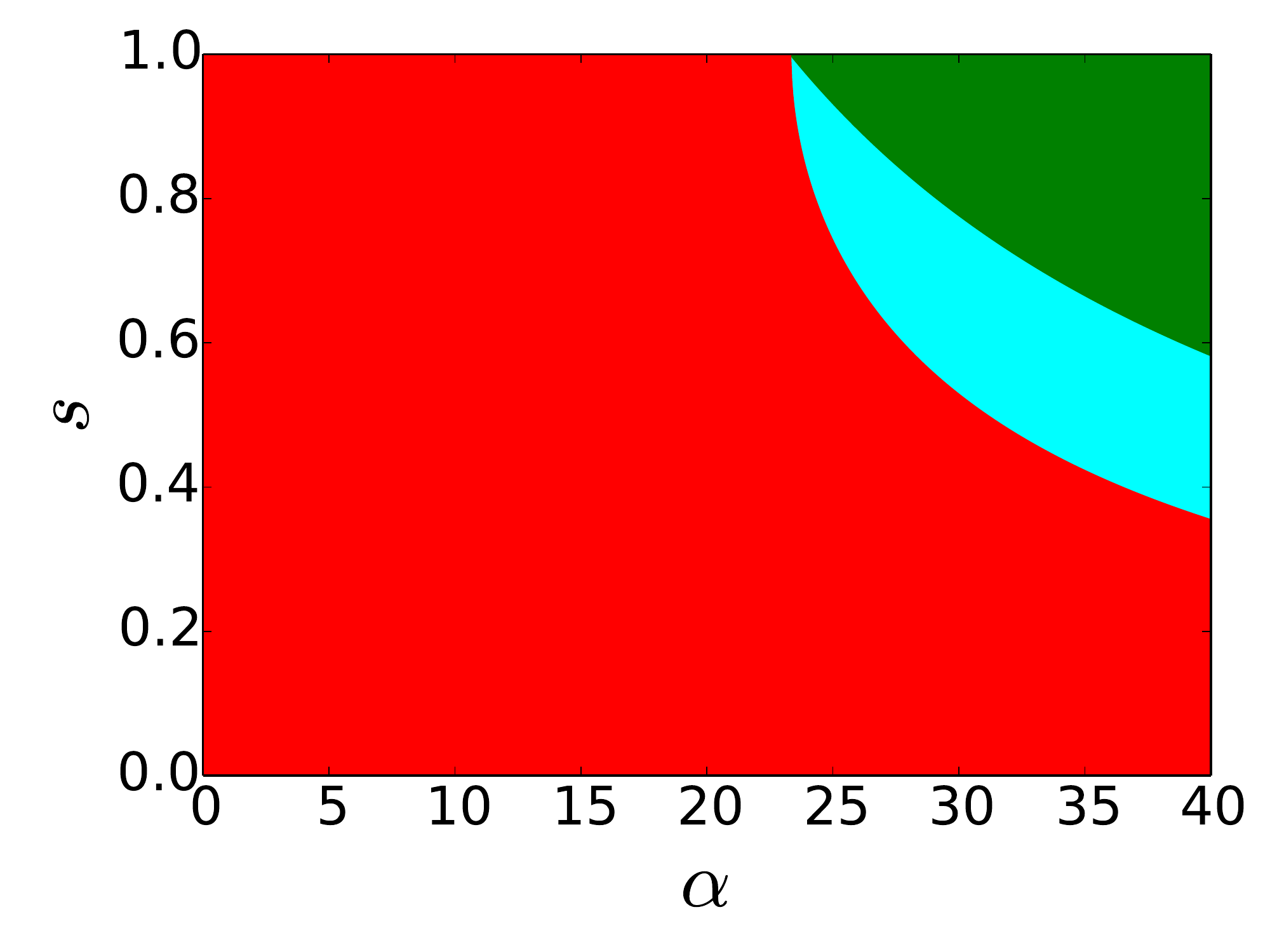}
        \caption{Fixed $\alpha/\beta=2$.}
    \end{subfigure}
    \quad 
    \begin{subfigure}{0.3\textwidth}
        \centering
        \includegraphics[width=\textwidth]{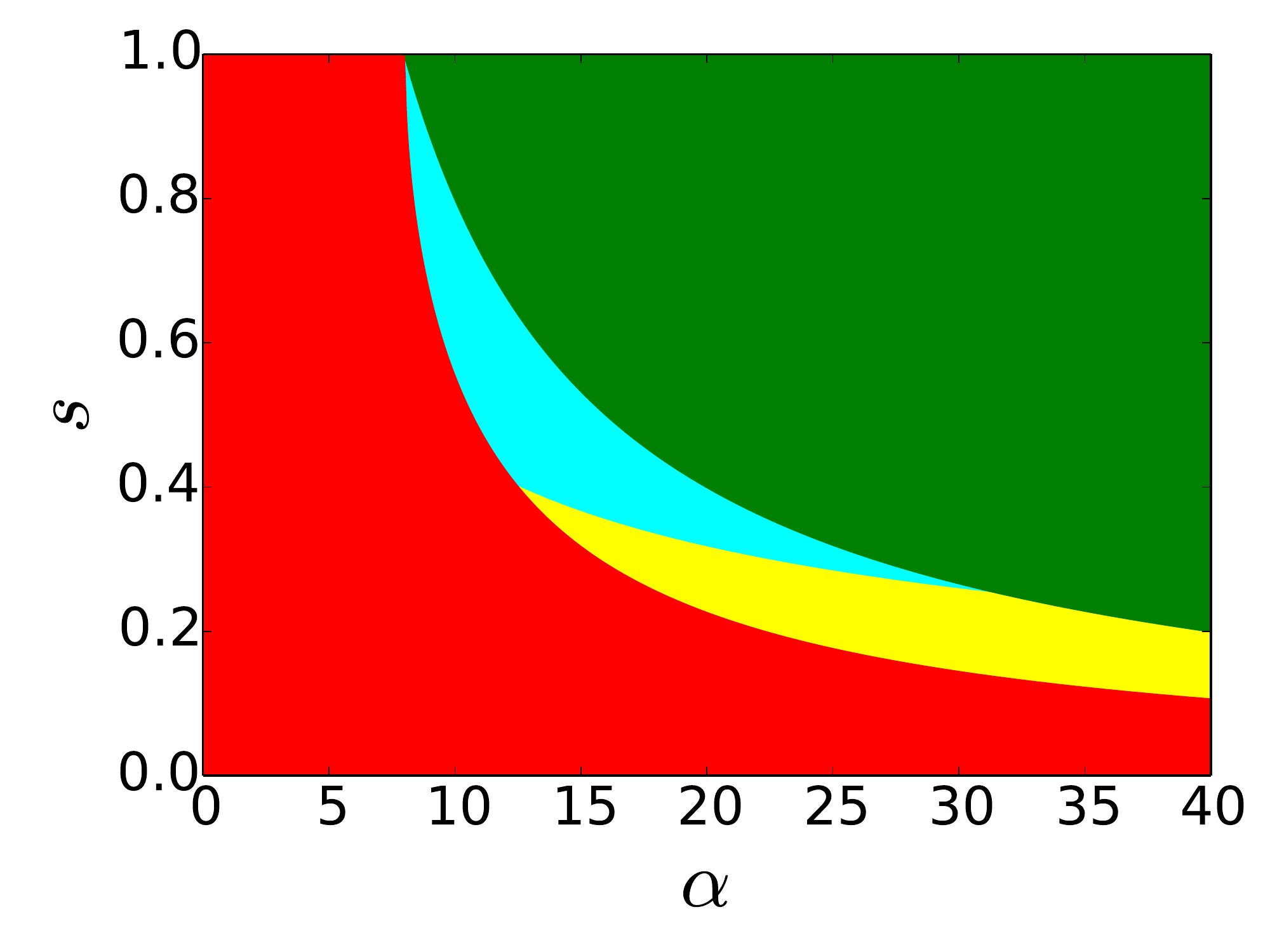}
        \caption{Fixed $\alpha/\beta = 4$.}
    \end{subfigure}
    \quad 
    \begin{subfigure}{0.3\textwidth}
        \centering
        \includegraphics[width=\textwidth]{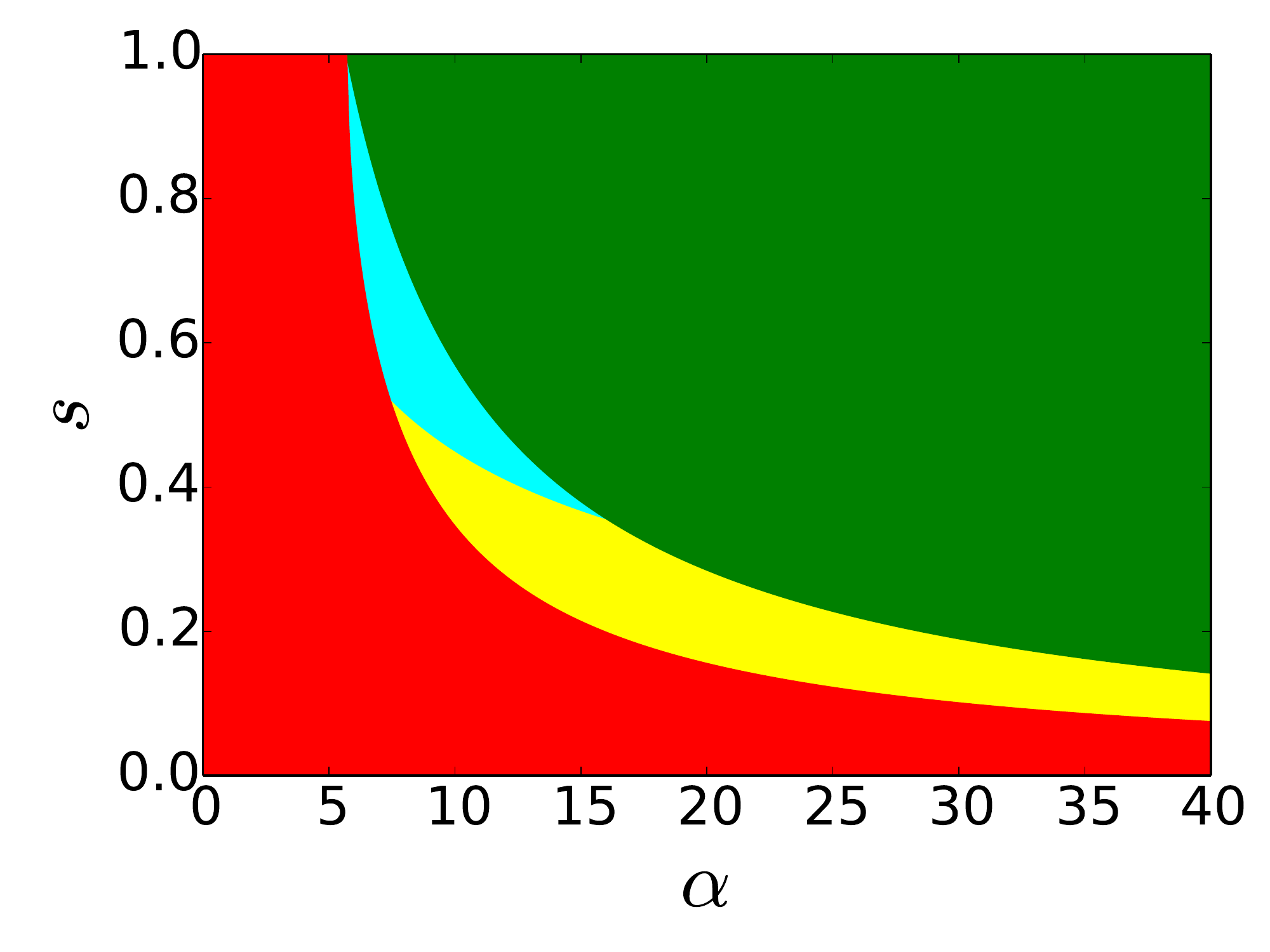}
        \caption{Fixed $\alpha/\beta = 6$.}
    \end{subfigure}
    \caption{Phase diagrams for exact community recovery for fixed $\alpha/\beta$, with $\alpha \in [0,40]$ and $s \in [0,1]$ on the axes. (Colors as in  Fig.~\ref{fig:phase_diagram_fixed_s}.)}
    \label{fig:phase_diagram_fixed_ratio}
\end{figure}

We remark that Theorem~\ref{thm:community_recovery_impossibility} provides a {\it partial converse} to the achievability result in Theorem~\ref{thm:community_recovery}: it is tight when $s^2 (\alpha + \beta)/2 > 1$, but the precise information-theoretic threshold is unknown when $s^2 (\alpha + \beta)/2 < 1$, which is the regime where exact graph matching fails. This leads to an interesting follow-up question: is exact graph matching {\it necessary} for the exact recovery of communities? 
We conjecture that it is not, which is formalized as follows. 

\begin{conjecture}
\label{conjecture:below_graph_matching}
There exists $\epsilon = \epsilon(\alpha, \beta, s) > 0$ such that if \eqref{eq:community_achievability} holds and 
\begin{equation}
s^2 \left( \frac{\alpha + \beta}{2} \right) \ge 1 - \epsilon,
\end{equation}
then there is an estimator $\wh{\boldsymbol{\sigma}} = \wh{\boldsymbol{\sigma}}(G_1, G_2)$ such that 
$
\lim\limits_{n \to \infty} \p ( \overlap ( \wh{\boldsymbol{\sigma}}, \boldsymbol{\sigma}) = 1 ) = 1.
$
\end{conjecture}

In words, we believe that the communities can be exactly recovered even in regimes where exact graph matching is information-theoretically impossible. We outline a possible way to prove this conjecture. The algorithm we shall use is the same one used in the proof of Theorem \ref{thm:community_recovery}: we compute~$\wh{\pi}$, the permutation which maximizes the number of agreeing edges across $G_1$ and $G_2$, and then run an optimal community recovery algorithm on the union graph $\wh{H} = G_1 \lor_{\wh{\pi}} G_2$. Define the {\it correctly-matched region} $\cC : = \{ i \in [n] : \wh{\pi}(i) = \pi_*(i) \}$. When $s^2 (\alpha + \beta)/2 < 1$, 
we have that
$\cC \neq [n]$ 
with high probability. 
However, we expect that $| \cC | = (1 - o(1)) n$; that is, $\wh{\pi}$ coincides with $\pi_*$ on all but a negligible fraction of vertices (which is known as almost exact recovery). This is the case in correlated Erd\H{o}s-R\'{e}nyi graphs \cite{cullina2020partial,wu2021settling}, so we expect it to hold for correlated SBMs as well. Let $\wh{H}_{\cC}$ be the subgraph of $\wh{H}$ restricted to the vertices in $\cC$. Since all vertices in $\cC$ have been correctly matched, we expect that (possibly in an {\it approximate} sense)
\begin{equation}
\label{eq:Hc}
\wh{H}_{\cC} \sim \mathrm{SBM} \left( | \cC|, \alpha (1 - (1 - s)^2) \frac{\log n }{n}, \beta (1 - (1 - s)^2) \frac{\log n}{n} \right).
\end{equation}
In particular, if \eqref{eq:community_achievability} holds, the communities of vertices in $\cC$ can be {\it exactly} recovered. For vertices not in $\cC$, note that most of the neighbors will be elements of $\cC$, which will have correct community labels. If $\alpha > \beta$, then the true community label of a given vertex is the same as the true label of most neighbors with high probability (when $\alpha < \beta$, the reverse is true) \cite{abbe2016exact}, hence the community labels of vertices not in $\cC$ can be correctly identified using a majority vote. 

Making the arguments above formal is a challenging task. For one, though we may expect~\eqref{eq:Hc} to hold if $\cC$ is a {\it fixed} set, it is in fact a {\it random} set depending on $G_1$, $G_2$, and $\pi_{*}$, so formally proving~\eqref{eq:Hc} requires a careful analysis. Moreover, we would like to use \eqref{eq:Hc} to argue that running a community recovery algorithm on $\wh{H}$ (rather than $\wh{H}_{\cC}$) perfectly recovers the communities in $\cC$. Rigorously justifying these points requires significant effort, so we leave it to future work.

\subsubsection{Multiple correlated stochastic block models}
\label{sub:multiple_community_recovery}

We next describe our results on how one can recover communities using $K$ correlated stochastic block models, again using graph matching as a subroutine. 
Considering more than two networks is more and more important in many applications, 
for instance in computational biology, 
where the increasing number of species for which protein-protein interaction networks are available can be leveraged 
for more powerful comparative studies~\cite{singh2008global,kazemi2019mpgm}. 

Formally, we construct $(G_1, \ldots, G_K) \sim \mathrm{CSBM}(n, p, q, s, K)$ as follows. 
First, generate a parent graph $G \sim \SBM(n,p,q)$, and let $\boldsymbol{\sigma}$ denote the community labels. 
Next, given~$G$, we construct $G_{1}$ as well as $G_{2}', \ldots, G_{K}'$ by independently subsampling $G$ with probability $s$. 
Finally, we let $\pi_{*}^{2}, \ldots, \pi_{*}^{K}$ be i.i.d.\ uniformly random permutations of $[n]$, independent of everything else, 
and for $2 \leq k \leq K$, 
we generate $G_{k}$ by relabeling the vertices of $G_{k}'$ according to $\pi_{*}^{k}$.

As in the case of two correlated graphs, the achievability and impossibility results depend on the structure of the union graph with respect to the true permutations $\pi_*^2, \ldots, \pi_*^K$. 
%In particular, we have the following achievability result. 

\begin{theorem}
\label{thm:community_recovery_K} 
Let $(G_1, \ldots, G_K) \sim \mathrm{CSBM} \left( n, \frac{\alpha \log n}{n}, \frac{\beta \log n}{n}, s, K  \right)$. 
Suppose that $s^{2} \left( \alpha + \beta \right) / 2 > 1$ and 
\begin{equation}
\label{eq:K_community_recovery}
| \sqrt{\alpha} - \sqrt{\beta} | > \sqrt{ \frac{2}{1 - (1 - s)^K} }.
\end{equation}
Then there is an estimator $\wh{\boldsymbol{\sigma}} = \wh{\boldsymbol{\sigma}}(G_1, \ldots, G_K)$ such that 
$
\lim\limits_{n \to \infty} \p ( \overlap(\wh{\boldsymbol{\sigma}}, \boldsymbol{\sigma}) = 1 ) = 1.
$
\end{theorem}

Analogously to Theorem~\ref{thm:community_recovery}, 
Theorem~\ref{thm:community_recovery_K} establishes the existence of a region of the parameter space where 
(i) there exists an algorithm that can exactly recover the communities using 
all of $G_{1}, G_{2}, \ldots, G_{K}$, 
but 
(ii) it is information-theoretically impossible to do so using only a strict subset of $G_{1}, G_{2}, \ldots, G_{K}$.

Our next result establishes an impossibility result which is analogous to Theorem \ref{thm:community_recovery_impossibility}. 

\begin{theorem}
\label{thm:K_community_recovery_impossibility}
Let $(G_1, \ldots, G_K) \sim \mathrm{CSBM} \left( n, \alpha \frac{\log n}{n}, \beta \frac{\log n}{n}, s, K \right)$ and suppose that 
\begin{equation}
\label{eq:K_community_impossibility}
| \sqrt{\alpha} - \sqrt{\beta}| < \sqrt{ \frac{2}{1 - (1 - s)^K} }.
\end{equation}
Then for any estimator $\widetilde{\boldsymbol{\sigma}} = \widetilde{\boldsymbol{\sigma}}(G_1, G_2)$, we have that 
$
\lim\limits_{n \to \infty} \p ( \overlap(\widetilde{\boldsymbol{\sigma}}, \boldsymbol{\sigma}) = 1 ) = 0.
$
\end{theorem}

We highlight a few interesting aspects of Theorems~\ref{thm:community_recovery_K} and~\ref{thm:K_community_recovery_impossibility}. As in the two-graph case, Theorem~\ref{thm:K_community_recovery_impossibility} provides a {\it partial} converse to the achievability result in Theorem~\ref{thm:community_recovery_K}: it is tight in the regime $s^2 (\alpha + \beta)/2 > 1$, but the correct threshold remains unknown when $s^2 (\alpha + \beta)/2 < 1$. Additionally, as $K$ increases, the achievability and impossibility conditions in \eqref{eq:K_community_recovery} and \eqref{eq:K_community_impossibility} converge to the conditions $|\sqrt{\alpha} - \sqrt{\beta}| > \sqrt{2}$ and $|\sqrt{\alpha} - \sqrt{\beta}| < \sqrt{2}$, which are the information-theoretic conditions for achievability and impossibility of community recovery in the {\it parent graph}~$G$. In words, the more correlated graphs we observe, the less information is lost when generating the observed graphs from the parent graph via the sampling process.

\subsection{Related work}\label{sec:related_work}

Our work naturally draws upon techniques in the graph matching literature as well as the community recovery literature. Here, we elaborate on relevant work in these fields that were not covered during the exposition of our model and main results. 

{\bf Graph Matching.} Most of the theoretical literature on graph matching has focused on correlated Erd\H{o}s-R\'{e}nyi random graphs, which was introduced by Pedarsani and Grossglauser~\cite{pedarsani2011privacy}. Significant progress has been made in recent years in characterizing the fundamental information-theoretic limits for recovering the latent vertex correspondence $\pi_*$. Cullina and Kiyavash~\cite{cullina2016improved,cullina2018exact} first derived the precise information-theoretic conditions for {\it exact} recovery of $\pi_*$ for sparse graphs (in a sublinear-degree regime), and recently Wu, Xu, and Yu \cite{wu2021settling} refined this to include linear degree regimes. 
Our results, in particular Theorems~\ref{thm:graph_matching} and~\ref{thm:graph_matching_impossibility}, 
are the natural generalizations of these previous works to correlated SBMs, determining the precise information-theoretic threshold for exact recovery in this setting (and improving upon~\cite{onaran2016optimal,cullina2016simultaneous}).

Weaker notions of recovery (e.g., almost exact recovery, partial recovery) have also been addressed for correlated Erd\H{o}s-R\'{e}nyi graphs (see \cite{cullina2020partial,ganassali2020tree,hall2020partial,ganassali2021impossibility,wu2021settling} for more details).  
Recent work by Shirani, Erkip, and Garg~\cite{shirani2021concentration}  provides necessary and sufficient conditions for almost exact recovery in correlated SBMs.  
Our work is also a part of the growing literature studying correlated random graphs beyond the Erd\H{o}s-R\'{e}nyi model~\cite{onaran2016optimal,cullina2016simultaneous,korula2014efficient,RS20,shirani2021concentration,yu2021power}.

A major open question is whether there exist \emph{efficient algorithms} for inferring $\pi_{*}$ in correlated Erd\H{o}s-R\'{e}nyi graphs. In particular, the estimators which are known to succeed up to the information-theoretic threshold are usually given by the solution to a combinatorial optimization problem, for which a brute force search takes $O(n!)$ time. Significant improvements were recently made by~\cite{mossel2019seeded, barak2019}, who provided $n^{O(\log n)}$ time algorithms for exactly recovering $\pi_*$. For values of $s$ close to 1, recent work provides polynomial-time algorithms for exact recovery~\cite{ding2021efficient, fan2020spectral, mao2021random}.

{\bf Community Recovery in Multi-layer SBMs.} 
We briefly review the literature on multi-layer SBMs, as it is the form of side information studied in the literature that is closest to our work. 
Multi-layer SBMs were first introduced by Holland, Laskey, and Leinhardt, in their original work that introduced stochastic block models~\cite{HLL83}. 
In this model, first a community labeling is chosen at random. 
Given the block structure, a collection of SBMs on the same vertex set with the same latent community labels are then generated, one for each layer, possibly with different (but known) edge formation probabilities. Variants of this model have been explored by several authors~\cite{han2015consistent,arroyo2020inference, paul2020null, paul2020spectral, lei2019consistent, ali2019latent, bhattacharya2020consistent}, 
but typically the layers are \emph{conditionally independent} given the community labels.  The works \cite{mayya2019mutual, ma2021community} additionally consider node-level information that is correlated with the latent community membership. 
While our work also considers multiple networks as side information, we emphasize that there are significant differences. 
For one, the networks we consider are not conditionally independent given the latent communities, but are also correlated through the formation of edges. 
Moreover, in the multi-layer setting the node labels are known, which completely removes the need for graph matching.

\subsection{Overview of graph matching proofs}\label{sec:proofs_overview}

%\subsection{Achievability of exact graph matching: Proof sketch of Theorem \ref{thm:graph_matching}}
\textbf{Achievability of exact graph matching: Proof sketch of Theorem~\ref{thm:graph_matching}.} 
Let $\cF_{\epsilon} := \left\{ (1-\epsilon)n/2 \leq |V_{+}|, |V_{-}| \leq (1+\epsilon)n/2 \right\}$ denote the event that the two communities are approximately balanced. 
Since the community labels are i.i.d.\ uniform, we have for any fixed $\epsilon > 0$ that $\p \left( \cF_{\epsilon} \right) = 1 - o(1)$  as $n \to \infty$; 
we may thus condition on $\cF_{\epsilon}$. 
Let $S_{k_1, k_2}$ be the set of permutations which mismatches $k_1$ vertices in $V_{+}$ and $k_2$ vertices in $V_{-}$. 
We show that 
if~\eqref{eq:recovery_condition} holds, 
then there exists $\epsilon = \epsilon(\alpha, \beta, s)$ sufficiently small so that
\begin{equation}
\label{eq:graph_matching_prob_bound}
\p \left( \wh{\pi} \in S_{k_1, k_2} \, \middle| \, \cF_{\epsilon} \right) \le n^{- \epsilon(k_1 + k_2)}.
\end{equation}
To bound the probability that $\wh{\pi} \neq \pi_*$, we then take a union bound over all the events $\{\wh{\pi} \in S_{k_1, k_2} \}$ such that $k_1 + k_2 \ge 1$, that is, there is at least one mismatched vertex, 
concluding the proof. 

The key technical result which enables the proof is \eqref{eq:graph_matching_prob_bound}; this is derived by deriving tight bounds for the generating function corresponding to the number of agreeing edges in $G_1$ and $G_2$ with respect to a given permutation. 
In prior work on the graph matching problem in correlated Erd\H{o}s-R\'{e}nyi graphs, 
as well as for correlated Gaussian matrices, 
the aforementioned generating functions could be exactly computed~\cite{cullina2016improved, cullina2018exact, wu2021settling}. 
An important difference between work on these models and ours is that the stochastic block model is {\it heterogeneous}: the probability of edge formation is not i.i.d.\ over all vertex pairs, but varies depending on the latent community labels of the vertex pairs. As a result, the generating functions of interest cannot be explicitly computed. To handle this heterogeneity, we develop new techniques for bounding these generating functions. Specifically, we derive recursive bounds for the generating functions of interest as a function of the number of vertices; see Section~\ref{sec:gen_fn_analysis} for details. We suspect that this method can be extended to analyze other classes of correlated networks with heterogeneous structure. 

%\subsection{Impossibility of exact graph matching: Proof sketch of Theorem \ref{thm:graph_matching_impossibility}}

\textbf{Impossibility of exact graph matching: Proof sketch of Theorem~\ref{thm:graph_matching_impossibility}.} 
Let $H$ be the intersection graph between $G_1$ and $G_2'$, that is, $(i,j)$ is an edge in $H$ if and only if $(i,j)$ is an edge in $G_1$ and $G_2'$. Equivalently, $(i,j)$ must be an edge in the parent graph $G$ and must be included in both $G_1$ and $G_2'$. 
Since the probability of the latter event is $s^2$, we see that 
$H \sim \mathrm{SBM} \left( n, \alpha s^2 \log(n)/n, \beta s^2 \log(n)/n \right)$.
If $s^2 (\alpha + \beta)/2 < 1$, then $H$ is not connected with probability tending to 1 as $n \to \infty$. In particular, $H$ has many singletons in this regime, 
which are vertices that have {\it non-overlapping neighborhoods} in $G_1$ and $G_2'$. 
Due to the lack of shared information, it is difficult to match such vertices across the two graphs, even for optimal estimators that have access to the ground-truth community labeling $\boldsymbol{\sigma}$. 
In particular, one can show that 
the maximum a posteriori (MAP) estimator of $\pi_*$ given $G_1$, $G_2$, \emph{and} $\boldsymbol{\sigma}$ 
cannot output $\pi_*$ with probability bounded away from zero, 
so neither can any other estimator. 

%To be precise, we study the performance of the maximum a posteriori (MAP) estimator of $\pi_*$ given $G_1$, $G_2$, \emph{and} $\boldsymbol{\sigma}$, which minimizes the probability of error. Since the MAP estimator cannot output $\pi_*$ with probability bounded away from zero, neither can {\it any} other estimator.

%We remark that the strategy of studying the connectivity threshold of the intersection graph was previously employed by Cullina and Kiyavash for correlated Erd\H{o}s-R\'{e}nyi graphs \cite{cullina2016improved, cullina2018exact}. We expect that this strategy can be used to prove converse arguments in other models of correlated random graphs (e.g., the Chung-Lu model) as well. We leave this for future work. 

\subsection{Discussion and future work}

In this work, we studied the problem of exact community recovery given multiple correlated SBMs as side information. Specifically, our goal was to understand how this side information changes the fundamental information-theoretic threshold for achievability and impossibility of exact community recovery. Strikingly, using multiple correlated SBMs allows one to exactly recover communities in regimes where it is information-theoretically impossible to do so using a single graph. 

Precisely, we determine the sharp information-theoretic condition for exact graph matching in a pair of correlated SBMs. We then apply this to determine conditions for achievability and impossibility of exact community recovery. In the regime where exact graph matching is achievable, we identify the precise information-theoretic conditions for achievability and impossibility of exact community recovery. We also discuss extensions with $K \ge 2$ correlated SBMs. % are given as side information. 

Our work leaves open several important avenues for future work, 
which we outline below.

\begin{itemize}
\item 
{\bf Closing the information-theoretic gaps in exact community recovery.} Together, Theorems~\ref{thm:community_recovery} and~\ref{thm:community_recovery_impossibility} show that in the regime $s^2 (\alpha + \beta) / 2 > 1$, we have identified the  information-theoretic threshold between impossibility and achievability for exact community recovery. However, we do not have achievability results for the regime $s^2 (\alpha + \beta)/2 < 1$, since exact graph matching is not possible in this case. This leads to the following natural question which is formalized in Conjecture \ref{conjecture:below_graph_matching}: {\it is exact graph matching needed for exact community recovery?} We believe the answer is {\it no}; we expect that showing this rigorously will lead to new algorithms for jointly synthesizing networks and identifying communities. 

%could mention the edge-case of $s^{2}(\alpha+\beta)/2 = 1+o(1)$, but this is not so interesting (and probably follows from our arguments if done slightly more carefully)

\item 
 {\bf Efficient algorithms.} Our achievability algorithms rely on graph matching as a subroutine, which is computationally expensive. %expensive -- the implementation we consider takes $O(n!)$ time. 
Do there exist {\it efficient} algorithms for graph matching in the correlated SBM model? If not, is it possible to recover communities exactly using a {\it polynomial-time relaxation} of the graph matching subroutine?

\item 
{\bf General correlated stochastic block models.} For simplicity of exposition, we focused on the simplest setting of the stochastic block model where there are two balanced communities. A natural future direction is to extend our results to account for more general SBMs with multiple communities (which are understood well in the single graph setting~\cite{Abbe_survey}). % of various sizes.

\item 
{\bf Beyond exact community recovery.} Besides exact recovery, natural notions of community recovery include almost exact recovery, where the goal is to recover all but a negligible fraction of community labels, and partial recovery, where the goal is 
%to recover more than half of the community labels. 
to do better than a random labeling. 
Using correlated networks as side information to accomplish these tasks is a natural and exciting direction.
A key challenge is that in the regimes where phase transitions occur for almost exact and partial recovery (see~\cite{Abbe_survey}), exact graph matching is information-theoretically impossible by Theorem~\ref{thm:graph_matching_impossibility}, hence this cannot be used as a black box. Solving this problem will lead to new methods for community detection based on data from multiple networks. 
\end{itemize}

\subsection{Notation}
\label{sec:notation}

Recall that the underlying vertex set is $V = [n] := \left\{ 1, 2, \ldots, n \right\}$. 
We denote by $\cS_{n}$ the set of permutations of $[n]$. 
Recall that $V_{+} := \left\{ i \in \left[ n \right] : \sigma_{i} = +1 \right\}$ and $V_{-} := \left\{ i \in \left[ n \right] : \sigma_{i} = -1 \right\}$ denote the vertices in the two communities. 

Let $\cE : = \{ \{i,j\} : i,j \in [n], i \neq j \}$ denote the set of all unordered vertex pairs. We will use $(i,j)$, $(j,i)$, and $\{i,j\}$ interchangeably to denote the unordered pair consisting of $i$ and $j$. 
Given 
% a community labeling 
$\boldsymbol{\sigma}$, we also define the sets 
$\cE^+(\boldsymbol{\sigma}) : = \left \{ (i,j) \in \cE : \sigma_i \sigma_j = +1 \right \}$ 
and 
$\cE^-(\boldsymbol{\sigma} ) : = \left \{ (i,j) \in \cE : \sigma_i \sigma_j = -1 \right \}$. 
% \begin{align*}
% \cE^+(\boldsymbol{\sigma}) & : = \left \{ (i,j) \in \cE : \sigma_i \sigma_j = +1 \right \} \\
% \cE^-(\boldsymbol{\sigma} ) & : = \left \{ (i,j) \in \cE : \sigma_i \sigma_j = -1 \right \}.
% \end{align*}
In words, $\cE^+(\boldsymbol{\sigma})$ is the set of {\it intra-community} vertex pairs, and $\cE^-(\boldsymbol{\sigma})$ is the set of {\it inter-community} vertex pairs. Note in particular that $\cE^+(\boldsymbol{\sigma})$ and $\cE^-(\boldsymbol{\sigma})$ partition $\cE$. 

We next introduce some notation pertaining to the construction of the correlated SBMs. Let $A$ be the adjacency matrix of $G_1$, let $B$ be the adjacency matrix of $G_2$, and let $B'$ be the adjacency matrix of $G_{2}'$. 
Note that, by construction, we have that $B'_{i,j} = B_{\pi_{*}(i),\pi_{*}(j)}$ for every $i,j$.
By the construction of the correlated SBMs, we have the following probabilities for every $(i,j) \in \cE$: 
\begin{align*}
\p \left( \left( A_{i,j}, B'_{i,j} \right) = (1,1) \, \middle| \, \boldsymbol{\sigma} \right) 
&= \begin{cases}
s^2 p &\text{if } \sigma_i = \sigma_{j}, \\
s^2 q &\text{if } \sigma_i \neq \sigma_{j};
\end{cases}\\
\p \left( \left( A_{i,j}, B'_{i,j} \right) = (1,0) \, \middle| \, \boldsymbol{\sigma} \right) 
&= \begin{cases}
s(1-s) p &\text{if }  \sigma_i = \sigma_{j}, \\
s(1-s) q &\text{if }  \sigma_i \neq \sigma_{j};
\end{cases}\\
\p \left( \left( A_{i,j}, B'_{i,j} \right) = (0,1) \, \middle| \, \boldsymbol{\sigma} \right) 
&= \begin{cases}
s(1-s) p &\text{if }  \sigma_i = \sigma_{j},  \\
s(1-s) q &\text{if }  \sigma_i \neq \sigma_{j};
\end{cases}\\
\p \left( \left( A_{i,j}, B'_{i,j} \right) = (0,0) \, \middle| \, \boldsymbol{\sigma} \right) 
&= \begin{cases}
1 - p(2s - s^2) &\text{if }  \sigma_i = \sigma_{j},  \\
1 - q(2s - s^2) &\text{if }  \sigma_i \neq \sigma_{j}.
\end{cases}
\end{align*}
For brevity, for $i,j \in \{0,1\}$ we write 
\[
p_{ij} : = \p \left( \left( A_{1,2}, B'_{1,2} \right) = (i,j) \, \middle| \, \boldsymbol{\sigma} \right) \qquad \text{if $\sigma_{1} = \sigma_{2}$}
\]
and 
\[
q_{ij} : = \p \left( \left( A_{1,2}, B'_{1,2} \right) = (i,j) \, \middle| \, \boldsymbol{\sigma} \right) \qquad \text{if $\sigma_{1} \neq \sigma_{2}$}.
\] 
For an event $\cA$, we denote by $\mathbf{1} \left( \cA \right)$ the indicator of $\cA$, which is $1$ if $\cA$ occurs and $0$ otherwise. 

\subsection{Outline}

The rest of this paper is organized as follows. 
In Section~\ref{sec:graph_matching_proof} we prove our main result, Theorem~\ref{thm:graph_matching}. 
Section~\ref{sec:impossibility_graph_matching} contains a proof of Theorem~\ref{thm:graph_matching_impossibility}. 
In Section~\ref{sec:impossibility_community_recovery} we prove Theorem~\ref{thm:community_recovery_impossibility}, 
and finally Section~\ref{sec:proofs_many_corr_SBMs} contains the proofs of Theorems~\ref{thm:community_recovery_K} and~\ref{thm:K_community_recovery_impossibility}. 

%\newpage 

\section{Exact graph matching for correlated SBMs: achievability}%{Proof of Theorem \ref{thm:graph_matching}}
\label{sec:graph_matching_proof}

In this section we prove Theorem~\ref{thm:graph_matching}. 
Recall that our objective is to find the ground truth permutation $\pi_{*}$. 
To this end, we study an estimator $\wh{\pi}$ which maximizes the number of agreeing edges in the two graphs, 
that is, the number of pairs of vertices connected in both. In other words, letting $A$ denote the adjacency matrix of $G_{1}$ and $B$ denote the adjacency matrix of $G_{2}$, the estimator is given by 
\begin{equation}
\label{eq:matching_estimator_v1}
\widehat{\pi}(G_1, G_2) \in \argmax_{\pi \in \cS_{n}} \sum\limits_{(i,j) \in \cE} A_{i,j} B_{\pi(i), \pi(j)}.
\end{equation}
When this estimator is not uniquely defined, that is, when the argmax set above is not a singleton, $\wh{\pi}(G_1, G_2)$ is chosen to be an arbitrary element of the argmax set. 

\begin{definition}[Lifted permutation]
For a permutation $\pi \in \cS_{n}$ on the vertices, define the corresponding \emph{lifted permutation} $\tau : \cE \to \cE$ on vertex pairs as $\tau((i,j)) := (\pi(i), \pi(j))$. As a shorthand, we write $\tau = \ell(\pi)$, and thus also $\tau_* : = \ell(\pi_*)$ and $\wh{\tau} := \ell(\wh{\pi})$. 
\end{definition}
Note that if a permutation $\pi$ maps two vertices to each other, then the lifted permutation $\tau = \ell(\pi)$ maps this (unordered) pair of vertices to itself; 
that is, if $\pi(1) = 2$ and $\pi(2) = 1$, then $\tau((1,2)) = (2,1) = (1,2)$. 
Observe that there is a one-to-one mapping between permutations on vertices (i.e., $\cS_{n}$) and lifted permutations. For this reason, finding the ground truth permutation~$\pi_{*}$ is equivalent to finding the ground truth lifted permutation $\tau_{*}$. 
Similarly, conditioning on $\pi_{*}$ is equivalent to conditioning on $\tau_{*}$. 

Using this notation, we can rewrite \eqref{eq:matching_estimator_v1} as 
\begin{equation}
\label{eq:matching_estimator_v2}
\wh{\pi}(G_1, G_2) \in 
\argmax_{\pi \in \cS_{n}} \sum\limits_{e \in \cE} A_{e} B_{\tau(e)},
\end{equation}
where $\tau = \ell(\pi)$, and $A_e = A_{i,j}$ if $e = (i,j)$. 
For a lifted permutation~$\tau$ define 
\[
X(\tau) := \sum\limits_{e \in \cE} A_{e} B_{\tau_{*}(e)} - \sum\limits_{e \in \cE} A_{e} B_{\tau(e)} = \sum\limits_{e \in \cE : \tau(e) \neq \tau_*(e)} \left(A_{e} B_{\tau_{*}(e)} - A_{e} B_{\tau(e)} \right).
\]
Observe that $X(\tau_{*}) = 0$ and that $\wh{\pi}(G_{1}, G_{2}) \in \argmin_{\pi \in \cS_{n}} X(\ell(\pi))$. 
Therefore this estimator is correct---that is, $\wh{\pi}(G_1, G_2) = \pi_*$---if for every lifted permutation $\tau \neq \tau_*$ we have that 
$X(\tau) > 0$. 
Conditioning on $\pi_{*}$ we thus have that 
\[
\p (\wh{\pi} \neq \pi_* ) 
\le \p (\exists\, \pi \neq \pi_* : X(\ell(\pi)) \le 0 ) 
= \E \left[ \p \left( \exists\, \pi \neq \pi_* : X(\ell(\pi)) \le 0 \, \middle| \, \pi_{*} \right) \right],
\]
so a union bound implies that
\[
\p (\wh{\pi} \neq \pi_* ) 
\le \E \left[ \sum\limits_{\pi \in \cS_{n}: \pi \neq \pi_*} \p \left(X(\ell(\pi)) \le 0 \, \middle| \, \pi_{*} \right) \right].
\]
To proceed, we shall bound the terms in the summation on the right hand side by studying the probability generating function (PGF) of $X(\tau)$ for any fixed lifted permutation $\tau$. 
More specifically, we will study the PGF of $X(\tau)$ given both $\pi_{*}$ (equivalently, $\tau_{*}$) and the community labeling $\boldsymbol{\sigma}$.

\subsection{Probabilistic bounds for \texorpdfstring{$X(\tau)$}{X(tau)}}

In this section, we establish large-deviations-type probability bounds for the event that $\wh{\tau} = \tau$, where $\tau$ is a fixed lifted permutation. In our analysis we derive probability bounds which hold pointwise given {\it any} community labeling $\boldsymbol{\sigma}$ and ground truth lifted permutation $\tau_{*}$. We then derive simpler expressions for the bounds that hold when the two communities are approximately balanced.  

To make these ideas more formal, we begin by defining some notation. 
Recall that 
given $\boldsymbol{\sigma}$, 
the set $\cE^+(\boldsymbol{\sigma})$ is the set of intra-community vertex pairs, while $\cE^-(\boldsymbol{\sigma})$ is the set of inter-community vertex pairs. 
Given $\boldsymbol{\sigma}$ and $\tau_{*}$, 
for a fixed lifted permutation $\tau$ we also define the quantities 
\begin{align*}
 M^{+} ( \tau ) & : = \left | \left \{ e \in \cE^+(\boldsymbol{\sigma}) : \tau(e) \neq \tau_*(e) \right \} \right|,  \\
M^-(\tau) & : = \left| \left \{ e \in \cE^-(\boldsymbol{\sigma}) : \tau(e) \neq \tau_*(e) \right \} \right|, \\
Y^+(\tau) & : = \sum\limits_{e \in \cE^+(\boldsymbol{\sigma}): \tau(e) \neq \tau_*(e)} A_{e} B_{\tau_{*}(e)}, \\
Y^-(\tau) & : = \sum\limits_{e \in \cE^-(\boldsymbol{\sigma}): \tau(e) \neq \tau_*(e)} A_{e} B_{\tau_{*}(e)}.
\end{align*}
In words, $M^+(\tau)$ is the number of mismatched intra-community vertex pairs. 
Furthermore, $Y^+(\tau)$ is the number of mismatched intra-community vertex pairs which contribute to the alignment score of the ground truth lifted permutation $\tau_* = \ell(\pi_*)$. We have analogous interpretations for the inter-community quantities $M^-(\tau)$ and $Y^-(\tau)$. 
Note that in addition to $\tau$, these quantities depend on $\boldsymbol{\sigma}$ and $\tau_{*}$ as well; however, we suppress this in the notation for simplicity. 
Observe also that $M^{+}(\tau)$ and $M^{-}(\tau)$ are deterministic functions of $\boldsymbol{\sigma}$, $\tau_{*}$, and $\tau$. 
On the other hand, given $\boldsymbol{\sigma}$ and $\tau_{*}$, and fixing $\tau$, 
the quantities $Y^{+}(\tau)$ and $Y^{-}(\tau)$ are random variables, since they depend on the two graphs $G_{1}$ and $G_{2}$ as well. 

Given a community labeling~$\boldsymbol{\sigma}$ and the ground truth lifted permutation $\tau_{*}$, for a fixed lifted permutation $\tau$ we shall study the PGF 
\[
\Phi^{\tau} ( \theta, \omega, \zeta) : = \E \left [ \theta^{X(\tau)} \omega^{Y^+(\tau)} \zeta^{Y^-(\tau)} \, \middle| \, \boldsymbol{\sigma} , \tau_* \right ].
\]

\begin{remark} 
Since our goal is to bound the probability of the event $\{ X(\tau) \le 0 \}$, 
it is perhaps more natural to study simply the PGF of $X(\tau)$, 
rather than the joint PGF of $X(\tau)$, $Y^+(\tau)$, and $Y^-(\tau)$. 
However, the success of the former approach requires that $s^2 (\alpha + \beta)/2 > 2$, which is suboptimal. For a tighter analysis, one must condition on the typical behavior of $Y^+(\tau)$ and $Y^-(\tau)$, which in turn requires us to consider the joint PGF. This idea was previously used to show that the information-theoretic threshold can be achieved in the graph matching problem for Erd\H{o}s-R\'{e}nyi graphs~\cite{cullina2018exact,wu2021settling}.
\end{remark}

The next lemma provides a useful bound for $\Phi^{\tau}$ for any $\boldsymbol{\sigma}$ and $\tau_{*}$; we defer its proof to Section~\ref{lemma:pgf_bound}.  

\begin{lemma}
\label{lemma:pgf_bound} 
Given a community labeling $\boldsymbol{\sigma}$ and the ground truth lifted permutation $\tau_{*}$, the following holds. 
Fix $\pi \in \cS_{n}$ and let $\tau = \ell(\pi)$. For any constants $\epsilon \in (0,1)$ and $1 \le \omega, \zeta \le 3$, it holds for all $n$ large enough that 
\begin{equation}
\label{eq:pgf_bound}
\Phi^{\tau} \left(1/\sqrt{n}, \omega, \zeta \right) \le \mathrm{exp} \left( - (1 - \epsilon) s^2 \left(  \alpha M^+(\tau) + \beta M^-(\tau)  \right) \frac{\log n}{n} \right).
\end{equation}
\end{lemma}

We remark that \eqref{eq:pgf_bound} bounds the probability generating function for the specific value $\theta = 1/\sqrt{n}$. This choice is somewhat arbitrary; the proof of Lemma \ref{lemma:pgf_bound} shows that the bound holds for all $\theta$ smaller than some positive function of $\epsilon, \alpha, \beta$, and $s$, and larger than $\log(n) / n$. Similarly, the requirement that $\omega, \zeta \le 3$ is arbitrary; we expect that, with a careful analysis, one could even let $\omega$ and $\zeta$ be slowly increasing functions of $n$.

To apply Lemma \ref{lemma:pgf_bound} later on, we need to compute/estimate $M^+(\tau)$ and $M^-(\tau)$. 
To this end, given $\boldsymbol{\sigma}$ and $\pi_{*}$, for non-negative integers $k_1$ and $k_2$, 
let $S_{k_{1}, k_{2}}$ denote the set of lifted permutations $\ell(\pi)$ 
where $\pi$ incorrectly matches $k_{1}$ vertices in $V_{+}$ and 
incorrectly matches $k_{2}$ vertices in $V_{-}$. 
That is, define 
\[
S_{k_1, k_2} : = \left\{ \ell(\pi) : \left| \left\{ i \in V_{+} : \pi(i) \neq \pi_{*} (i) \right\} \right| = k_{1} \text{ and } \left| \left\{ i \in V_{-} : \pi(i) \neq \pi_{*} (i) \right\} \right| = k_{2} \right\}.
\]
Note that $S_{k_{1}, k_{2}}$ is defined given 
$\boldsymbol{\sigma}$ and $\pi_{*}$; 
%a community labeling $\boldsymbol{\sigma}$ and the ground truth permutation $\pi_{*}$; 
however, for simplicity we omit these from the notation. 
The next lemma employs simple counting arguments to compute $M^+(\tau)$ and $M^-(\tau)$ for $\tau \in S_{k_1, k_2}$. 
In essence, it shows how to go from mismatches in the vertex permutation $\pi$ to mismatches in the lifted permutation $\tau = \ell(\pi)$. 
We note that a variant of this result in a related but slightly different setting was stated (without proof) in~\cite{onaran2016optimal}; 
we present the details for completeness. 

\begin{lemma}
\label{lemma:M} 
Fix $\pi \in \cS_{n}$ and let $\tau = \ell(\pi)$. 
Given $\boldsymbol{\sigma}$ and $\pi_{*}$, 
let $k_{1}$ and $k_{2}$ be such that $\tau \in S_{k_1, k_2}$. 
Then we have that  
\begin{align}
\label{eq:M^+}
M^+(\tau) & = \binom{k_1}{2} + k_1 ( |V_{+}| - k_1) + \binom{k_2}{2} + k_2 (|V_{-}| - k_2) - \left|E_{tr}^{+}\right|; \\
\label{eq:M^-}
M^-(\tau) & = k_1 |V_{-}| + k_2 |V_{+}| - k_1 k_2 - \left|E_{tr}^{-}\right|,
\end{align}
where 
\begin{align*}
E_{tr}^{+} &:= 
\left\{ (u,v) \in \cE^{+}(\boldsymbol{\sigma}) : \pi(u) = \pi_{*} (v), \pi(v) = \pi_{*}(u) \right\}, \\
E_{tr}^{-} &:= 
\left\{ (u,v) \in \cE^{-}(\boldsymbol{\sigma}) : \pi(u) = \pi_{*} (v), \pi(v) = \pi_{*}(u) \right\}.
\end{align*}
That is, $E_{tr}^{+}$ is the set of vertex pairs from the same community which are transposed under~$\pi$ compared to $\pi_{*}$,  
and an analogous description holds for $E_{tr}^{-}$. 
Moreover, we have the bounds 
$\left|E_{tr}^{+}\right|, \left|E_{tr}^{-}\right| \le (k_1 + k_2)/2$.
\end{lemma}

\begin{proof}
Let $e = (i,j)$. 
Observe first that 
if $\pi(i) = \pi_{*}(i)$ and $\pi(j) = \pi_{*}(j)$, 
then also $\tau(e) = \tau_{*}(e)$, 
and hence this pair does not contribute to $M^{+}(\tau)$ or $M^{-}(\tau)$. 
Thus in order for $e = (i,j)$ to contribute to $M^{+}(\tau)$ or $M^{-}(\tau)$, 
we must have either $\pi(i) \neq \pi_{*}(i)$ or $\pi(j) \neq \pi_{*}(j)$.

We start by deriving \eqref{eq:M^+}. 
Let us first consider the contribution to $M^{+}(\tau)$ from pairs of vertices in $V_{+}$. 
The number of pairs of vertices $i,j \in V_{+}$ 
such that $\pi(i) \neq \pi_{*}(i)$ and $\pi(j) \neq \pi_{*}(j)$ 
is $\binom{k_1}{2}$, 
while the number of pairs of vertices $i,j \in V_{+}$ 
such that one is correctly matched by $\pi$ and the other is incorrectly matched 
is $k_1 (|V_{+}| - k_1)$. 
These give the first two terms in~\eqref{eq:M^+}. 
However, not all of these pairs of vertices have $\tau(e) \neq \tau_{*}(e)$. 
Specifically, if $i,j \in V_{+}$ are such that 
$\pi(i) = \pi_{*}(j)$ and $\pi(j) = \pi_{*}(i)$, 
then both $i$ and $j$ are mismatched (and hence counted above), 
yet $\tau(e) = \tau_{*}(e)$ (and hence should not be counted). 
This leads to the subtraction in~\eqref{eq:M^+}. 
The contribution to $M^{+}(\tau)$ from pairs in $V_{-}$ is analogous.

We now turn to deriving \eqref{eq:M^-}. 
The number of pairs where $i \in V_{+}$ and $j \in V_{-}$ 
such that $\pi(i) \neq \pi_{*}(i)$ 
is $k_{1} \left| V_{-} \right|$. 
Similarly, the number of pairs where $i \in V_{+}$ and $j \in V_{-}$ 
such that $\pi(j) \neq \pi_{*}(j)$ 
is $k_{2} \left| V_{+} \right|$. 
Here we have double-counted pairs $i \in V_{+}$ and $j \in V_{-}$ 
such that $\pi(i) \neq \pi_{*}(i)$ \emph{and} $\pi(j) \neq \pi_{*}(j)$; 
there are $k_{1} k_{2}$ such pairs. 
Thus the number of pairs $i \in V_{+}$ and $j \in V_{-}$ 
such that $\pi(i) \neq \pi_{*}(i)$ or $\pi(j) \neq \pi_{*}(j)$ 
is $k_1 |V_{-}| + k_2 |V_{+}| - k_1 k_2$. 
However, not all of these pairs of vertices have $\tau(e) \neq \tau_{*}(e)$. 
Specifically, if $i \in V_{+}$ and $j \in V_{-}$ are such that 
$\pi(i) = \pi_{*}(j)$ and $\pi(j) = \pi_{*}(i)$, 
then both $i$ and $j$ are mismatched (and hence counted above), 
yet $\tau(e) = \tau_{*}(e)$ (and hence should not be counted). 
This leads to the subtracted term in~\eqref{eq:M^-}. 

Finally, the total number of transpositions (of $\pi$ compared to $\pi_{*}$) satisfies $2 \left( \left|E_{tr}^{+} \right| + \left|E_{tr}^{-} \right| \right) \le k_1 + k_2$, since each transposition involves two mismatched vertices and $k_1 + k_2$ is the total number of mismatched vertices. This leads to the bounds $\left|E_{tr}^{+}\right|, \left|E_{tr}^{-}\right| \le (k_1 + k_2 )/2$ as desired.
\end{proof}

The combinatorial formulas \eqref{eq:M^+} and \eqref{eq:M^-} are somewhat unwieldy to use directly. Fortunately, we can derive relatively simple linear lower bounds when the two communities are approximately balanced. To formalize this idea, we first introduce the following ``nice" event. 

\begin{definition}[Balanced communities]
For $\epsilon > 0$ define the event 
\[
\cF_\epsilon := 
\left\{
\left(1 - \frac{\epsilon}{2} \right) \frac{n}{2} 
\le |V_{+}|, |V_{-}| 
\le \left( 1 + \frac{\epsilon}{2} \right) \frac{n}{2}
\right\}.
\]
\end{definition}

Note that whether or not $\cF_{\epsilon}$ holds depends only on the community labels $\boldsymbol{\sigma}$. Also, since the community labels are i.i.d.\ uniform, we have for any fixed $\epsilon > 0$ that $\p \left( \cF_{\epsilon} \right) = 1 - o(1)$  as $n \to \infty$.

Now fix $\epsilon > 0$ and a lifted permutation $\tau$. 
Our next goal is to find simple lower bounds for $M^+(\tau)$ and $M^-(\tau)$, given community labels $\boldsymbol{\sigma}$ such that $\cF_{\epsilon}$ holds, and given $\tau_{*}$. 
To this end, let $k_{1}$ and $k_{2}$ be such that $\tau \in S_{k_1, k_2}$. 
We distinguish two cases: 
\begin{itemize}
    \item \emph{Case 1:} both $k_{1}$ and $k_{2}$ are small; specifically, $k_{1} \leq \frac{\epsilon}{2} |V_{+}|$ and $k_2 \le \frac{\epsilon}{2} |V_{-}|$.
    \item \emph{Case 2:} either $k_{1}$ or $k_{2}$ is large; 
    specifically, either $k_{1} \geq \frac{\epsilon}{2} |V_{+}|$ or $k_2 \geq \frac{\epsilon}{2} |V_{-}|$.
\end{itemize}
We start with the first case, when $k_{1} \leq \frac{\epsilon}{2} |V_{+}|$ and $k_2 \le \frac{\epsilon}{2} |V_{-}|$. 
\begin{lemma}\label{lem:Mtau_lowerbounds_ksmall}
Fix $\epsilon > 0$. Given community labels $\boldsymbol{\sigma}$ such that $\cF_{\epsilon}$ holds, let $k_{1}$ and $k_{2}$ be such that $k_{1} \leq \frac{\epsilon}{2} |V_{+}|$ and $k_2 \le \frac{\epsilon}{2} |V_{-}|$. 
Given $\boldsymbol{\sigma}$ and $\pi_{*}$, 
let $\tau$ be a lifted permutation such that $\tau \in S_{k_1, k_2}$. 
For all $n$ large enough we have the following bounds: 
\begin{align}
    M^+(\tau) &\geq (1 - \epsilon)\frac{n}{2} (k_1 + k_2), \label{eq:M+_regime1} \\
    M^-(\tau) &\geq (1 - \epsilon)\frac{n}{2} (k_1 + k_2). \label{eq:M-_regime1}
\end{align}
\end{lemma}
\begin{proof}
For $n$ sufficiently large, we have the following lower bound for $M^{+}(\tau)$: 
\begin{align*}
M^+(\tau) & \stackrel{(a)}{\ge} k_1 (|V_{+}| - k_1) + k_2( |V_{-}| - k_2) - \frac{k_1 + k_2}{2} %\nonumber \\
\stackrel{(b)}{\ge} \left( 1 - \frac{\epsilon}{2} \right) \left( k_1 |V_{+}| + k_2 |V_{-}| \right) - \frac{k_1 + k_2}{2} \nonumber \\
& \stackrel{(c)}{\ge} \left( \left (1 - \frac{\epsilon}{2} \right)^2 \frac{n}{2} - 1 \right) (k_1 + k_2) %\nonumber \\
%\label{eq:M+_regime1}
\stackrel{(d)}{\ge} (1 - \epsilon)\frac{n}{2} (k_1 + k_2),
\end{align*}
%\begin{align}
%M^+(\tau) & \stackrel{(a)}{\ge} k_1 (|V_1| - k_1) + k_2( |V_2| - k_2) - \frac{k_1 + k_2}{2} \nonumber \\
%& \stackrel{(b)}{\ge} \left( 1 - \frac{\epsilon}{2} \right) \left( k_1 |V_1| + k_2 |V_2| \right) - \frac{k_1 + k_2}{2} \nonumber \\
%& \stackrel{(c)}{\ge} \left( \left (1 - \frac{\epsilon}{2} \right)^2 \frac{n}{2} - 1 \right) (k_1 + k_2) \nonumber \\
%\label{eq:M+_regime1}
%& \stackrel{(d)}{\ge} (1 - \epsilon)\frac{n}{2} (k_1 + k_2),
%\end{align}
where $(a)$ follows from ignoring positive terms in the formula \eqref{eq:M^+} and bounding $\left|E_{tr}^{+}\right|$ by $(k_1 + k_2)/2$, 
$(b)$ uses the upper bounds 
$k_{1} \leq \frac{\epsilon}{2} |V_{+}|$ 
and $k_2 \le \frac{\epsilon}{2} |V_{-}|$, 
$(c)$ uses the lower bounds 
$|V_{+}|, |V_{-}| \geq (1-\epsilon/2)n/2$, 
which hold on the event $\cF_\epsilon$, 
and finally $(d)$ uses $(1 - \epsilon/2)^2 > 1 - \epsilon$ 
and the fact that $n$ is sufficiently large. 
Turning to $M^-(\tau)$, we have the following lower bound: 
\begin{align*}
%\label{eq:M-_regime1}
M^-(\tau) & \stackrel{(e)}{\ge} k_1 |V_{-}| + k_2 |V_{+}| - k_1 k_2 - \frac{k_1 + k_2}{2} 
 = k_1 \left( |V_{-} | - \frac{k_2}{2} - \frac{1}{2} \right) + k_2 \left( |V_{+}| - \frac{k_1}{2} - \frac{1}{2} \right) \nonumber  \\
& \stackrel{(f)}{\ge} \left( 1 - \frac{\epsilon}{2} \right) ( k_1 |V_{-} | + k_2 |V_{+}| ) 
 \stackrel{(g)}{\ge} \left( 1 - \frac{\epsilon}{2} \right)^2 \frac{n}{2} (k_1 + k_2) 
 \ge (1 - \epsilon) \frac{n}{2} (k_1 + k_2),
\end{align*}
%\begin{align}
%\label{eq:M-_regime1}
%M^-(\tau) & \stackrel{(e)}{\ge} k_1 |V_2| + k_2 |V_1| - k_1 k_2 - \frac{k_1 + k_2}{2} \nonumber \\
%& = k_1 \left( |V_2 | - \frac{k_2}{2} - \frac{1}{2} \right) + k_2 \left( |V_1| - \frac{k_1}{2} - \frac{1}{2} \right) \nonumber  \\
%& \stackrel{(f)}{\ge} \left( 1 - \frac{\epsilon}{2} \right) ( k_1 |V_2 | + k_2 |V_1| ) \nonumber \\
%& \stackrel{(g)}{\ge} \left( 1 - \frac{\epsilon}{2} \right)^2 \frac{n}{2} (k_1 + k_2) \nonumber \\
%& \ge (1 - \epsilon) \frac{n}{2} (k_1 + k_2),
%\end{align}
where $(e)$ follows from bounding $\left|E_{tr}^{-} \right|$ by $(k_1 + k_2)/2$ in the formula \eqref{eq:M^-}, 
$(f)$ uses $k_1 + 1 \le \epsilon |V_{+}|$ and $k_2 + 1 \le \epsilon |V_{-}|$, and finally $(g)$ uses the lower bounds $|V_{+}|, |V_{-}| \geq (1-\epsilon/2)n/2$, which hold on the event $\cF_\epsilon$. 
\end{proof}

Combining these estimates with Lemma~\ref{lemma:pgf_bound}, 
the following lemma bounds the conditional probability that the 
estimate $\wh{\pi}$ has $k_{1}$ mismatches in $V_{1}$ and $k_{2}$ mismatches in $V_{2}$, for small $k_{1}$ and~$k_{2}$. 
\begin{lemma}
\label{lemma:prob_bound_1}
Fix constants $\alpha, \beta > 0$, $s \in [0,1]$, and $\epsilon \in (0,1)$ such that 
$s^2 (\alpha + \beta)/2 > (1+\epsilon)(1-\epsilon)^{-2}$. 
Given $\boldsymbol{\sigma}$, 
let $k_{1}$ and $k_{2}$ be such that 
$k_1 \le \frac{\epsilon}{2} |V_{+}|$ and $k_2 \le \frac{\epsilon}{2} |V_{-}|$. 
For all $n$ large enough we have that 
\begin{equation}
\label{eq:prob_bound_1}
\p \left( \wh{\tau} \in S_{k_1, k_2} \, \middle| \, \boldsymbol{\sigma}, \tau_* \right) \mathbf{1} \left(\cF_{\epsilon}\right) \le n^{ - \epsilon(k_1 + k_2) }.
\end{equation}
\end{lemma}

\begin{proof}
Let $\tau \in S_{k_1, k_2}$. We then have that 
\begin{align*}
\p \left( \wh{\tau} = \tau \, \middle| \, \boldsymbol{\sigma} , \tau_* \right) 
& \stackrel{(a)}{\le} \p \left( X(\tau) \le 0 \, \middle| \,  \boldsymbol{\sigma}, \tau_* \right) 
= \p \left( n^{-X(\tau)/2} \ge 1 \, \middle| \, \boldsymbol{\sigma}, \tau_{*} \right) \\
& \stackrel{(b)}{\le} \Phi^\tau \left(1/\sqrt{n}, 1, 1 \right) 
\stackrel{(c)}{\le} \exp \left( - (1 - \epsilon) s^2 \left(  \alpha M^+(\tau) +  \beta M^-(\tau) \right) \frac{\log n}{n} \right),
\end{align*} 
where $(a)$ is due to the observation made earlier that 
$\wh{\tau}$ is a minimizer of $X(\tau)$, and $X(\tau_{*}) = 0$; 
$(b)$ is due to Markov's inequality; 
and $(c)$ follows from Lemma~\ref{lemma:pgf_bound}, for all $n$ large enough. 

The estimate above allows us to bound the probability of interest via a union bound. 
To do this, we need to estimate $\left| S_{k_{1},k_{2}} \right|$. Since there are $k_1 + k_2$ mismatched vertices in total, there are at most $\binom{n}{k_1 + k_2}$ ways to choose the set of mismatched vertices (this is a loose upper bound, since this formula disregards how many mismatched vertices there are of each community). The number of possible permutations on the mismatched vertices is at most $(k_1 + k_2)!$. 
Therefore 
\[
\left| S_{k_{1},k_{2}} \right| 
\leq \binom{n}{k_1 + k_2} (k_1 + k_2)! 
= \frac{n!}{(n - k_1 - k_2)!} \le n^{k_1 + k_2}.
\]
Thus a union bound implies that 
\begin{multline}\label{eq:union_bound_exponent}
\p \left( \wh{\tau} \in S_{k_1, k_2} \, \middle| \, \boldsymbol{\sigma} , \tau_{*} \right) 
 \le \left| S_{k_1, k_2} \right| \max\limits_{\tau \in S_{k_1, k_2} } \p \left( \wh{\tau} = \tau \, \middle| \, \boldsymbol{\sigma} , \tau_{*} \right)   \\
\le \max\limits_{\tau \in S_{k_1, k_2} } \exp \left( (k_1 + k_2) \log n - (1 - \epsilon) s^2 \left( \alpha M^+(\tau) + \beta M^- (\tau) \right) \frac{\log n}{n} \right).
\end{multline}
On the event $\cF_\epsilon$, provided that $n$ is large enough, and $k_1 \le \frac{\epsilon}{2} |V_{+}|$ and $k_2 \le \frac{\epsilon}{2} |V_{-}|$, 
we may use the bounds in Lemma~\ref{lem:Mtau_lowerbounds_ksmall} to bound the exponent in \eqref{eq:union_bound_exponent} from above by
\[
\left\{ 1 - (1-\epsilon)^{2} s^{2}\left( \alpha + \beta \right) / 2  \right\} \left( k_{1} + k_{2} \right) \log n 
  \leq -\epsilon(k_1 + k_2) \log n,
\]
where the second inequality follows from the assumption that 
$s^2 (\alpha + \beta)/2 > (1+\epsilon)(1-\epsilon)^{-2}$. 
Plugging this into~\eqref{eq:union_bound_exponent} 
we have thus obtained~\eqref{eq:prob_bound_1}. 
\end{proof}

Next, we consider the second case, when either $k_{1}$ or $k_{2}$ is large; specifically, either $k_{1} \geq \frac{\epsilon}{2} |V_{+}|$ or $k_2 \geq \frac{\epsilon}{2} |V_{-}|$. 
Our goal is to obtain lemmas analogous to Lemmas~\ref{lem:Mtau_lowerbounds_ksmall} and~\ref{lemma:prob_bound_1} in this case as well. 
\begin{lemma}\label{lem:Mtau_lowerbounds_klarge} 
Fix $\pi \in \cS_{n}$ and let $\tau = \ell(\pi)$. 
Fix $\epsilon > 0$. Given community labels $\boldsymbol{\sigma}$ such that $\cF_{\epsilon}$ holds, and given $\pi_{*}$, 
let $k_{1}$ and $k_{2}$ be such that $\tau \in S_{k_1, k_2}$. 
For all $n$ large enough we have the following bounds: 
\begin{align}
    M^+(\tau) &\geq (1 - \epsilon)\frac{n}{4} (k_1 + k_2), \label{eq:M+_regime2} \\
    M^-(\tau) &\geq (1 - \epsilon)\frac{n}{4} (k_1 + k_2). \label{eq:M-_regime2}
\end{align}
\end{lemma}
\begin{proof}
On the event $\cF_\epsilon$, we have that 
\begin{align*}
M^+(\tau) & \ge \binom{k_1}{2} + k_1 ( |V_{+}| - k_1) + \binom{k_2}{2} + k_2 (|V_{-}| - k_2) - \frac{k_1 + k_2}{2} \nonumber \\
& = k_1 \left( |V_{+}|- \frac{k_1 + 2}{2}  \right) + k_2 \left( |V_{-}| - \frac{k_2 + 2}{2} \right) \nonumber  \\
& \stackrel{(h)}{\ge} \frac{1}{2} \left( k_1 (|V_{+}| - 2) + k_2 (|V_{-}| - 2) \right) \nonumber %\\
 \stackrel{(i)}{\ge} \left( 1 - \epsilon \right) \frac{n}{4} (k_1 + k_2),
\end{align*}
where $(h)$ is due to $k_1 \le |V_{+}|$ and $k_2 \le |V_{-}|$, 
and $(i)$ uses $|V_{+}| - 2 \ge (1 - \epsilon/2)|V_{+}|$ and $|V_{-}| - 2 \ge (1 - \epsilon/2) |V_{-}|$, as well as $|V_{+}|, |V_{-}| \ge (1 - \epsilon/2) n/2$, which all hold on the event $\cF_\epsilon$ for all $n$ large enough. 
For $M^-(\tau)$, we can use identical arguments to obtain~\eqref{eq:M-_regime2}. 
\end{proof}
Note that Lemma~\ref{lem:Mtau_lowerbounds_klarge} makes no assumptions on $k_{1}$ or $k_{2}$; 
however, the obtained lower bounds are smaller by a factor of $1/2$ compared to the bounds obtained in Lemma~\ref{lem:Mtau_lowerbounds_ksmall} when $k_{1}$ and $k_{2}$ are both small. 
The bounds in Lemma~\ref{lem:Mtau_lowerbounds_klarge} are used to obtain the following result, which is the analogue of Lemma~\ref{lemma:prob_bound_1}. 

\begin{lemma}
\label{lemma:prob_bound_2} 
Fix constants $\alpha, \beta > 0$, $s \in [0,1]$, and $\epsilon \in (0,1)$ such that 
$s^2 (\alpha + \beta)/2 > (1+\epsilon)(1-\epsilon)^{-2}$.  
There exists 
$\delta = \delta \left( \alpha, \beta, s, \epsilon \right) > 0$ such that the following holds. 
Given $\boldsymbol{\sigma}$, 
let $k_{1}$ and $k_{2}$ be such that 
either 
$k_1 \ge \frac{\epsilon}{2} |V_{+}|$ 
or 
$k_2 \ge \frac{\epsilon}{2} |V_{-}|$. 
For all $n$ large enough we have that 
\begin{equation}
\label{eq:prob_bound_2}
\p \left(\wh{\tau} \in S_{k_1, k_2} \, \middle| \, \boldsymbol{\sigma} , \tau_* \right) \mathbf{1}(\cF_\epsilon) \le n^{ - \delta(k_1 + k_2)}.
\end{equation}
\end{lemma}

Due to the additional factor of $1/2$ in the lower bounds for $M^+(\tau)$ and $M^-(\tau)$ in Lemma~\ref{lem:Mtau_lowerbounds_klarge} (compared to Lemma~\ref{lem:Mtau_lowerbounds_ksmall}), 
one could replicate the proof of Lemma~\ref{lemma:prob_bound_1} to show that if $s^2 (\alpha + \beta)/2 > 2$, then \eqref{eq:prob_bound_2} holds for appropriate $\delta$. In order to prove an achievability result for the {\it correct} threshold $s^2 ( \alpha + \beta)/2 > 1$, we employ a more careful analysis in which we condition on typical values of $Y^+(\tau)$ and $Y^-(\tau)$. Similar ideas were used in previous work on achieving the information-theoretic threshold for exact recovery in correlated Erd\H{o}s-R\'{e}nyi graphs \cite{cullina2018exact,wu2021settling}. Since the proof of Lemma~\ref{lemma:prob_bound_2} is more involved, we defer it to Section~\ref{subsec:lemma_prob_bound}. 

\subsection{Proof of Theorem \ref{thm:graph_matching}}

The proof of Theorem~\ref{thm:graph_matching} now readily follows from Lemmas~\ref{lemma:prob_bound_1} and~\ref{lemma:prob_bound_2}. 
\begin{proof}[Proof of Theorem~\ref{thm:graph_matching}] 
By assumption we have that $s^2 ( \alpha + \beta)/2 > 1$. 
Let $\epsilon > 0$ be sufficiently small so that 
$s^2 (\alpha + \beta)/2 > (1+\epsilon)(1-\epsilon)^{-2}$, 
and hence the conditions of Lemmas~\ref{lemma:prob_bound_1} and~\ref{lemma:prob_bound_2} are satisfied. 
Let $\delta$ be given by Lemma~\ref{lemma:prob_bound_2} 
and let $\gamma := \min \left\{ \epsilon, \delta \right\}$. 

We first argue that we may assume that the event $\cF_{\epsilon}$ holds. 
We have that 
\[
\p \left( \wh{\pi} \neq \pi_{*} \right) 
= \p \left( \wh{\tau} \neq \tau_* \right) 
= \E \left[ \p \left( \wh{\tau} \neq \tau_{*} \, \middle| \boldsymbol{\sigma}, \tau_{*} \right) \right]
\leq \E \left[ \p \left( \wh{\tau} \neq \tau_{*} \, \middle| \boldsymbol{\sigma}, \tau_{*} \right) \mathbf{1}\left(\cF_\epsilon \right) \right] + \p \left( \cF_{\epsilon}^{c} \right).
\]
Since the community labels are i.i.d.\ uniform, we have that 
$\p \left( \cF_{\epsilon}^{c} \right) \to 0$ as $n \to \infty$,  
and thus it remains to be shown that 
$\E \left[ \p \left( \wh{\tau} \neq \tau_{*} \, \middle| \boldsymbol{\sigma}, \tau_{*} \right) \mathbf{1}\left(\cF_\epsilon \right) \right] \to 0$ as $n \to \infty$. 

If $\wh{\tau} \neq \tau_{*}$, 
then $\wh{\pi}$ must have some incorrectly matched vertices (since $\wh{\pi}$ is a permutation, it cannot have just a single mismatched vertex); 
in other words, we must have that $\wh{\tau} \in S_{k_{1}, k_{2}}$ for some $k_{1}$ and $k_{2}$ satisfying $k_{1} + k_{2} \geq 2$. 
Thus by Lemmas~\ref{lemma:prob_bound_1} and~\ref{lemma:prob_bound_2} we have that 
\[
\p \left( \wh{\tau} \neq \tau_{*} \, \middle| \boldsymbol{\sigma}, \tau_{*} \right) \mathbf{1}\left(\cF_\epsilon \right) 
= 
\sum_{k_{1}, k_{2} : k_{1} + k_{2} \geq 2} 
\p \left( \wh{\tau} \in S_{k_{1}, k_{2}} \, \middle| \boldsymbol{\sigma}, \tau_{*} \right) \mathbf{1}\left(\cF_\epsilon \right)
\leq 
\sum_{k_{1}, k_{2} : k_{1} + k_{2} \geq 2} n^{-\gamma \left( k_{1} + k_{2} \right)}.
\]
Note that there are $\ell+1$ different pairs $(k_{1}, k_{2})$ such that $k_{1} + k_{2} = \ell$. Therefore 
\[
\sum_{k_{1}, k_{2} : k_{1} + k_{2} \geq 2} n^{-\gamma \left( k_{1} + k_{2} \right)} 
\leq \sum_{\ell = 2}^{\infty} \left( \ell + 1 \right) n^{-\gamma \ell}
= n^{-2\gamma} \sum_{\ell = 0}^{\infty} \left( \ell + 3 \right) n^{-\gamma \ell} 
\leq C n^{-2\gamma}
\]
for some finite constant $C$ depending only on $\gamma$ (and hence only on $\alpha$, $\beta$, and $s$). 
Putting together the previous two displays and taking an expectation we obtain that 
\[
\E \left[ \p \left( \wh{\tau} \neq \tau_{*} \, \middle| \boldsymbol{\sigma}, \tau_{*} \right) \mathbf{1}\left(\cF_\epsilon \right) \right] 
\leq C n^{-2\gamma},
\]
which concludes the proof. 
\end{proof}

\subsection{Generating function analysis: Proof of Lemma \ref{lemma:pgf_bound}}
\label{sec:gen_fn_analysis}

\subsubsection{Cycle decomposition of the PGF}

We begin by presenting a convenient representation of $X(\tau)$ as a sum of independent random variables (conditioned on $\boldsymbol{\sigma}$ and $\tau_*$),  based on an appropriate cycle decomposition. 
Let $\cC$ be the cycle decomposition of the lifted permutation $\tau_*^{-1} \circ \tau$, 
and note that the pairs for which $\tau_*(e) = \tau(e)$ are the fixed points of $ \tau_*^{-1} \circ \tau$. 
We can then write 
\begin{align*}
X(\tau) = \sum\limits_{e \in \cE: \tau(e) \neq \tau_*(e)} \left( A_e B_{\tau_*(e)} - A_e B_{\tau(e)} \right) 
& = \sum\limits_{C \in \cC: |C| \ge 2} \ \sum\limits_{e \in C} \left( A_e B_{\tau_*(e)} - A_e B_{\tau(e)} \right) \\
& = : \sum\limits_{C \in \cC: |C| \ge 2} X_C(\tau). 
\end{align*}
Note that 
$\left( \tau_{*} \circ \tau_{*}^{-1} \circ \tau \right) (e) 
= \tau(e)$, 
and hence 
$\left\{ \tau(e) \right\}_{e \in C} = \left\{ \tau_{*}(e) \right\}_{e \in C}$. 
Therefore $X_{C}(\tau)$ is a function of 
$\left\{ \left( A_{e}, B_{\tau_{*}(e)} \right) \right\}_{e \in C}
= \left\{ \left( A_{e}, B_{e}' \right) \right\}_{e \in C}$. 
Given $\boldsymbol{\sigma}$ and $\tau_*$, these only depend on the entries of the adjacency matrix of the parent graph corresponding to pairs $e \in C$, 
as well as the sampling variables corresponding to pairs $e \in C$. 
Thus, due to the disjointness of cycles, 
the random variables 
$\left\{ X_{C} (\tau) \right\}_{C \in \cC: |C| \ge 2}$ 
are mutually independent (given $\boldsymbol{\sigma}$ and $\tau_*$). 
This implies, in particular, that for any $\theta \in \R$ we have that
\[
\E \left[ \theta^{X(\tau)} \, \middle| \, \boldsymbol{\sigma}, \tau_* \right] 
= \prod\limits_{C \in \cC: |C| \ge 2} \E \left [ \theta^{X_C(\tau)} \, \middle| \,  \boldsymbol{\sigma}, \tau_{*} \right ].
\]
A similar factorization holds for $\Phi^{\tau}$, which is the PGF of interest. First, define
\begin{align*}
Y_C^+(\tau) & : = \sum\limits_{e \in C \cap \cE^+(\boldsymbol{\sigma}): \tau(e) \neq \tau_*(e)} A_{e} B_{\tau_*(e)}, \\
Y_C^-(\tau) & : = \sum\limits_{e \in C \cap \cE^-(\boldsymbol{\sigma}): \tau(e) \neq \tau_*(e)} A_{e} B_{\tau_*(e)}.
\end{align*}
Again due to the disjointness of cycles, 
the triples 
$\left\{ \left( X_{C}(\tau), Y_{C}^{+}(\tau), Y_{C}^{-}(\tau) \right) \right\}_{C \in \cC: |C| \ge 2}$ 
are mutually independent (given $\boldsymbol{\sigma}$ and $\tau_*$), so we have the factorization
\begin{equation}
\label{eq:Phi_factorization}
\Phi^\tau ( \theta, \omega, \zeta) = \prod\limits_{C \in \cC: |C| \ge 2} \E \left[ \theta^{X_{C}(\tau)} \omega^{Y_{C}^{+}(\tau)} \zeta^{Y_{C}^{-}(\tau)} \, \middle| \,  \boldsymbol{\sigma} , \tau_{*} \right] = :\prod\limits_{C \in \cC : |C| \ge 2} \Phi_{C}^{\tau}(\theta, \omega, \zeta).
\end{equation}
Given the factorization in \eqref{eq:Phi_factorization}, a key intermediate goal is to bound $\Phi_{C}^{\tau}$ for $C$ such that $\left|C\right| \geq 2$. This is accomplished by the following lemma. 

\begin{lemma}
\label{lemma:pgf_of_a_cycle} 
Given $\boldsymbol{\sigma}$ and $\tau_{*}$, the following holds. 
Fix a lifted permutation $\tau$, 
and let $C$ be a cycle in $ \tau_*^{-1} \circ \tau$ such that $|C| \geq 2$. 
Then for any constants $\epsilon \in (0,1)$ and $1 \le \omega, \zeta \le 3$, it holds for all~$n$ large enough that 
\begin{equation}
\label{eq:pgf_of_a_cycle}
\Phi_C^\tau \left( \frac{1}{\sqrt{n}}, \omega, \zeta \right) \le \mathrm{exp} \left( - (1 - \epsilon)s^2 \left( \alpha \left|C \cap \cE^+(\boldsymbol{\sigma}) \right| + \beta \left|C \cap \cE^-(\boldsymbol{\sigma}) \right| \right) \frac{\log n}{n} \right).
\end{equation}
\end{lemma}

The proof can be found in Section~\ref{subsubsec:pgf_of_a_cycle}. 
We remark that prior literature studying similar PGFs in different contexts (correlated Erd\H{o}s-R\'{e}nyi graphs or correlated Gaussian matrices) was able to derive \emph{exact} expressions for the PGF of a cycle due to the i.i.d.\ structure of the model considered~\cite{cullina2016improved,cullina2018exact,wu2021settling}. Deriving exact formulae for the PGF of a cycle in correlated stochastic block models is significantly more challenging due to the {\it heterogeneity} induced by different community labels in the cycle. Specifically, 
if the elements of the cycle are labelled differently, one obtains different formulae for the PGFs, even if the \emph{number} of inter-community and intra-community edges within the cycle are the same. The proof of Lemma~\ref{lemma:pgf_of_a_cycle} instead focuses on establishing simple, recursive bounds for the PGF, which ultimately leads to the right hand side in \eqref{eq:pgf_of_a_cycle}. We expect that this technique may be useful more generally in heterogeneous random graphs with independent structure, such as those generated from the Chung-Lu model~\cite{chung2006complex}. 

We now prove Lemma \ref{lemma:pgf_bound}, which follows readily from Lemma \ref{lemma:pgf_of_a_cycle}.

\begin{proof}[Proof of Lemma \ref{lemma:pgf_bound}]
Using \eqref{eq:Phi_factorization} and \eqref{eq:pgf_of_a_cycle}, we have the bound
\[
\Phi^\tau(\theta, \omega, \zeta) \le \mathrm{exp} \left( - (1 - \epsilon)s^2 \frac{\log n}{n} \sum\limits_{C \in \cC : |C| \ge 2} \left( \alpha \left| C \cap \cE^+(\boldsymbol{\sigma}) \right| + \beta \left| C \cap \cE^-(\boldsymbol{\sigma})  \right| \right) \right).
\]
Since the cycles of $\cC$ partition $\cE = \cE^+(\boldsymbol{\sigma}) \cup \cE^-(\boldsymbol{\sigma})$, we have that 
\[
\sum\limits_{C \in \cC : |C| \ge 2} \left| C \cap \cE^+(\boldsymbol{\sigma}) \right|  = \left| \cE^+(\boldsymbol{\sigma}) \right| - \sum\limits_{C \in \cC : |C| = 1} \left| C \cap \cE^+(\boldsymbol{\sigma}) \right| 
 = \left| \{ e \in \cE^+(\boldsymbol{\sigma}) : \tau(e) \neq \tau_*(e) \} \right| = M^+(\tau).
\]
Similarly, 
\[
\sum\limits_{C \in \cC : |C| \ge 2} \left| C \cap \cE^-(\boldsymbol{\sigma}) \right| = M^-(\tau),
\]
and the desired result immediately follows.
\end{proof}

\subsubsection{Bounding the PGF of a cycle: Proof of Lemma~\ref{lemma:pgf_of_a_cycle}}
\label{subsubsec:pgf_of_a_cycle}

In the following we assume that $\boldsymbol{\sigma}$ and $\tau_{*}$ are given. 
We also fix a lifted permutation $\tau$, 
as well as a cycle $C$ in $\tau_{*}^{-1} \circ \tau$ with $|C| \geq 2$. 
We enumerate the elements of $C$ by $e_{1}, \ldots, e_{|C|}$, 
where $\left( \tau_{*}^{-1} \circ \tau \right) \left( e_{k} \right) = e_{k+1}$ 
for every $k \in \left\{ 1, \ldots, |C| - 1 \right\}$, 
and $\left( \tau_{*}^{-1} \circ \tau \right) \left( e_{|C|} \right) = e_{1}$. 
For convenience of notation, we also define $e_{|C|+1} := e_{1}$, 
so that $\left( \tau_{*}^{-1} \circ \tau \right) \left( e_{k} \right) = e_{k+1}$ 
for every $1 \leq k \leq |C|$. 
Observe that, by applying $\tau_{*}$ to both sides of this equality, we have that 
\begin{equation}\label{eq:tau_tau*_switch}
\tau(e_{k}) 
= \left( \tau_{*} \circ \tau_{*}^{-1} \circ \tau \right) \left( e_{k} \right) 
= \tau_{*} \left( e_{k + 1} \right). 
\end{equation} 
Additionally, for $1 \leq k \leq |C|$, 
we set 
$\lambda_{k} := + 1$ if $e_{k} \in \cE^{+} \left( \boldsymbol{\sigma} \right)$ 
and 
$\lambda_{k} := - 1$ if $e_{k} \in \cE^{-} \left( \boldsymbol{\sigma} \right)$. 
Observe that for every $i,j \in \{0,1\}$ and $1 \leq k \leq |C|$ we have that 
\begin{equation}\label{eq:edge_pair_probabilities}
\p \left( \left( A_{e_{k}}, B_{\tau_{*}(e_{k})} \right) = (i,j) \, \middle| \, \boldsymbol{\sigma}, \tau_{*} \right) 
= 
\p \left( \left( A_{e_{k}}, B_{e_{k}}' \right) = (i,j) \, \middle| \, \boldsymbol{\sigma}, \tau_{*} \right) 
= 
\begin{cases}
p_{ij} &\text{ if } \lambda_{k} = +1, \\
q_{ij} &\text{ if } \lambda_{k} = -1.
\end{cases}
\end{equation}
Moreover, note that, given $\boldsymbol{\sigma}$ and $\tau_{*}$, 
the random pairs 
$\left\{ \left( A_{e_{k}}, B_{\tau_{*}(e_{k})} \right) \right\}_{k=1}^{|C|} 
= \left\{ \left( A_{e_{k}}, B_{e_{k}}' \right) \right\}_{k=1}^{|C|}$ 
are mutually independent. 
Next, for $1 \leq k \leq |C|$, define the random variables 
\begin{align*}
X_{k} &:= \sum_{\ell = 1}^{k} A_{e_{\ell}} B_{\tau_{*}(e_{\ell})} - A_{e_{\ell}} B_{\tau(e_{\ell})}, \\
Y_{k}^{+} &:= \sum_{\ell = 1}^{k} \mathbf{1} \left( \lambda_{\ell} = +1 \right) A_{e_{\ell}} B_{\tau_{*}(e_{\ell})}, \\
Y_{k}^{-} &:= \sum_{\ell = 1}^{k} \mathbf{1} \left( \lambda_{\ell} = -1 \right) A_{e_{\ell}} B_{\tau_{*}(e_{\ell})}.
\end{align*}
In particular, by construction we have that $X_{|C|} = X_{C}(\tau)$, $Y_{|C|}^{+} = Y_{C}^{+}(\tau)$, and $Y_{|C|}^{-} = Y_{C}^{-}(\tau)$.  
Due to~\eqref{eq:tau_tau*_switch}, 
as well as using $B_{\tau_{*}(e)} = B_{e}'$ for every $e \in \cE$, 
we may also write these quantities as 
\begin{align*}
X_{k} &= \sum_{\ell = 1}^{k} A_{e_{\ell}} B_{\tau_{*}(e_{\ell})} - A_{e_{\ell}} B_{\tau_{*}(e_{\ell+1})} 
= \sum_{\ell = 1}^{k} A_{e_{\ell}} B_{e_{\ell}}' - A_{e_{\ell}} B_{e_{\ell+1}}', \\
Y_{k}^{+} &= \sum_{\ell = 1}^{k} \mathbf{1} \left( \lambda_{\ell} = +1 \right) A_{e_{\ell}} B_{e_{\ell}}', \\
Y_{k}^{-} &= \sum_{\ell = 1}^{k} \mathbf{1} \left( \lambda_{\ell} = -1 \right) A_{e_{\ell}} B_{e_{\ell}}'.
\end{align*}
From the display above we also have that the increments satisfy 
\begin{align}
\label{eq:X_increment}
X_{k} - X_{k - 1} 
&= \begin{cases}
1 &\text{ if } \left(A_{e_{k}}, B'_{e_{k}} \right) = (1,1), B'_{e_{k + 1}} = 0, \\
-1 &\text{ if } \left(A_{e_{k}}, B'_{e_{k}} \right) = (1,0), B'_{e_{k + 1}} = 1, \\
0 &\text{ else};
\end{cases} \\
\label{eq:Y+_increment}
Y_{k}^{+} - Y_{k - 1}^{+} 
&= \begin{cases}
1 &\text{ if } \lambda_{k} = +1 \text{ and } \left(A_{e_{k}}, B'_{e_{k}} \right) = (1,1), \\
0 &\text{ else};
\end{cases} \\
\label{eq:Y-_increment}
Y_{k}^{-} - Y_{k - 1}^{-} 
&= \begin{cases}
1 &\text{ if } \lambda_{k} = -1 \text{ and } \left(A_{e_{k}}, B'_{e_{k}} \right) = (1,1), \\
0 &\text{ else}.
\end{cases}
\end{align}
Note, in particular, that none of these increments depend on $A_{e_{k+1}}$. 
Next, for $1 \leq k \leq |C|$ and $i,j,m \in \{0,1\}$, define the PGF 
\[
\phi_{k, ij, m} \left(\theta, \omega, \zeta \right) 
:= \E \left[ \theta^{X_{k}} \omega^{Y_{k}^{+}} \zeta^{Y_{k}^{-}} \, \middle| \, \boldsymbol{\sigma}, \tau_{*}, \left(A_{e_{1}}, B'_{e_{1}} \right) = (i,j), B'_{e_{k + 1}} = m \right],
\]
where we note that $\phi_{|C|,ij,m}$ is only defined when $j=m$, since $e_{|C|+1} \equiv e_{1}$. 
The following proposition relates these PGFs to $\Phi_{C}^{\tau}$, which is the PGF of interest. 
\begin{proposition}
\label{prop:pgf_equivalence}
Consider the setting described above. We then have that 
\[
\Phi_{C}^{\tau} \left(\theta, \omega, \zeta \right) 
= \begin{cases}
\sum\limits_{i,j \in \{0,1 \}} p_{ij} \phi_{|C|, ij, j}(\theta, \omega, \zeta) 
&\text{ if } \lambda_{1} = +1, \\
\sum\limits_{i,j \in \{0,1 \}} q_{ij} \phi_{|C|, ij, j}(\theta, \omega, \zeta) 
&\text{ if } \lambda_{1} = -1.
\end{cases}
\]
\end{proposition}
\begin{proof}
First, recall that $e_{|C|+1} \equiv e_{1}$, 
so conditioning on 
$\left(A_{e_{1}}, B'_{e_{1}} \right)$ 
is the same as conditioning on 
$\left(A_{e_{1}}, B'_{e_{1}} \right)$ and $B'_{e_{|C|+1}}$. 
The claim then follows from the definition of $\Phi_{C}^{\tau}$ 
by conditioning on $\left(A_{e_{1}}, B'_{e_{1}} \right)$ 
and recalling the probabilities in~\eqref{eq:edge_pair_probabilities}. 
\end{proof}

The usefulness of defining the PGFs $\left\{ \phi_{k, ij, m} \right\}_{1 \leq k \leq |C|; i,j,m \in \{0,1\}}$ is that we can compute them recursively in $k$ in a straightforward manner. 
To see this, let $2 \leq k \leq |C| - 1$ and first consider the case of $m = 0$. 
By conditioning on $\left( A_{e_{k}}, B'_{e_{k}} \right)$ and using the tower rule, we have that 
\begin{multline*}
\phi_{k,ij,0} \left( \theta, \omega, \zeta \right) 
= \E_{\left( A_{e_{k}}, B'_{e_{k}} \right)} \left[ \E \left[ \theta^{X_{k - 1} + (X_{k} - X_{k - 1})} \omega^{Y_{k-1}^{+} + (Y_{k}^{+} - Y_{k-1}^{+})} \zeta^{Y_{k-1}^{-} + (Y_{k}^{-} - Y_{k-1}^{-})} \right. \right. \\ 
\left. \left. \, \middle| \, \boldsymbol{\sigma}, \tau_{*}, \left(A_{e_{1}}, B'_{e_{1}} \right) = (i,j), B'_{e_{k + 1}} = 0, \left(A_{e_{k}}, B'_{e_{k}} \right) \right] \right].
\end{multline*}
With the additional conditioning on $\left( A_{e_{k}}, B'_{e_{k}} \right)$, 
the increments $(X_{k} - X_{k-1})$, $(Y_{k}^{+} - Y_{k-1}^{+})$, and $(Y_{k}^{-} - Y_{k - 1}^{-})$ are now deterministic. 
Indeed, from~\eqref{eq:X_increment} we see that, since $B'_{e_{k + 1}} = 0$, the increment $X_{k} - X_{k - 1}$ is equal to 1 if $\left( A_{e_{k}}, B'_{e_{k}} \right) = (1,1)$, otherwise it is zero. 
Similar statements can be made about the other increments based on~\eqref{eq:Y+_increment} and~\eqref{eq:Y-_increment}. 
Pulling the contributions from the increments out and 
putting everything together, we have that 
\begin{equation}\label{eq:phi_k_rec0}
\phi_{k, ij, 0} = \begin{cases}
(p_{00} + p_{10}) \phi_{k - 1, ij, 0} + (p_{01} + p_{11} \theta \omega ) \phi_{k - 1, ij, 1}  &\text{ if } \lambda_{k} = +1, \\
(q_{00} + q_{10} ) \phi_{k - 1, ij, 0} + (q_{01} + q_{11} \theta \zeta ) \phi_{k - 1, ij, 1}  &\text{ if } \lambda_{k} = -1.
\end{cases}
\end{equation}
Repeating similar arguments for the case $m = 1$ gives the following recursion for $2 \leq k \leq |C| - 1$: 
\begin{equation}
\label{eq:phi_k_recursion}
\begin{pmatrix}
\phi_{k, ij, 0} \\
\phi_{k, ij, 1}
\end{pmatrix}
= 
\begin{cases}
\begin{pmatrix}
p_{00} + p_{10} & p_{01} + p_{11} \theta \omega  \\
p_{00} + p_{10} \theta^{-1} & p_{01} + p_{11} \omega 
\end{pmatrix}
\begin{pmatrix}
\phi_{k - 1, ij, 0} \\
\phi_{k - 1, ij, 1}
\end{pmatrix} &\text{ if } \lambda_k = +1, \\
& \\
\begin{pmatrix}
q_{00} + q_{10} & q_{01} + q_{11} \theta \zeta  \\
q_{00} + q_{10} \theta^{-1} & q_{01} + q_{11} \zeta \\
\end{pmatrix}
\begin{pmatrix}
\phi_{k - 1, ij, 0} \\
\phi_{k - 1, ij, 1}
\end{pmatrix} &\text{ if } \lambda_k = -1.
\end{cases}
\end{equation}
The appropriate part of the recursion also holds for $k = |C|$, 
noting that $\phi_{|C|,ij,m}$ is only defined for $j = m$. 
To complete the description of the recursion, we have to also give the initial conditions, which are the PGFs for $k=1$. These can be computed as follows: 
\begin{equation}
\label{eq:phi_1}
\phi_{1, ij, m} = 
\begin{cases}
\theta^{-1} &\text{ if } (i,j) = (1,0), m = 1, \\
\theta \omega &\text{ if } (i,j) = (1,1), m = 0, \lambda_{1} = +1, \\
\theta \zeta &\text{ if } (i,j) = (1,1), m = 0, \lambda_{1} = -1, \\
\omega &\text{ if } (i,j) = (1,1), m = 1, \lambda_{1} = +1, \\
\zeta &\text{ if } (i,j) = (1,1), m = 1, \lambda_{1} = -1, \\
1 & \text{ else}.
\end{cases}
\end{equation}
To analyze the recursion~\eqref{eq:phi_k_recursion}, we first derive a useful relation between 
$\phi_{k,ij,1}$ and $\phi_{k,ij,0}$, as stated in the following lemma. 
Define 
\begin{equation}\label{eq:R}
R := \max \left\{ \frac{1 + 2s}{1 - s} \left( \frac{p_{10}}{p_{00}} \right), \frac{1 + 2s}{1 - s} \left( \frac{q_{10}}{q_{00}} \right) \right\},
\end{equation}
and note that $R = O(\log(n) / n)$ for every fixed $s \in [0,1]$. (Since $p_{10}$ and $q_{10}$ each contain a factor of $(1-s)$, this holds also for $s = 1$.) 
\begin{lemma}
\label{lemma:phi_k_bound} 
Consider the setting described above. Then for any $2 \leq k \leq |C| - 1$, $i,j \in \{0,1\}$, $1 \le \omega, \zeta \le 3$, and $\theta$ satisfying $0 < \theta \leq 1 - R$, we have that 
\begin{equation}
\label{eq:phi_k_bound}
\phi_{k, ij, 1} \le \left(1 + R\theta^{-1} \right) \phi_{k, ij, 0}. 
\end{equation}
\end{lemma}
\begin{proof}
Our proof is by induction on $k$. 
We first check that the base case holds for all $i,j \in \{0,1\}$. 
For $(i,j) = (0,0)$ or $(i,j) = (0,1)$, 
we may take the base case to be $k = 1$. 
Indeed, in these cases we have, from~\eqref{eq:phi_1}, that 
$\phi_{1,ij,0} = \phi_{1,ij,1} = 1$, 
so~\eqref{eq:phi_k_bound} holds. 
For $(i,j) = (1,0)$ or $(i,j) = (1,1)$, 
we shall take the base case to be $k=2$. 

Consider now the case of $(i,j) = (1,0)$. 
From~\eqref{eq:phi_k_recursion} and~\eqref{eq:phi_1} we then have that 
\begin{align}
\phi_{2, 10, 0} 
&= \begin{cases}
p_{00} + p_{10} + p_{01} \theta^{-1}  + p_{11} \omega   &\text{ if } \lambda_2 = + 1, \\
q_{00} + q_{10} + q_{01} \theta^{-1} + q_{11} \zeta  &\text{ if } \lambda_2 = -1;
\end{cases} \label{eq:phi_2100}\\
\phi_{2, 10, 1} 
&= \begin{cases}
p_{00}  + p_{10} \theta^{-1} + p_{01} \theta^{-1} + p_{11}  \theta^{-1} \omega  &\text{ if } \lambda_2 = + 1, \\
q_{00} + q_{10} \theta^{-1} + q_{01} \theta^{-1} + q_{11} \theta^{-1} \zeta  &\text{ if } \lambda_2 = - 1.
\end{cases} \notag
\end{align}
First consider the case of $\lambda_{2} = +1$. 
Using $\omega \leq 3$, we have that 
$\phi_{2,10,1} \leq p_{00} + p_{10} \theta^{-1} + p_{01} \theta^{-1} + 3 p_{11} \theta^{-1}$. 
Now using $p_{11} = \frac{s}{1-s} p_{10}$ and simplifying, we have that 
$\phi_{2,10,1} \leq p_{00} + p_{01} \theta^{-1} + \frac{1+2s}{1-s} p_{10} \theta^{-1}$. 
By expanding the product it can be verified that 
\[
p_{00} + p_{01} \theta^{-1} + \frac{1+2s}{1-s} p_{10} \theta^{-1} 
\leq 
\left( 1 + \frac{(1 + 2s)p_{10}}{(1 - s) p_{00}} \theta^{-1} \right) \phi_{2, 10, 0},
\]
which concludes the check of~\eqref{eq:phi_k_bound} in this case. 
Analogously, if $\lambda_{2} = -1$, then 
\[
\phi_{2,10,1} \leq \left( 1 + \frac{(1 + 2s)q_{10}}{(1 - s) q_{00}} \theta^{-1} \right) \phi_{2, 10, 0},
\]
concluding the check of the base case for $(i,j) = (1,0)$. 

Finally, consider the case of $(i,j) = (1,1)$. 
From~\eqref{eq:phi_k_recursion} and~\eqref{eq:phi_1} we then have that 
\begin{align}
\phi_{2, 11, 0} 
&= \begin{cases}
p_{00} \theta \omega + p_{10} \theta \omega + p_{01} \omega + p_{11} \theta \omega^2 &\text{ if } \lambda_1 = \lambda_2 = +1, \\ 
q_{00} \theta \omega + q_{10} \theta \omega + q_{01} \omega + q_{11} \theta \zeta \omega &\text{ if } \lambda_1 = +1, \lambda_2 = -1, \\
p_{00} \theta \zeta + p_{10} \theta \zeta + p_{01} \zeta + p_{11} \theta \zeta \omega &\text{ if } \lambda_1 = -1, \lambda_2 = + 1, \\
q_{00} \theta \zeta + q_{10} \theta \zeta + q_{01} \zeta + q_{11} \theta \zeta^2 &\text{ if } \lambda_1 = \lambda_2 = -1;
\end{cases} \label{eq:phi_2110}\\
\phi_{2, 11, 1} & = \begin{cases}
p_{00} \theta \omega + p_{10} \omega + p_{01} \omega + p_{11} \omega^2 &\text{ if } \lambda_1 = \lambda_2 = +1, \\
q_{00} \theta \omega + q_{10} \omega + q_{01} \omega + q_{11} \zeta \omega &\text{ if } \lambda_1 = + 1, \lambda_2 = -1, \\
p_{00} \theta \zeta + p_{10} \zeta + p_{01} \zeta + p_{11} \zeta \omega &\text{ if } \lambda_1 = - 1, \lambda_2 = + 1, \\
q_{00} \theta \zeta + q_{10} \zeta + q_{01} \zeta + q_{11} \zeta^2 &\text{ if } \lambda_1 = \lambda_2 = - 1.
\end{cases} \notag
\end{align}
Now if $\lambda_{1} = +1$, 
then
we have that 
$\phi_{2,11,1} = \theta \omega \phi_{2,10,1}$ 
and 
$\phi_{2,11,0} = \theta \omega \phi_{2,10,0}$, 
so~\eqref{eq:phi_k_bound} follows from the previous paragraph. 
If $\lambda_{1} = -1$, then we have that 
$\phi_{2,11,1} = \theta \zeta \phi_{2,10,1}$ 
and 
$\phi_{2,11,0} = \theta \zeta \phi_{2,10,0}$,
so~\eqref{eq:phi_k_bound} again follows from the previous paragraph.

Now that we have fully checked all base cases, we turn to the inductive step. 
Suppose that $\lambda_{k} = +1$; the other case where $\lambda_{k} = -1$ is similar (with $\{p_{ij}\}$ replaced with $\{ q_{ij} \}$ and $\omega$ replaced with $\zeta$). 
From the recursion~\eqref{eq:phi_k_recursion}, we have that~\eqref{eq:phi_k_bound} is equivalent to 
\begin{multline*}
\left(p_{00} + p_{10} \theta^{-1} \right) \phi_{k-1,ij,0} + (p_{01} + p_{11}  \omega) \phi_{k-1,ij,1} \\
\leq 
\left( 1 + R \theta^{-1} \right) 
\left( (p_{00} + p_{10}) \phi_{k-1,ij,0} + (p_{01} + p_{11} \theta \omega) \phi_{k-1,ij,1} \right),
\end{multline*}
which in turn is equivalent to 
\begin{multline*}
\left( p_{01} + p_{11} \omega - \left( 1 + R \theta^{-1} \right) \left( p_{01} + p_{11} \theta \omega \right) \right) 
\phi_{k-1,ij,1} \\
\leq 
\left( \left( 1 + R \theta^{-1} \right) \left( p_{00} + p_{10} \right) - \left( p_{00} + p_{10} \theta^{-1} \right) \right) 
\phi_{k-1,ij,0}.
\end{multline*}
Note that the coefficient on the left hand side satisfies 
\[
p_{01} + p_{11} \omega - \left( 1 + R \theta^{-1} \right) \left( p_{01} + p_{11} \theta \omega \right) 
\leq p_{11} \omega 
\leq 3 p_{11},
\]
so it suffices to show that 
\[
3 p_{11} \phi_{k-1,ij,1} 
\leq 
\left( \left( 1 + R \theta^{-1} \right) \left( p_{00} + p_{10} \right) - \left( p_{00} + p_{10} \theta^{-1} \right) \right) 
\phi_{k-1,ij,0}.
\]
By the induction hypothesis we have that 
$\phi_{k-1,ij,1} \leq \left( 1 + R \theta^{-1} \right) \phi_{k-1,ij,0}$, 
so it suffices to show that 
\begin{equation}\label{eq:phi_k_rec_bound_to_show}
3 \left( 1 + R \theta^{-1} \right) p_{11}
\leq 
\left( 1 + R \theta^{-1} \right) \left( p_{00} + p_{10} \right) - \left( p_{00} + p_{10} \theta^{-1} \right).
\end{equation}
The assumption $\theta \leq 1 - R$ implies that 
$1 + R\theta^{-1} \leq \theta^{-1}$. 
Using this and also that 
$p_{11} = \frac{s}{1-s} p_{10}$, 
we may bound the left hand side of~\eqref{eq:phi_k_rec_bound_to_show} as follows: 
\begin{align}
3 \left( 1 + R \theta^{-1} \right) p_{11}
&\leq \frac{3s}{1-s} p_{10} \theta^{-1} 
= \left( \frac{1+2s}{1-s} \cdot \frac{p_{10}}{p_{00}} \cdot p_{00} - p_{10} \right) \theta^{-1} \notag \\
&\leq \left( R p_{00} - p_{10} \right) \theta^{-1} 
= R\theta^{-1} p_{00} - p_{10} \theta^{-1}, \label{eq:phi_k_rec_final}
\end{align}
where we also used the definition of $R$. 
The right hand side of~\eqref{eq:phi_k_rec_final} is at most the right hand side of~\eqref{eq:phi_k_rec_bound_to_show}, which concludes the proof. 
\end{proof}

We are now ready to put everything together to prove Lemma~\ref{lemma:pgf_of_a_cycle}. 

\begin{proof}[Proof of Lemma~\ref{lemma:pgf_of_a_cycle}] 
Set $\theta := 1/\sqrt{n}$. 
Since $R$, as defined in~\eqref{eq:R}, satisfies $R = O \left( \log(n) / n \right)$, 
the condition $0 < \theta \leq 1 - R$ of Lemma~\ref{lemma:phi_k_bound} is satisfied for all $n$ large enough. 
Moreover, since $R\theta^{-1} = O \left( \log(n) / \sqrt{n} \right)$, 
we can make $R\theta^{-1}$ arbitrarily small for $n$ large enough. 
To simplify notation in what we follows, we write 
\begin{align*}
    C^{+} &:= \left| \left\{ 1 \leq k \leq |C| : \lambda_{k} = +1 \right\} \right| = \left| C \cap \cE^{+} \left( \boldsymbol{\sigma} \right) \right|, \\
    C^{-} &:= \left| \left\{ 1 \leq k \leq |C| : \lambda_{k} = -1 \right\} \right| = \left| C \cap \cE^{-} \left( \boldsymbol{\sigma} \right) \right|.
\end{align*}

Our goal is to bound $\Phi_{C}^{\tau}$. 
Due to Proposition~\ref{prop:pgf_equivalence}, 
it suffices to bound the PGFs 
$\phi_{|C|, ij,j}$ for $i,j \in \{0,1\}$. 
To do this, we use the recursion~\eqref{eq:phi_k_recursion}, 
as well as Lemma~\ref{lemma:phi_k_bound}. 
Ideally, we would like to present a streamlined argument that works for all $i,j \in \left\{ 0, 1 \right\}$ simultaneously. 
However, there are minor differences in boundary cases for different values of $i,j \in \{0,1\}$. 
Specifically, while the bound~\eqref{eq:phi_k_bound} in Lemma~\ref{lemma:phi_k_bound} holds for all $2 \leq k \leq |C| - 1$ and all $i,j \in \{0,1\}$, 
for $k=1$ it only holds when $i = 0$ (see the beginning of the proof of Lemma~\ref{lemma:phi_k_bound}). 
Furthermore, $\phi_{|C|,ij,m}$ is only defined for $j = m$. 
For these reasons, we bound $\phi_{|C|, ij,j}$ 
separately for each $i,j \in \{0,1\}$ (while minimizing repeated arguments).

We first consider the case of $(i,j) = (0,0)$ and bound $\phi_{|C|,00,0}$. 
In this case the bound~\eqref{eq:phi_k_bound} in Lemma~\ref{lemma:phi_k_bound} holds for all $1 \leq k \leq |C| - 1$. 
Noting that for $(i,j) = (0,0)$ the recursion~\eqref{eq:phi_k_rec0} holds for all $2 \leq k \leq |C|$, by plugging in~\eqref{eq:phi_k_bound} we obtain that the following holds for all $2 \leq k \leq |C|$: 
\begin{equation}\label{eq:phi_k00_rec}
\phi_{k, 00, 0} 
\leq
\begin{cases}
\left(p_{00} + p_{10} + \left( 1 + R\theta^{-1} \right) (p_{01} + p_{11} \theta \omega ) \right) \phi_{k - 1, 00, 0} &\text{ if } \lambda_{k} = +1, \\
\left(q_{00} + q_{10} + \left( 1 + R\theta^{-1} \right) (q_{01} + q_{11} \theta \zeta ) \right) \phi_{k - 1, 00, 0}  &\text{ if } \lambda_{k} = -1.
\end{cases}
\end{equation}
To simplify this recursion, 
first note that $p_{01} + p_{11} \theta \omega = (1+o(1)) p_{01}$ as $n \to \infty$, 
since $p_{01}$ and $p_{11}$ are on the same order, $\omega$ is bounded, and $\theta = o(1)$. Also recall that $R\theta^{-1} = o(1)$. 
Consequently we have that 
\[
p_{00} + p_{10} + \left( 1 + R\theta^{-1} \right) (p_{01} + p_{11} \theta \omega )
= 
p_{00} + p_{10} + (1+o(1)) p_{01}
= 1 - (1+o(1)) p_{11}
\]
as $n \to \infty$. 
Thus for any fixed $\epsilon \in (0,1)$ we have, for all $n$ large enough, that 
\[
p_{00} + p_{10} + \left( 1 + R\theta^{-1} \right) (p_{01} + p_{11} \theta \omega )
\leq 
1 - (1-\epsilon) p_{11} 
\leq 
\exp \left( - (1-\epsilon) p_{11} \right),
\]
where we have used the inequality $1 + x \leq \exp(x)$. 
Similarly we have that 
\[
q_{00} + q_{10} + \left( 1 + R\theta^{-1} \right) (q_{01} + q_{11} \theta \zeta )
\leq 
\exp \left( - (1-\epsilon) q_{11} \right)
\]
for all $n$ large enough. 
Plugging these inequalities back into~\eqref{eq:phi_k00_rec}, 
for all $n$ large enough the following holds for all $2 \leq k \leq |C|$: 
\begin{equation}\label{eq:phi_k00_rec_simple}
\phi_{k, 00, 0} 
\leq
\begin{cases}
\exp \left( - (1-\epsilon) p_{11} \right) \phi_{k - 1, 00, 0} &\text{ if } \lambda_{k} = +1, \\
\exp \left( - (1-\epsilon) q_{11} \right) \phi_{k - 1, 00, 0}  &\text{ if } \lambda_{k} = -1.
\end{cases}
\end{equation}
Iterating this inequality and noting that $\phi_{1,00,0} = 1$, 
we have thus obtained that
\begin{equation}\label{eq:phi_C00_bound}
\phi_{|C|, 00, 0} 
\leq
\begin{cases}
\exp \left( - (1-\epsilon) \left\{ \left( C^{+} - 1 \right) p_{11} + C^{-} q_{11} \right\} \right) &\text{ if } \lambda_{1} = +1, \\
\exp \left( - (1-\epsilon) \left\{ C^{+} p_{11} + \left( C^{-} - 1 \right) q_{11} \right\} \right) &\text{ if } \lambda_{1} = -1.
\end{cases}
\end{equation}

Next, we turn to the case of $(i,j) = (0,1)$, with the goal of bounding $\phi_{|C|,01,1}$. 
First, we shall bound $\phi_{|C|-1,01,0}$. 
By the same arguments as before (using the recursion and Lemma~\ref{lemma:phi_k_bound}), we have that, for all $n$ large enough, the following holds for all $2 \leq k \leq |C|-1$: 
\begin{equation}\label{eq:phi_k01_rec_simple}
\phi_{k, 01, 0} 
\leq
\begin{cases}
\exp \left( - (1-\epsilon) p_{11} \right) \phi_{k - 1, 01, 0} &\text{ if } \lambda_{k} = +1, \\
\exp \left( - (1-\epsilon) q_{11} \right) \phi_{k - 1, 01, 0}  &\text{ if } \lambda_{k} = -1.
\end{cases}
\end{equation}
Iterating this inequality and noting that $\phi_{1,01,0} = 1$, 
we thus have that
\[
\phi_{|C|-1,01,0} \leq 
\begin{cases}
\exp \left( - (1-\epsilon) \left\{ \left( C^{+} - 2 \right) p_{11} + C^{-} q_{11} \right\} \right) &\text{ if } \lambda_{1} = \lambda_{|C|} = +1, \\
\exp \left( - (1-\epsilon) \left\{ \left( C^{+} - 1 \right) p_{11} + \left( C^{-} - 1 \right) q_{11} \right\} \right) &\text{ if } \lambda_{1} \cdot \lambda_{|C|} = -1, \\
\exp \left( - (1-\epsilon) \left\{ C^{+} p_{11} + \left( C^{-} - 2 \right) q_{11} \right\} \right) &\text{ if } \lambda_{1} = \lambda_{|C|} = -1.
\end{cases}
\]
Recall that $p_{11}, q_{11} = O \left( \log(n) / n \right)$, 
and so 
$\exp(p_{11}), \exp(q_{11}) = 1 + O \left( \log(n) / n \right)$. 
Therefore regardless of the value of $\lambda_{|C|}$, we have that 
\begin{equation}\label{eq:phi_C01_bound_almost}
\phi_{|C|-1,01,0} \leq 
\left( 1 + O \left( \tfrac{\log n}{n} \right) \right) \cdot 
\begin{cases}
\exp \left( - (1-\epsilon) \left\{ \left( C^{+} - 1 \right) p_{11} + C^{-} q_{11} \right\} \right) &\text{ if } \lambda_{1} = +1, \\
\exp \left( - (1-\epsilon) \left\{ C^{+} p_{11} + \left( C^{-} - 1 \right) q_{11} \right\} \right) &\text{ if } \lambda_{1} = -1.
\end{cases}
\end{equation}
Now turning to $\phi_{|C|,01,1}$, the recursion and Lemma~\ref{lemma:phi_k_bound} together give that 
\[
\phi_{|C|,01,1} \leq 
\begin{cases}
\left( p_{00} + p_{10} \theta^{-1} + \left( 1 + R \theta^{-1} \right) \left( p_{01} + p_{11} \omega \right) \right) \phi_{|C|-1,01,0} &\text{ if } \lambda_{|C|} = +1, \\
\left( q_{00} + q_{10} \theta^{-1} + \left( 1 + R \theta^{-1} \right) \left( q_{01} + q_{11} \zeta \right) \right) \phi_{|C|-1,01,0} &\text{ if } \lambda_{|C|} = -1.
\end{cases}
\]
Recalling the values of the parameters in these coefficients, regardless of the value of $\lambda_{|C|}$ we have that 
\begin{equation}\label{eq:phi_C011}
\phi_{|C|,01,1} \leq 
\left( 1 + O \left( \tfrac{\log n}{\sqrt{n}} \right) \right) \phi_{|C|-1,01,0}.
\end{equation}
Plugging this back into~\eqref{eq:phi_C01_bound_almost}, we obtain that 
\begin{equation}\label{eq:phi_C01_bound}
\phi_{|C|,01,1} \leq 
\left( 1 + O \left( \tfrac{\log n}{\sqrt{n}} \right) \right) \cdot 
\begin{cases}
\exp \left( - (1-\epsilon) \left\{ \left( C^{+} - 1 \right) p_{11} + C^{-} q_{11} \right\} \right) &\text{ if } \lambda_{1} = +1, \\
\exp \left( - (1-\epsilon) \left\{ C^{+} p_{11} + \left( C^{-} - 1 \right) q_{11} \right\} \right) &\text{ if } \lambda_{1} = -1.
\end{cases}
\end{equation}

Next, we turn to the case of $(i,j) = (1,0)$, with the goal of bounding $\phi_{|C|,10,0}$. 
Note that in this case the bound in~\eqref{eq:phi_k_bound} only holds for $2 \leq k \leq |C| -1$. 
By the same arguments as before (using the recursion and Lemma~\ref{lemma:phi_k_bound}), we have that, for all $n$ large enough, the following holds for all $3 \leq k \leq |C|$: 
\begin{equation}\label{eq:phi_k10_rec_simple}
\phi_{k, 10, 0} 
\leq
\begin{cases}
\exp \left( - (1-\epsilon) p_{11} \right) \phi_{k - 1, 10, 0} &\text{ if } \lambda_{k} = +1, \\
\exp \left( - (1-\epsilon) q_{11} \right) \phi_{k - 1, 10, 0}  &\text{ if } \lambda_{k} = -1.
\end{cases}
\end{equation}
Iterating this inequality gives that 
\[
\phi_{|C|,10,0} \leq 
\phi_{2,10,0} \cdot 
\begin{cases}
\exp \left( - (1-\epsilon) \left\{ \left( C^{+} - 2 \right) p_{11} + C^{-} q_{11} \right\} \right) &\text{ if } \lambda_{1} = \lambda_{2} = +1, \\
\exp \left( - (1-\epsilon) \left\{ \left( C^{+} - 1 \right) p_{11} + \left( C^{-} - 1 \right) q_{11} \right\} \right) &\text{ if } \lambda_{1} \cdot \lambda_{2} = -1, \\
\exp \left( - (1-\epsilon) \left\{ C^{+} p_{11} + \left( C^{-} - 2 \right) q_{11} \right\} \right) &\text{ if } \lambda_{1} = \lambda_{2} = -1.
\end{cases}
\]
From~\eqref{eq:phi_2100} we have that 
$\phi_{2,10,0} = 1 + O \left( \log(n) / \sqrt{n} \right)$, regardless of the value of $\lambda_{2}$. 
Using again that $\exp(p_{11}), \exp(q_{11}) = 1 + O \left( \log(n) / n \right)$, 
we thus have that 
\begin{equation}\label{eq:phi_C10_bound}
\phi_{|C|,10,0} \leq 
\left( 1 + O \left( \tfrac{\log n}{\sqrt{n}} \right) \right) \cdot 
\begin{cases}
\exp \left( - (1-\epsilon) \left\{ \left( C^{+} - 1 \right) p_{11} + C^{-} q_{11} \right\} \right) &\text{ if } \lambda_{1} = +1, \\
\exp \left( - (1-\epsilon) \left\{ C^{+} p_{11} + \left( C^{-} - 1 \right) q_{11} \right\} \right) &\text{ if } \lambda_{1} = -1.
\end{cases}
\end{equation}

Finally, we turn to the case of $(i,j) = (1,1)$, with the goal of bounding $\phi_{|C|,11,1}$. Similarly to~\eqref{eq:phi_C011}, we have that 
\begin{equation}\label{eq:C111}
\phi_{|C|,11,1} 
\leq 
\left( 1 + O \left( \tfrac{\log n}{\sqrt{n}} \right) \right) 
\phi_{|C|-1,11,0},
\end{equation}
and so in the following we bound $\phi_{|C|-1,11,0}$. 
By the recursion and Lemma~\ref{lemma:phi_k_bound}, we have that, for all $n$ large enough, the following holds for all $3 \leq k \leq |C|-1$:
\begin{equation}\label{eq:phi_k11_rec_simple}
\phi_{k, 11, 0} 
\leq
\begin{cases}
\exp \left( - (1-\epsilon) p_{11} \right) \phi_{k - 1, 11, 0} &\text{ if } \lambda_{k} = +1, \\
\exp \left( - (1-\epsilon) q_{11} \right) \phi_{k - 1, 11, 0}  &\text{ if } \lambda_{k} = -1.
\end{cases}
\end{equation}
Iterating this inequality gives that 
\[
\phi_{|C|-1,11,0} \leq 
\phi_{2,11,0} \cdot 
\begin{cases}
\exp \left( - (1-\epsilon) \left\{ \left( C^{+} - 3 \right) p_{11} + C^{-} q_{11} \right\} \right) &\text{ if } \lambda_{1} = \lambda_{2} = \lambda_{|C|} = +1, \\
\exp \left( - (1-\epsilon) \left\{ \left( C^{+} - 2 \right) p_{11} + \left( C^{-} - 1 \right) q_{11} \right\} \right) &\text{ if } \left| \left\{ i \in \{1,2,|C|\} : \lambda_{i} = +1 \right\} \right| = 2, \\
\exp \left( - (1-\epsilon) \left\{ \left( C^{+} - 1 \right) p_{11} + \left( C^{-} - 2 \right) q_{11} \right\} \right) &\text{ if } \left| \left\{ i \in \{1,2,|C|\} : \lambda_{i} = +1 \right\} \right| = 1, \\
\exp \left( - (1-\epsilon) \left\{ C^{+} p_{11} + \left( C^{-} - 3 \right) q_{11} \right\} \right) &\text{ if } \lambda_{1} = \lambda_{2} = \lambda_{|C|} = -1.
\end{cases}
\]
From~\eqref{eq:phi_2110} we have that $\phi_{2,11,0} = O \left( 1 / \sqrt{n} \right)$, regardless of the values of $\lambda_{1}$ and $\lambda_{2}$. 
Using again that $\exp(p_{11}), \exp(q_{11}) = 1 + O \left( \log(n) / n \right)$, 
we thus have that 
\[
\phi_{|C|-1,11,0} \leq 
O \left( \tfrac{1}{\sqrt{n}} \right) \cdot 
\begin{cases}
\exp \left( - (1-\epsilon) \left\{ \left( C^{+} - 1 \right) p_{11} + C^{-} q_{11} \right\} \right) &\text{ if } \lambda_{1} = +1, \\
\exp \left( - (1-\epsilon) \left\{ C^{+} p_{11} + \left( C^{-} - 1 \right) q_{11} \right\} \right) &\text{ if } \lambda_{1} = -1.
\end{cases}
\]
Plugging this back into~\eqref{eq:C111}, we thus have that 
\begin{equation}\label{eq:phi_C11_bound}
\phi_{|C|,11,1} \leq 
O \left( \tfrac{1}{\sqrt{n}} \right) \cdot 
\begin{cases}
\exp \left( - (1-\epsilon) \left\{ \left( C^{+} - 1 \right) p_{11} + C^{-} q_{11} \right\} \right) &\text{ if } \lambda_{1} = +1, \\
\exp \left( - (1-\epsilon) \left\{ C^{+} p_{11} + \left( C^{-} - 1 \right) q_{11} \right\} \right) &\text{ if } \lambda_{1} = -1.
\end{cases}
\end{equation}

We have now computed bounds for $\phi_{|C|,ij,j}$ for all $i,j \in \{0,1\}$, 
and so we are now ready to bound $\Phi_{C}^{\tau}$. 
Suppose that $\lambda_{1} = +1$; the other case is analogous. 
By Proposition~\ref{prop:pgf_equivalence} and the bounds in~\eqref{eq:phi_C00_bound},~\eqref{eq:phi_C01_bound},~\eqref{eq:phi_C10_bound}, and~\eqref{eq:phi_C11_bound}, 
we have that 
\begin{align*}
\Phi_{C}^{\tau} 
&= p_{00} \phi_{|C|,00,0} + p_{01} \phi_{|C|,01,1} + p_{10} \phi_{|C|,10,0} + p_{11} \phi_{|C|,11,1} \\
&\leq 
\exp \left( - (1-\epsilon) \left\{ \left( C^{+} - 1 \right) p_{11} + C^{-} q_{11} \right\} \right) 
\cdot 
\left\{ p_{00} + \left( 1 + O \left( \tfrac{\log n}{\sqrt{n}} \right) \right) \left( p_{01} + p_{10} \right) + O \left( \tfrac{1}{\sqrt{n}} \right) p_{11} \right\}.
\end{align*} 
Observe that 
\[
p_{00} + \left( 1 + O \left( \tfrac{\log n}{\sqrt{n}} \right) \right) \left( p_{01} + p_{10} \right) + O \left( \tfrac{1}{\sqrt{n}} \right) p_{11} 
= 1 - p_{11} + O \left( \tfrac{\log^{2}(n)}{n^{3/2}} \right),
\]
so for all $n$ large enough this is at most $1 - (1-\epsilon) p_{11} \leq \exp \left( - (1-\epsilon) p_{11} \right)$. 
Plugging this back into the inequality above, we obtain that 
\[
\Phi_{C}^{\tau} 
\leq 
\exp \left( - \left( 1 - \epsilon \right) \left\{ C^{+} p_{11} + C^{-} q_{11} \right\} \right).
\]
Recalling the definitions of $p_{11}$ and $q_{11}$ shows that we have obtained the desired inequality. 
\end{proof}

\subsection{Proof of Lemma \ref{lemma:prob_bound_2}}
\label{subsec:lemma_prob_bound}

For any $t^{+}$ and $t^{-}$ (to be chosen later) we have that 
\begin{align}
\p \left( \wh{\tau} \in S_{k_1, k_2} \, \middle| \, \boldsymbol{\sigma} , \tau_* \right) & \le \p \left( \exists \tau \in S_{k_1, k_2} : X(\tau) \le 0 \, \middle| \, \boldsymbol{\sigma}, \tau_* \right) \nonumber \\
\label{eq:prob_bound_part1}
& \le \p \left( \exists \tau \in S_{k_1, k_2}: X(\tau) \le 0, Y^+(\tau) \ge t^+, Y^-(\tau) \ge t^- \, \middle| \, \boldsymbol{\sigma}, \tau_* \right) \\
\label{eq:prob_bound_part2}
&\quad + \p \left( \exists \tau \in S_{k_1, k_2} : Y^+(\tau) \leq t^+ \, \middle| \, \boldsymbol{\sigma}, \tau_* \right) \\
\label{eq:prob_bound_part3}
&\quad + \p \left( \exists \tau \in S_{k_1, k_2} : Y^-(\tau) \leq t^- \, \middle| \, \boldsymbol{\sigma}, \tau_* \right).
\end{align}
In the following we bound from above each of these three terms, 
starting with~\eqref{eq:prob_bound_part1}. For any $\tau \in S_{k_1, k_2}$, and any $\theta \in (0,1]$ and $\omega, \zeta \ge 1$, we have that 
\begin{multline*}
\p \left(X(\tau) \le 0, Y^+(\tau) \ge t^+, Y^-(\tau) \ge t^- \, \middle| \, \boldsymbol{\sigma}, \tau_* \right) \\
\begin{aligned}
&= \sum\limits_{k \le 0} \sum\limits_{k^+ \ge t^+} \sum\limits_{k^- \ge t^-} \p \left( \left( X(\tau), Y^+(\tau), Y^-(\tau) \right) = \left( k, k^+, k^- \right) \, \middle| \, \boldsymbol{\sigma}, \tau_* \right) \\
&\le \sum\limits_{k = -\infty}^\infty \sum\limits_{k^+ = -\infty}^\infty \sum\limits_{k^- = -\infty}^\infty \theta^{k} \omega^{k^{+} - t^{+}} \zeta^{k^{-} - t^{-}} \p \left( \left( X(\tau), Y^+(\tau), Y^-(\tau) \right) = \left( k, k^{+}, k^{-} \right) \, \middle| \, \boldsymbol{\sigma}, \tau_* \right) \\
&= \omega^{-t^+} \zeta^{-t^-} \Phi^\tau (\theta, \omega, \zeta).
\end{aligned}
\end{multline*}
By taking a union bound over $\tau \in S_{k_1, k_2}$ and setting $(\theta, \omega, \zeta) := \left(1/\sqrt{n}, e, e \right)$, we can thus bound the expression 
in~\eqref{eq:prob_bound_part1} from above by  
\[
\left|S_{k_1, k_2} \right| \max\limits_{\tau \in S_{k_1, k_2} } e^{- t^{+} - t^{-}} \Phi^\tau \left( 1/\sqrt{n}, e, e \right).
\]
Using the estimate $\left|S_{k_1, k_2} \right| \le n^{k_1 + k_2}$ (see the proof of Lemma~\ref{lemma:prob_bound_1}) and also Lemma~\ref{lemma:pgf_bound},  
we thus have that 
\begin{multline*}
\p \left( \exists \tau \in S_{k_1, k_2}: X(\tau) \le 0, Y^+(\tau) \ge t^+, Y^-(\tau) \ge t^- \, \middle| \, \boldsymbol{\sigma}, \tau_* \right) \\
\le 
\max\limits_{\tau \in S_{k_1, k_2} } 
\exp \left( (k_1 + k_2) \log n - \left( t^+ + t^- + (1 - \epsilon)s^2 \left( \alpha M^+(\tau) + \beta M^-(\tau) \right) \frac{\log n}{n} \right) \right).
\end{multline*} 
Noting that we may choose $t^{+}$ and $t^{-}$ as functions of $\tau$, 
set 
\begin{equation}\label{eq:tplustminus}
t^+ : = (1 - \epsilon) s^2 \alpha \frac{\log n}{n} M^+(\tau) \qquad \text{and} \qquad t^- : = (1 - \epsilon)s^2 \beta \frac{\log n}{n} M^-(\tau). 
\end{equation}
In this way the expression 
in~\eqref{eq:prob_bound_part1} is bounded from above by 
\begin{equation}
\label{eq:exponent_bound1} 
\max\limits_{\tau \in S_{k_1, k_2} } 
\exp \left( (k_1 + k_2) \log n - 2(1 - \epsilon)s^2 \left( \alpha M^+(\tau) + \beta M^-(\tau) \right) \frac{\log n}{n} \right).
\end{equation}
On the event $\cF_{\epsilon}$, provided that $n$ is large enough, 
we may use the bounds in Lemma~\ref{lem:Mtau_lowerbounds_klarge} 
for $M^+(\tau)$ and $M^-(\tau)$ to bound the exponent in~\eqref{eq:exponent_bound1} from above by 
\[
\left\{ 1 - (1-\epsilon)^{2} s^{2}\left( \alpha + \beta \right) / 2  \right\} \left( k_{1} + k_{2} \right) \log n 
  \leq -\epsilon(k_1 + k_2) \log n,
\]
where the second inequality follows from the assumption that 
$s^2 (\alpha + \beta)/2 > (1+\epsilon)(1-\epsilon)^{-2}$. 
We have thus obtained, for all $n$ large enough, that 
\[
\p \left( \exists \tau \in S_{k_1, k_2}: X(\tau) \le 0, Y^+(\tau) \ge t^+, Y^-(\tau) \ge t^- \, \middle| \, \boldsymbol{\sigma}, \tau_* \right) \mathbf{1} \left(\cF_{\epsilon} \right) 
\leq 
n^{- \epsilon(k_1 + k_2)}. 
\]

Next, we turn to bounding~\eqref{eq:prob_bound_part2}, and recall that we have set $t^{+}$ as in~\eqref{eq:tplustminus}. 
We shall first relate $Y^+(\tau)$ to a similar quantity which depends only on the correctly matched region of the corresponding vertex permutation $\pi$. 
Formally, given $\pi_{*}$ (equivalently, $\tau_{*}$), define the sets
\[
F(\pi) := \{ v \in V : \pi(v) = \pi_*(v) \} 
\qquad 
\text{ and }
\qquad 
\binom{F(\pi)}{2} := \{ \{u,v\} : u,v \in F(\pi), u \neq v \}.
\]
In words, $F(\pi)$ is the set of correctly matched vertices according to $\pi$, 
and $\binom{F(\pi)}{2}$ is the set of unordered pairs in $F(\pi)$. 
We can then write
\begin{align*}
Y^+(\tau) & = \sum\limits_{e \in \cE^+(\boldsymbol{\sigma}): \tau(e) \neq \tau_*(e)} A_e B_{\tau_{*}(e)}  
\stackrel{(a)}{=} \sum\limits_{e \in \cE^+(\boldsymbol{\sigma})\setminus \left( \binom{F(\pi)}{2} \cup E_{tr}^{+} \right)} A_e B_{\tau_{*}(e)} \\
& = \sum\limits_{e \in \cE^+(\boldsymbol{\sigma}) \setminus \binom{F(\pi)}{2}} A_e B_{\tau_{*}(e)} - \sum\limits_{e \in E_{tr}^+} A_e B_{\tau_{*}(e)} 
\stackrel{(b)}{\ge} \sum\limits_{e \in \cE^+(\boldsymbol{\sigma}) \setminus \binom{F(\pi)}{2}} A_e B_{\tau_{*}(e)} - \frac{n}{2}.
\end{align*}
Above, 
$(a)$ follows since $\tau(e) = \tau_*(e)$ 
if and only if 
either both endpoints of $e$ are fixed points of~$\pi$ 
or the endpoints of $e$ are a transposition in $\pi$; 
and $(b)$ follows since $A_e B_{\tau_{*}(e)} \in \{0,1\}$, 
so the second summation is at most $\left|E_{tr}^{+}\right| \le (k_1 + k_2)/2$ (by Lemma \ref{lemma:M}), which in turn is at most $n/2$. 
Hence $Y^{+}(\tau) \leq t^{+}$ implies that 
\begin{equation}\label{eq:sum_depending_only_on_F_and_tplus}
\sum\limits_{e \in \cE^+(\boldsymbol{\sigma}) \setminus \binom{F(\pi)}{2}} A_e B_{\tau_{*}(e)} \le t^{+} + \frac{n}{2}.
\end{equation}
To abbreviate notation, for $F \subseteq V$ let 
$H_{F} := \cE^+(\boldsymbol{\sigma}) \setminus \binom{F}{2}$ 
(where we suppress dependence on $\boldsymbol{\sigma}$ in the notation for simplicity). 
Noting that 
$M^{+}(\tau) \leq \left| H_{F(\pi)} \right|$ 
and recalling the definition of $t^{+}$,~\eqref{eq:sum_depending_only_on_F_and_tplus} further implies that 
\begin{equation}\label{eq:sum_depending_only_on_F}
\sum\limits_{e \in H_{F(\pi)}} A_e B_{\tau_{*}(e)} 
\le (1 - \epsilon) s^2 \alpha \frac{\log n}{n} \left| H_{F(\pi)} \right| + \frac{n}{2}.
\end{equation}
Importantly, note that (given $\boldsymbol{\sigma}$ and $\tau_{*}$) the sum in~\eqref{eq:sum_depending_only_on_F} depends on $\pi$ (equivalently, $\tau$) only through $F(\pi)$. 
The same holds for the right hand side of~\eqref{eq:sum_depending_only_on_F}. 
Therefore if there exists $\tau \in S_{k_{1}, k_{2}}$ such that $Y^{+}(\tau) \leq t^{+}$, 
then there exists $F \subseteq V$ such that 
$\left| V_{+} \setminus F \right| = k_{1}$, 
$\left| V_{-} \setminus F \right| = k_{2}$,  
and the inequality 
\begin{equation}\label{eq:sum_F_no_pi}
Z_{F} := 
\sum\limits_{e \in H_{F}} A_e B_{\tau_{*}(e)} 
\le (1 - \epsilon) s^2 \alpha \frac{\log n}{n} \left| H_{F} \right| + \frac{n}{2}.
\end{equation}
holds. 
Thus turning to~\eqref{eq:prob_bound_part2}, a union bound gives that 
\begin{align}
&\p \left( \exists \tau \in S_{k_1, k_2} : Y^+(\tau) \le t^+ \, \middle| \, \boldsymbol{\sigma} , \tau_* \right) \notag \\
&\qquad\le \p \left( \exists F \subseteq V \text{ with } |V_{+} \setminus F| = k_{1}, |V_{-} \setminus F| = k_2 : 
Z_{F} \le (1 - \epsilon) s^2 \alpha \frac{\log n}{n} \left| H_{F} \right| + \frac{n}{2} \, \middle| \, \boldsymbol{\sigma}, \tau_* \right) \notag \\
&\qquad\le \sum\limits_{F \subseteq V: |V_{+} \setminus F| = k_{1}, |V_{-} \setminus F| = k_{2}} \p \left( Z_{F} \le (1 - \epsilon) s^2 \alpha \frac{\log n}{n} \left| H_{F} \right| + \frac{n}{2} \, \middle| \, \boldsymbol{\sigma}, \tau_* \right) \notag \\
&\qquad\le 2^{n} \max\limits_{F \subseteq V: |V_{+} \setminus F| = k_{1}, |V_{-} \setminus F| = k_{2}} \p \left( Z_{F} \le (1 - \epsilon) s^2 \alpha \frac{\log n}{n} \left| H_{F} \right| + \frac{n}{2} \, \middle| \, \boldsymbol{\sigma}, \tau_* \right). \label{eq:F_union_bound}
\end{align}

Before continuing, we make a brief remark about the purpose of the above computations. 
If we were to deal with $Y^{+}(\tau)$ directly, and take a union bound over all $\tau \in S_{k_{1}, k_{2}}$, 
we would gain a factor of $\left| S_{k_{1}, k_{2}} \right| \leq n^{k_{1}+k_{2}} = \exp \left( \Theta \left( n \log n \right) \right)$ from the union bound, which would be too large for our purposes. 
This is why it is important to switch from $Y^{+}(\tau)$ to the sum in~\eqref{eq:sum_F_no_pi}: it allows us to take a union bound over a much smaller set, resulting in a factor of only $2^{n}$, as in~\eqref{eq:F_union_bound}.

Continuing the proof, our goal is to bound the probability in~\eqref{eq:F_union_bound}. 
Notice that (conditioned on $\boldsymbol{\sigma}$ and $\tau_{*}$) for every $e \in \cE^+(\boldsymbol{\sigma})$ we have that
\[
A_e B_{\tau_{*}(e)} \sim \mathrm{Bernoulli} \left( s^2 \alpha \frac{\log n}{n} \right),
\]
and these random variables are (conditioned on $\boldsymbol{\sigma}$ and $\tau_{*}$) mutually independent across $e \in \cE^+(\boldsymbol{\sigma})$. Hence 
(conditioned on $\boldsymbol{\sigma}$ and $\tau_{*}$) we have that 
$Z_{F} \sim \Bin \left( \left| H_{F} \right|, s^{2} \alpha \log(n) / n \right)$. 
In particular, note that 
$\E \left[ Z_{F} \, \middle| \, \boldsymbol{\sigma}, \tau_{*} \right] 
= \left| H_{F} \right| s^{2} \alpha \log(n) / n$.

Note that for any $F \subseteq V$ such that 
$|V_{+} \setminus F| = k_{1}$ and $|V_{-} \setminus F| = k_{2}$, 
and for any $\pi$ such that $\ell(\pi) \in S_{k_{1},k_{2}}$,  
we have that $\left| H_{F} \right| \geq M^{+}(\ell(\pi))$. 
Therefore Lemma~\ref{lem:Mtau_lowerbounds_klarge} implies that 
$\left| H_{F} \right| \geq ((1- \epsilon)/4)(k_{1} + k_{2}) n$ 
for all $n$ large enough. 
Recall that we assume that either 
$k_1 \ge \frac{\epsilon}{2} |V_{+}|$ 
or $k_2 \ge \frac{\epsilon}{2} |V_{-}|$. 
Therefore on the event $\cF_{\epsilon}$ 
we have that 
$k_{1} + k_{2} \geq (\epsilon/2) \left( 1 - \epsilon/2 \right) n / 2$. 
Thus on the event $\cF_{\epsilon}$ we have that 
$\left| H_{F} \right| = \Omega \left( n^{2} \right)$. 
Hence on the event $\cF_{\epsilon}$ we have, 
for all $n$ large enough, that 
\begin{equation*}
\label{eq:t+_bound}
(1 - \epsilon) s^2 \alpha \frac{\log n}{n} \left| H_{F} \right|  + \frac{n}{2} 
\le (1 - \epsilon/2) \left| H_{F} \right| s^2 \alpha \frac{\log n}{n} 
= (1-\epsilon/2) \E \left[ Z_{F} \, \middle| \, \boldsymbol{\sigma}, \tau_{*} \right].
\end{equation*}
Thus for all $n$ large enough we have that
\[
\p \left( Z_{F} \le (1 - \epsilon) s^2 \alpha \frac{\log n}{n} \left| H_{F} \right| + \frac{n}{2} \, \middle| \, \boldsymbol{\sigma}, \tau_* \right) \mathbf{1}(\cF_\epsilon)
\leq 
\p \left( Z_{F} \le \left(1 - \frac{\epsilon}{2} \right) \E \left[ Z_{F} \, \middle| \, \boldsymbol{\sigma}, \tau_{*} \right] \, \middle| \, \boldsymbol{\sigma}, \tau_* \right) \mathbf{1}(\cF_\epsilon).
\]
By Bernstein's inequality we have that 
\begin{align*}
\p \left( Z_{F} \le \left( 1 - \frac{\epsilon}{2} \right) \E \left[ Z_{F} \, \middle| \, \boldsymbol{\sigma}, \tau_{*} \right] \, \middle| \, \boldsymbol{\sigma}, \tau_* \right) 
&\leq 
\mathrm{exp} \left( - \frac{ \frac{\epsilon^2}{8} \E \left[ Z_{F} \, \middle| \,\boldsymbol{\sigma}, \tau_* \right]^2}{\mathrm{Var} \left(Z_{F} \, \middle| \, \boldsymbol{\sigma}, \tau* \right) + \frac{1}{3} \E \left[ Z_{F} \, \middle| \, \boldsymbol{\sigma}, \tau_* \right]} \right) \\
&\leq 
\mathrm{exp} \left( - \frac{3 \epsilon^2}{32} \E \left[ Z_{F} \, \middle| \, \boldsymbol{\sigma}, \tau_* \right] \right),
\end{align*}
where the second inequality uses the fact that 
$\mathrm{Var}\left( Z_{F} \, \middle| \,\boldsymbol{\sigma}, \tau_* \right) \le \E \left[ Z_{F} \, \middle| \, \boldsymbol{\sigma}, \tau_* \right]$. 
Recall that 
\[
\E \left[ Z_{F} \, \middle| \, \boldsymbol{\sigma}, \tau_{*} \right] 
= \left| H_{F} \right| s^{2} \alpha \frac{\log n}{n} 
\geq \frac{1-\epsilon}{4}(k_{1}+k_{2}) s^{2} \alpha \log n 
\]
for all $n$ large enough. 
Putting everything together, we have thus shown, 
for any $F \subseteq V$ such that 
$|V_{+} \setminus F| = k_{1}$ and $|V_{-} \setminus F| = k_{2}$, 
that
\[
\p \left( Z_{F} \le (1 - \epsilon) s^2 \alpha \frac{\log n}{n} \left| H_{F} \right| + \frac{n}{2} \, \middle| \, \boldsymbol{\sigma}, \tau_* \right) \mathbf{1}(\cF_\epsilon)
\leq 
\exp \left( - \frac{ 3 \epsilon^2 (1 - \epsilon) }{128} (k_1 + k_2) s^2 \alpha  \log n \right)
\]
for all $n$ large enough. 
Plugging this back into~\eqref{eq:F_union_bound}, we obtain that 
\begin{align*}
\p \left( \exists \tau \in S_{k_1, k_2} : Y^+(\tau) \le t^+ \, \middle| \, \boldsymbol{\sigma} , \tau_* \right) \mathbf{1}(\cF_\epsilon)
&\leq 
\exp \left( n \log 2 - \frac{ 3 \epsilon^2 (1 - \epsilon) }{128} (k_1 + k_2) s^2 \alpha  \log n \right) \\
&\leq 
\exp \left( - \frac{ \epsilon^2 (1 - \epsilon) }{50} (k_1 + k_2) s^2 \alpha  \log n \right)
\end{align*}
for all $n$ large enough, 
where the second inequality follows because $(k_{1} + k_{2}) \log n = \Omega \left( n \log n \right)$, which is asymptotically much larger than $n \log 2$. 
This concludes the bound for~\eqref{eq:prob_bound_part2}.

Turning to~\eqref{eq:prob_bound_part3}, 
repeating identical steps as above also shows, for all $n$ large enough, that 
\[
\p \left( \exists \tau \in S_{k_1, k_2} : Y^{-}(\tau) \le t^{-} \, \middle| \, \boldsymbol{\sigma}, \tau_* \right) \mathbf{1}(\cF_\epsilon)
\le 
\exp \left( - \frac{\epsilon^2 (1 - \epsilon) }{50} (k_1 + k_2) s^2 \beta  \log n \right).
\]

Putting everything together, if we let 
$\delta_0 : = 
\min \left\{ \epsilon, 
\epsilon^{2}(1 - \epsilon) s^{2} \alpha /50, 
\epsilon^{2} (1 - \epsilon) s^{2} \beta /50 \right\}$, 
then the terms \eqref{eq:prob_bound_part1}, \eqref{eq:prob_bound_part2}, and \eqref{eq:prob_bound_part3} are all at most $n^{ - \delta_0 (k_1 + k_2)}$, for all $n$ large enough. 
This gives a total bound of 
$3n^{ - \delta_{0}(k_1 + k_2)} \leq n^{ - \delta_{0}(k_1 + k_2)/2}$ 
for all $n$ large enough, 
concluding the proof of Lemma~\ref{lemma:prob_bound_2}.

\section{Exact graph matching for correlated SBMs: impossibility}
\label{sec:impossibility_graph_matching}

In this section we prove Theorem~\ref{thm:graph_matching_impossibility}, 
showing that it is impossible to exactly match the two correlated SBMs $G_{1}$ and $G_{2}$ whenever $s^{2}(\alpha + \beta) / 2 < 1$. 
While this was previously proven in~\cite{cullina2016simultaneous}, we provide a proof for completeness. 
At a high level, the strategy behind the proof 
%of this result 
is as follows. 

When $s^2 \left( \alpha + \beta \right) / 2 < 1$, we show that there are many vertices in~$G$ such that the corresponding vertices in $G_{1}$ and $G_{2}'$  have non-overlapping neighborhoods. Due to this lack of shared information, such vertices are challenging to correctly match in the two graphs, even for the maximum a posteriori (MAP) estimator that is given $G_{1}$ and $G_{2}$. For this reason, the MAP estimator is likely to output an incorrect vertex correspondence. Since the MAP estimator minimizes the probability of error, we conclude that no other estimator can do better (in particular, no estimator can output the correct correspondence with probability bounded away from zero). 

The input to the estimation problem is the pair of labeled graphs $G_{1}$ and $G_{2}$; equivalently, in the following we use the respective adjacency matrices $A$ and $B$. 
To compute the MAP estimator, we need to derive the posterior distribution of $\pi_{*}$ given $A$ and $B$. 
This is unfortunately quite challenging in correlated SBMs, since the probability of edge formation depends on the (unknown) latent community memberships of vertices. 
To carry out a tractable analysis, we shall provide extra information 
to the estimator: 
we assume that $\boldsymbol{\sigma}$ is also known; 
that is, we assume knowledge of the community memberships of all vertices in $G_{1}$. 
Providing this extra information can only make the problem of estimating $\pi_{*}$ easier, 
yet it turns out that recovering $\pi_{*}$ is still impossible even with this extra information.

\subsection{Properties of the posterior distribution}

Before deriving the posterior distribution of $\pi_{*}$ given $A$, $B$, and $\boldsymbol{\sigma}$, we define some relevant notation. Given $\boldsymbol{\sigma}$, for a lifted permutation $\tau$ and $i,j \in \{0,1\}$, define
\begin{align*}
\mu^+(\tau)_{ij} & := \sum\limits_{e \in \cE^+(\boldsymbol{\sigma})} \mathbf{1} \left( \left( A_{e}, B_{\tau(e)} \right) = \left(i,j \right) \right), \\
\mu^-(\tau)_{ij} & := \sum\limits_{e \in \cE^-(\boldsymbol{\sigma})} \mathbf{1} \left( \left( A_{e}, B_{\tau(e)} \right) = \left(i,j \right) \right).
\end{align*}
Additionally define 
\[
\nu^+(\tau) := \sum\limits_{e \in \cE^+(\boldsymbol{\sigma})} B_{\tau(e)}, 
\qquad 
\qquad
\qquad
\nu^-(\tau) := \sum\limits_{e \in \cE^-(\boldsymbol{\sigma})} B_{\tau(e)}.
\]
% \begin{align*}
% \nu^+(\tau) & : = \sum\limits_{e \in \cE^+(\boldsymbol{\sigma})} B_{\tau(e)}, \\
% \nu^-(\tau) & : = \sum\limits_{e \in \cE^-(\boldsymbol{\sigma})} B_{\tau(e)}.
% \end{align*}
With these notations in place, 
and recalling the definitions of $\left\{ p_{ij} \right\}_{i,j \in \{0,1\}}$ and $\left\{ q_{ij} \right\}_{i,j \in \{0,1\}}$, 
the following lemma determines the posterior distribution of $\pi_{*}$ given~$A$, $B$, and $\boldsymbol{\sigma}$. 
\begin{lemma}[Posterior distribution]
\label{lemma:likelihood}
Let $\pi \in \cS_n$ and let $\tau = \ell(\pi)$ be the corresponding lifted permutation. There is a constant $c = c(A,B,\boldsymbol{\sigma})$ such that
\begin{equation}
\label{eq:likelihood}
\p \left( \pi_* = \pi \, \middle| \, A, B, \boldsymbol{\sigma} \right) 
= c \left( \frac{p_{00} p_{11}}{ p_{01} p_{10}} \right)^{\mu^+(\tau)_{11}} \left( \frac{q_{00} q_{11}}{q_{01} q_{10}} \right)^{\mu^-(\tau)_{11}} \left( \frac{p_{01}}{p_{00}} \right)^{\nu^+(\tau)} \left( \frac{q_{01}}{q_{00}} \right)^{\nu^-(\tau)}.
\end{equation}
\end{lemma}

\begin{proof}
By Bayes' rule, we have that 
\[
\p \left(\pi_* = \pi \, \middle| \, A, B, \boldsymbol{\sigma} \right) 
= \frac{ \p \left(A,B \, \middle| \, \pi_* = \pi, \boldsymbol{\sigma} \right) \p \left( \pi_* = \pi \, \middle| \, \boldsymbol{\sigma} \right) }{ \p \left(A,B \, \middle| \, \boldsymbol{\sigma} \right) }.
\]
Since the permutation $\pi_{*}$ is chosen uniformly at random and independently of the community labels, we have that 
$\p \left( \pi_{*} = \pi \, \middle| \, \boldsymbol{\sigma} \right) 
= \p(\pi_{*} = \pi) = 1/n!$. 
Moreover, the term in the denominator only depends on $A$, $B$, and $\boldsymbol{\sigma}$ (it does not depend on $\pi$). We can therefore write
\[
\p \left( \pi_{*} = \pi \, \middle| \, A, B, \boldsymbol{\sigma} \right) 
= c(A,B,\boldsymbol{\sigma}) \p \left(A,B \, \middle| \, \pi_* = \pi, \boldsymbol{\sigma} \right), 
\]
where 
$c(A,B,\boldsymbol{\sigma}) = 1 / \left( n! \, \p \left(A,B \, \middle| \, \boldsymbol{\sigma} \right) \right)$. 
We now focus on computing $\p \left( A,B \, \middle| \, \pi_* = \pi, \boldsymbol{\sigma} \right)$. 
Given~$\boldsymbol{\sigma}$, the edge formation processes in the parent graph $G$ are mutually independent across pairs of vertices. 
Since the subsampling procedure is also independent across pairs of vertices, we have that 
\begin{equation}
\label{eq:marginal1}
\p(A,B \mid \pi_* = \pi, \boldsymbol{\sigma}) = 
\left( p_{00}^{\mu^+(\tau)_{00}} p_{01}^{\mu^+(\tau)_{01}} p_{10}^{\mu^+(\tau)_{10}} p_{11}^{\mu^+(\tau)_{11}} \right) 
\left( q_{00}^{\mu^-(\tau)_{00}} q_{01}^{\mu^-(\tau)_{01}} q_{10}^{\mu^-(\tau)_{10}} q_{11}^{\mu^-(\tau)_{11}} \right).
\end{equation}
To simplify this expression, note that we can write 
\begin{align*}
\mu^+(\tau)_{00} & = \sum\limits_{e \in \cE^+(\boldsymbol{\sigma})} \left( 1 - A_{e} \right) \left(1 - B_{\tau(e)} \right) = \sum\limits_{e \in \cE^+(\boldsymbol{\sigma})} \left(1 - A_{e} \right) - \nu^+(\tau) + \mu^+(\tau)_{11}, \\
\mu^+(\tau)_{01} & = \sum\limits_{e \in \cE^+(\boldsymbol{\sigma})} \left( 1 - A_{e} \right) B_{\tau(e)} = \nu^+(\tau) - \mu^+(\tau)_{11}, \\
\mu^+(\tau)_{10} & = \sum\limits_{e \in \cE^+(\boldsymbol{\sigma})} A_e \left(1 - B_{\tau(e)} \right) = \sum\limits_{e \in \cE^+(\boldsymbol{\sigma})} A_e  - \mu^+(\tau)_{11},
\end{align*}
with similar expressions for $\mu^{-}(\tau)_{ij}$. The only terms on the right hand sides above that depend on $\tau$ (and therefore $\pi$) are $\nu^+(\tau)$ and $\mu^+(\tau)_{11}$; the remaining terms only depend on $A$ and $\boldsymbol{\sigma}$. 
We therefore have that 
\[
p_{00}^{\mu^+(\tau)_{00}} p_{01}^{\mu^+(\tau)_{01}} p_{10}^{\mu^+(\tau)_{10}} p_{11}^{\mu^+(\tau)_{11}} 
= C(A, \boldsymbol{\sigma}) \left( \frac{p_{00} p_{11}}{ p_{01} p_{10}} \right)^{\mu^+(\tau)_{11}} \left( \frac{p_{01}}{p_{00}} \right)^{\nu^+(\tau)}
\]
where $C(A, \boldsymbol{\sigma})$ depends only on $A$ and $\boldsymbol{\sigma}$. 
A similar expression holds for the other factor in~\eqref{eq:marginal1}, 
with $p_{ij}$ replaced with $q_{ij}$, $\mu^{+}(\tau)_{ij}$ replaced with $\mu^{-}(\tau)_{ij}$, and $\nu^{+}(\tau)$ replaced with $\nu^{-}(\tau)$. 
Plugging these back into~\eqref{eq:marginal1} we obtain~\eqref{eq:likelihood}. 
\end{proof}

Recall that 
$p_{00}, q_{00} = 1 - o(1)$ as $n \to \infty$, 
and that 
$p_{01}, p_{10}, p_{11}, q_{01}, q_{10}, q_{11}$ 
are all on the order $\log(n) / n$, 
implying that 
$p_{00} p_{11} > p_{01} p_{10}$ and 
$q_{00} q_{11} > q_{01} q_{10}$ 
for all $n$ large enough. 
Thus a useful consequence of Lemma~\ref{lemma:likelihood} is that $\p \left(\pi_* = \pi \, \middle| \,  A,B,\boldsymbol{\sigma} \right)$ is increasing in $\mu^+(\tau)_{11}$ and $\mu^-(\tau)_{11}$, and decreasing in $\nu^+(\tau)$ and $\nu^-(\tau)$. Building on these observations, the following results establish conditions under which two lifted permutations, $\tau$ and $\tau'$, satisfy $\mu^+(\tau)_{11} \ge \mu^+(\tau')_{11}$ or $\nu^+(\tau) = \nu^+(\tau')$, with similar statements about $\mu^-$ and $\nu^-$. 
These will be used later to analyze the performance of the MAP estimator.

\begin{proposition}
\label{prop:nu} 
Given $\boldsymbol{\sigma}$, the following holds. 
Let $\pi_a, \pi_b \in \cS_n$. If 
\begin{equation}
\label{eq:nu_condition}
\pi_{a}(V_{+}) = \pi_{b}(V_{+}) \qquad \text{and} \qquad \pi_{a}(V_{-}) = \pi_{b}(V_{-}),
\end{equation}
then $\nu^+(\ell(\pi_a)) = \nu^+(\ell(\pi_b))$ and $\nu^-(\ell(\pi_a)) = \nu^-(\ell(\pi_b))$.
\end{proposition}

\begin{proof}
We prove the claim for $\nu^+$; the other claim follows from identical arguments. First note that $\cE^+(\boldsymbol{\sigma}) = \binom{V_{+}}{2} \cup \binom{V_{-}}{2}$, so we can write 
\begin{equation}
\label{eq:nu1}
\nu^+(\ell(\pi_a)) = \sum\limits_{(i,j) \in \binom{V_{+}}{2}} B_{\pi_{a}(i), \pi_{a}(j)} + \sum\limits_{(i,j) \in \binom{V_{-}}{2}} B_{\pi_{a}(i), \pi_{a}(j)}.
\end{equation}
In light of the assumption~\eqref{eq:nu_condition}, the mapping $\pi_{a}^{-1} \circ \pi_{b}: V_{+} \to V_{+}$ is a bijection. The first summation on the right hand side of~\eqref{eq:nu1} is therefore equal to 
\[
\sum\limits_{(i,j) \in \binom{V_{+}}{2}} B_{\left( \pi_{a} \circ \pi_{a}^{-1} \circ \pi_{b} \right)(i), \left(\pi_{a} \circ \pi_{a}^{-1} \circ \pi_{b} \right)(j)} = \sum\limits_{(i,j) \in \binom{V_{+}}{2}} B_{\pi_{b}(i), \pi_{b}(j)}.
\]
Similarly, since $\pi_{a}^{-1} \circ \pi_{b} :V_{-} \to V_{-}$ is a bijection in light of \eqref{eq:nu_condition}, the second summation on the right hand side of \eqref{eq:nu1} is equal to 
\[
\sum\limits_{(i,j) \in \binom{V_{-}}{2}} B_{\pi_{b}(i), \pi_{b}(j)}.
\]
Plugging the previous two displays back into~\eqref{eq:nu1} we obtain that 
$\nu^+(\ell(\pi_a)) = \nu^+(\ell(\pi_b))$.  
\end{proof}

\begin{proposition}
\label{prop:mu}
Let $\tau_a$ and $\tau_b$ be lifted permutations 
such that  whenever $A_{e} B_{\tau_{b}(e)} = 1$ 
we also have that 
$A_{e}B_{\tau_{a}(e)} = 1$. 
Then $\mu^+(\tau_a)_{11} \ge \mu^+(\tau_b)_{11}$ and $\mu^-(\tau_a)_{11} \ge \mu^-(\tau_b)_{11}$.
\end{proposition}
 
\begin{proof}
The condition on $\tau_a$ and $\tau_b$ in the statement  implies that $A_{e} B_{\tau_{a}(e)} \ge A_{e} B_{\tau_{b}(e)}$ for all $e \in \cE$, and the desired result follows from the formulas for $\mu^+$ and $\mu^-$. 
\end{proof}

\subsection{Performance of the MAP estimator and proof of Theorem~\ref{thm:graph_matching_impossibility}}

The following lemma shows how one may use the simple propositions above to bound the probability that the MAP estimator outputs a given permutation. Before stating the lemma, we recall a few properties of the MAP estimator. The estimator is formally given by
\begin{equation}\label{eq:MAP_estimator}
\wh{\pi}_{\MAP} \in \argmax_{\pi \in \cS_n} \p \left( \pi_* = \pi \, \middle| \, A, B, \boldsymbol{\sigma} \right).
\end{equation}
In words, $\wh{\pi}_{\MAP}$ is the mode of the posterior distribution $\{\p \left(\pi_* = \pi \, \middle| \, A, B, \boldsymbol{\sigma} \right)\}_{\pi \in \cS_n}$. When the argmax set is not a singleton, $\wh{\pi}_{\MAP}$ is a uniform random element of the argmax set. The MAP estimator is \emph{optimal}, in the sense that it minimizes the probability of error (see, e.g., \cite[Chapter~4]{poor_book}). 

For $\pi \in \cS_n$, define the set 
\[
T^\pi \equiv T^{\pi}(A,B) : = \left \{ i \in [n]: \forall j \in [n], A_{i, j} B_{\pi(i),\pi(j)} = 0 \right \},
\]
as well as $T^\pi_{+} : = T^\pi \cap V_{+}$ and $T^\pi_{-} : = T^\pi \cap V_{-}$. 
Note that 
$T_{+}^{\pi}$ and $T_{-}^{\pi}$ 
are functions of 
$A$, $B$, and~$\boldsymbol{\sigma}$. 
In words, if $\pi$ is the true vertex correspondence and $i \in T^{\pi}$, 
then the neighbors of $i$ in $G_{1}$ and the neighbors of $i$ in $G_{2}'$ are disjoint sets. 
Due to the lack of overlapping information, it becomes difficult for the MAP estimator to correctly match $i$ in $G_{1}$ with its counterpart $\pi(i)$ in $G_{2}$. 
The following lemma formalizes this (where we use the standard convention that $0! = 1$). 

\begin{lemma}[MAP estimator]
\label{lemma:map} 
For all $n$ large enough 
and for any $\pi \in \cS_{n}$ we have that 
\[
\p \left( \wh{\pi}_{\MAP} = \pi \, \middle| \, A,B, \boldsymbol{\sigma} \right) \le \frac{1}{\left|T_{+}^{\pi} \right|! \cdot \left|T_{-}^{\pi} \right|!}.
\]
\end{lemma}

\begin{proof}%[Proof of Lemma \ref{lemma:map}]
Fix $\pi \in \cS_{n}$ and suppose that $A$, $B$, and $\boldsymbol{\sigma}$ are given. 
Let $\rho_{1}$ be any permutation of $T_{+}^{\pi}$ and 
let $\rho_{2}$ be any permutation of $T_{-}^{\pi}$. 
Construct a new permutation 
$\pi' = \pi' \left( \pi, \rho_{1}, \rho_{2} \right)$ 
as follows: 
\begin{itemize}
    \item For $i \in \left[ n \right] \setminus T^{\pi}$, let $\pi'(i) := \pi(i)$. 
    \item For $i \in T_{+}^{\pi}$, let $\pi'(i) := \pi \left( \rho_{1}(i) \right)$.  
    \item For $i \in T_{-}^{\pi}$, let $\pi'(i) := \pi \left( \rho_{2}(i) \right)$.  
\end{itemize}
Let $\cT$ be the set of permutations $\pi'$ constructed in this way. 
Since each choice of $\rho_{1}$ and $\rho_{2}$ leads to a distinct $\pi'$, 
we have that 
$\left| \cT \right| = \left| T_{+}^{\pi} \right|! \cdot \left| T_{-}^{\pi} \right|!$. 

A useful consequence of this construction is that 
$\pi' \left( V_{+} \right) = \pi \left( V_{+} \right)$ 
and $\pi' \left( V_{-} \right) = \pi \left( V_{-} \right)$. 
By Proposition \ref{prop:nu}, this implies that 
\begin{equation}\label{eq:nu_comparison}
\nu^{+} \left( \ell \left( \pi' \right) \right) = \nu^{+} \left( \ell \left( \pi \right) \right) 
\qquad 
\text{ and }
\qquad 
\nu^{-} \left( \ell \left( \pi' \right) \right) = \nu^{-} \left( \ell \left( \pi \right) \right). 
\end{equation}

Furthermore, note that if $A_{i,j} B_{\pi(i), \pi(j)} = 1$, 
then we must have $i,j \in [n] \setminus T^{\pi}$ by definition. 
The construction of $\pi'$ implies that $\pi'(i) = \pi(i)$ and $\pi'(j) = \pi(j)$ for such $i$ and $j$. 
Hence we have that 
$A_{i,j} B_{\pi'(i), \pi'(j)} = A_{i,j} B_{\pi(i), \pi(j)} = 1$ 
for such $i$ and $j$. 
By Proposition~\ref{prop:mu} we thus have that 
\begin{equation}\label{eq:mu_comparison}
\mu^{+}(\ell(\pi'))_{11} \geq \mu^{+}(\ell(\pi))_{11} 
\qquad 
\text{ and }
\qquad 
\mu^{-}(\ell(\pi'))_{11} \geq \mu^{-}(\ell(\pi))_{11}. 
\end{equation}

In light of Lemma~\ref{lemma:likelihood}, 
as well as the observations on monotonicity made after its proof, 
\eqref{eq:nu_comparison} and~\eqref{eq:mu_comparison} 
together imply, for all $n$ large enough, that 
\begin{equation}
\label{eq:pi_ordering}
\p \left( \pi_* = \pi \, \middle| \, A, B, \boldsymbol{\sigma} \right) \le \p \left( \pi_* = \pi' \, \middle| \, A, B, \boldsymbol{\sigma} \right).
\end{equation}
Now we distinguish two cases. 
First, if $\pi$ is not a maximizer of 
$\left\{ \p \left( \pi_* = \wt{\pi} \, \middle| \, A, B, \boldsymbol{\sigma} \right) \right\}_{\wt{\pi} \in \cS_{n}}$, 
then we have that 
$\p \left( \wh{\pi}_{\MAP} = \pi \, \middle| \, A, B, \boldsymbol{\sigma} \right) = 0$, 
so the claim holds trivially. 
On the other hand, if $\pi$ is a maximizer of 
$\left\{ \p \left( \pi_* = \wt{\pi} \, \middle| \, A, B, \boldsymbol{\sigma} \right) \right\}_{\wt{\pi} \in \cS_{n}}$, 
then (by~\eqref{eq:pi_ordering}) so is $\pi'$ for every $\pi' \in \cT$. 
Therefore the set 
$\argmax_{\wt{\pi} \in \cS_n} \p \left( \pi_* = \wt{\pi} \, \middle| \, A, B, \boldsymbol{\sigma} \right)$ 
has at least $\left| \cT \right|$ elements. 
Since $\wh{\pi}_{\MAP}$ picks an element of the argmax set uniformly at random, 
this implies that 
\[
\p \left( \wh{\pi}_{\MAP} = \pi \, \middle| \, A, B, \boldsymbol{\sigma} \right) \le \frac{1}{\left| \cT \right|} 
= \frac{1}{\left| T_{+}^{\pi} \right|! \cdot \left| T_{-}^{\pi} \right|!}. \qedhere
\] 
\end{proof}

Next, the following lemma establishes lower bounds for $\left|T_{+}^{\pi} \right|$ and $\left|T_{-}^{\pi} \right|$ in the case where $\pi$ is the ground truth vertex permutation. 
Before stating the result, for $\pi \in \cS_n$ we define the measure  
$\p_\pi(\cdot) := \p \left( \cdot \, \middle| \, \pi_{*} = \pi \right)$. 
Additionally, let $\E_{\pi}$ and $\mathrm{Var}_{\pi}$ denote the expectation and variance operators corresponding to the measure $\p_\pi$. 

\begin{lemma}
\label{lemma:T_sizes}
Suppose that 
$s^2 \left( \alpha + \beta \right) / 2 < 1$. 
Then there exists $\gamma = \gamma(\alpha, \beta, s) > 0$ such that 
\[
\lim\limits_{n \to \infty} \min\limits_{\pi \in \cS_n} \p_\pi \left( \left|T_{+}^{\pi} \right|, \left|T_{-}^{\pi} \right| \ge n^{\gamma} \right) = 1.
\]
\end{lemma}

The proof of the lemma is based on estimating the first and second moments of $\left|T_{+}^{\pi} \right|$ and $\left|T_{-}^{\pi} \right|$ under the measure $\p_\pi$. While the proof techniques are quite standard, the proof is somewhat tedious, so we defer it to Section~\ref{subsec:T_szes}.

We are now ready to prove the impossibility result for graph matching in correlated SBMs.

\begin{proof}[Proof of Theorem~\ref{thm:graph_matching_impossibility}] 
As mentioned before, we prove a stronger claim; 
namely, we show that even if $\boldsymbol{\sigma}$ is provided as extra information, 
for any estimator $\wt{\pi} = \wt{\pi}(G_{1}, G_{2}, \boldsymbol{\sigma})$ 
we have that 
$\lim_{n \to \infty} \p \left( \wt{\pi} = \pi_{*} \right) = 0$.  
To this end, we study the MAP estimator 
$\wh{\pi}_{\MAP} = \wh{\pi}_{\MAP} \left( A, B, \boldsymbol{\sigma} \right)$ 
of $\pi_{*}$ given $A$, $B$, and $\boldsymbol{\sigma}$ (see~\eqref{eq:MAP_estimator}). 
Since the MAP estimator minimizes the probability of error (see, e.g., \cite[Chapter~4]{poor_book}), it suffices to show that 
$\lim_{n \to \infty} \p \left( \wh{\pi}_{\MAP} = \pi_{*} \right) = 0$. 

To compute/bound 
$\p \left( \wh{\pi}_{\MAP} = \pi_{*} \right)$,  
we may first condition on $\pi_{*}$ and then on $A$, $B$, and $\boldsymbol{\sigma}$. Since $\pi_{*} \in \cS_{n}$ is uniformly random, we have that 
\[
\p \left( \wh{\pi}_{\MAP} = \pi_{*} \right) 
= 
\frac{1}{n!} \sum_{\pi \in \cS_{n}} \sum_{A, B, \boldsymbol{\sigma}} \p \left( \wh{\pi}_{\MAP} = \pi \, \middle| \, A, B, \boldsymbol{\sigma}, \pi_{*} = \pi \right) \p \left( A, B, \boldsymbol{\sigma} \, \middle| \, \pi_{*} = \pi \right).
\]
Note that $\wh{\pi}_{\MAP}$ is a function of $A$, $B$, and $\boldsymbol{\sigma}$ (and perhaps additional randomness, in case the maximizer of the posterior distribution is not unique). 
Therefore 
$\p \left( \wh{\pi}_{\MAP} = \pi \, \middle| \, A, B, \boldsymbol{\sigma}, \pi_{*} = \pi \right)
= \p \left( \wh{\pi}_{\MAP} = \pi \, \middle| \, A, B, \boldsymbol{\sigma} \right)$, 
that is, we may remove the event $\left\{ \pi_{*} = \pi \right\}$ from the conditioning. Plugging this back into the display above and using the bound of Lemma~\ref{lemma:map} we obtain that 
\begin{equation}\label{eq:map_bound}
\p \left( \wh{\pi}_{\MAP} = \pi_{*} \right) 
\leq 
\frac{1}{n!} \sum_{\pi \in \cS_{n}}
\E_{\pi} \left[ \frac{1}{\left|T_{+}^{\pi} \right|! \cdot \left|T_{-}^{\pi} \right|!} \right],
\end{equation}
where the expectation is over $A$, $B$, and $\boldsymbol{\sigma}$ (recall that $T_{+}^{\pi}$ and $T_{-}^{\pi}$ are functions of $A$, $B$, and~$\boldsymbol{\sigma}$). 
Let $\gamma = \gamma \left(\alpha, \beta, s \right) > 0$ be the constant given by Lemma~\ref{lemma:T_sizes}, and 
for $\pi \in \cS_n$ define the event 
$\cA_{\pi} : = \left\{ \left|T_{+}^{\pi} \right|, \left|T_{-}^{\pi} \right| \ge n^{\gamma} \right\}$.  
By definition we have that 
\[
\E_{\pi} \left[ \frac{1}{\left|T_{+}^{\pi} \right|! \cdot \left|T_{-}^{\pi} \right|!} \right] 
\leq \frac{1}{\left( n^{\gamma} ! \right)^{2}} + \p_{\pi} \left( \cA_{\pi}^{c} \right).
\]
Plugging this into~\eqref{eq:map_bound} we thus have that 
\[
\p \left( \wh{\pi}_{\MAP} = \pi_{*} \right) 
\leq \frac{1}{\left( n^{\gamma} ! \right)^{2}} + \frac{1}{n!} \sum_{\pi \in \cS_{n}} \p_{\pi} \left( \cA_{\pi}^{c} \right) 
\leq \frac{1}{\left( n^{\gamma} ! \right)^{2}} + \max_{\pi \in \cS_{n}} \p_{\pi} \left( \cA_{\pi}^{c} \right).
\]
Both terms on the right hand side go to $0$ as $n \to \infty$; 
the latter term converging to $0$ as $n \to \infty$ is due to Lemma~\ref{lemma:T_sizes}. 
\end{proof}

\subsection{Lower bounding \texorpdfstring{$\left|T_{+}^{\pi}\right|$}{|T+pi|} and \texorpdfstring{$\left|T_{-}^{\pi}\right|$}{|T-pi|}: Proof of Lemma~\ref{lemma:T_sizes}}
\label{subsec:T_szes}

Fix $\pi \in \cS_{n}$; throughout the proof we condition on the event $\left\{ \pi_{*} = \pi \right\}$. 
Given also $\boldsymbol{\sigma}$, we have that 
\[
A_{i, j} B_{\pi(i),\pi(j)} \sim \begin{cases}
\mathrm{Bernoulli} \left( s^2 \alpha \frac{\log n}{n} \right) &\text{ if } (i,j) \in \cE^+(\boldsymbol{\sigma}) \\
\mathrm{Bernoulli} \left( s^2 \beta \frac{\log n}{n} \right) &\text{ if } (i,j) \in \cE^-(\boldsymbol{\sigma}).
\end{cases}
\]
Moreover, for fixed $i \in [n]$ the random variables 
$\left\{ A_{i,j} B_{\pi(i),\pi(j)} \right\}_{j \in [n] \setminus \{i\}}$ 
are mutually independent (given $\left\{ \pi_{*} = \pi \right\}$ and $\boldsymbol{\sigma}$). 
Hence if $i \in V_{+}$, then we have that 
\[
\p_{\pi} \left( i \in T^{\pi} \, \middle| \, \boldsymbol{\sigma} \right) 
= \left( 1 - s^{2} \alpha \frac{\log n}{n} \right)^{\left| V_{+} \right| - 1} \left( 1 - s^{2} \beta \frac{\log n}{n} \right)^{\left| V_{-} \right|}.
\]
Note that $\left| V_{+} \right|$ and $\left| V_{-} \right|$ are typically approximately $n/2$, and hence the conditional probability above is typically approximately $n^{-s^{2}(\alpha + \beta)/2}$. To make this precise, we introduce some further notation. For $\epsilon \in (0,1)$ define 
\begin{align*}
    \delta &:= 1 - \left( 1 + \epsilon / 2 \right)^{2} s^{2} \left( \alpha + \beta \right) / 2, \\
    \lambda &:= 1 - \left( 1 - \epsilon \right) s^{2} \left( \alpha + \beta \right) / 2.
\end{align*} 
In the following we fix $\epsilon \in (0,1)$ such that 
\begin{equation}\label{eq:params}
\delta > 0 
\qquad 
\text{ and }
\qquad 
\lambda > 0 
\qquad 
\text{ and }
\qquad 
\lambda < 2 \delta.
\end{equation}
Such an $\epsilon \in (0,1)$ exists due to the assumption that $s^{2} \left( \alpha + \beta \right) / 2 < 1$. 
Recall that on the event $\cF_{\epsilon}$ we have that 
$\left| V_{+} \right|, \left| V_{-} \right| \leq (1+ \epsilon /2) n/2$. 
Thus if $\boldsymbol{\sigma}$ is such that the event $\cF_\epsilon$ holds, 
then 
\begin{align*}
\log \p_{\pi} \left( i \in T^{\pi} \, \middle| \, \boldsymbol{\sigma} \right) 
&\ge \left( 1 + \frac{\epsilon}{2} \right) \frac{n}{2} \left( \log \left( 1 - s^{2} \alpha \frac{\log n}{n} \right) + \log \left( 1 - s^{2} \beta \frac{\log n}{n} \right) \right)  \\
& \ge \left( 1 + \frac{\epsilon}{2} \right)^2 \frac{n}{2} \left( - s^2 (\alpha + \beta) \frac{\log n}{n} \right)  
= \left( \delta - 1 \right) \log n,
\end{align*}
where the second inequality holds for all $n$ large enough, 
since $\log(1-x) \geq -(1+\epsilon/2)x$ for all $x>0$ small enough. 
Thus, on the event $\cF_{\epsilon}$ we have that 
$\p_{\pi} \left( i \in T^{\pi} \, \middle| \, \boldsymbol{\sigma} \right) 
\geq n^{\delta - 1}$ 
for all $n$ large enough. 
By linearity of expectation this gives a lower bound on the (conditional) expectation of $\left| T_{+}^{\pi} \right|$: 
if $\boldsymbol{\sigma}$ is such that $\cF_\epsilon$ holds, then for all $n$ large enough we have that 
\begin{equation}
\label{eq:ET_lower_bound}
\E_{\pi} \left[ \left|T_{+}^{\pi} \right| \, \middle| \, \boldsymbol{\sigma} \right] 
\geq \left| V_{+} \right| n^{\delta - 1} 
\geq \frac{1-\epsilon/2}{2} n^{\delta} 
\geq \frac{1}{4} n^{\delta}.
\end{equation}

To establish a probabilistic lower bound for $\left| T_{+}^{\pi} \right|$, 
we proceed by bounding its (conditional) variance. 
For $i \in [n]$ let $X_{i} := \mathbf{1} \left( i \in T_{+}^{\pi} \right)$ 
be the indicator variable that $i \in T_{+}^{\pi}$. 
We then have that 
\begin{equation}\label{eq:var_decomposition}
\Var_{\pi} \left( \left| T_{+}^{\pi} \right| \, \middle| \, \boldsymbol{\sigma} \right) 
= \Var_{\pi} \left( \sum_{i \in V_{+}} X_{i} \, \middle| \, \boldsymbol{\sigma} \right) 
= \sum_{i \in V_{+}} \Var_{\pi} \left( X_i \, \middle| \, \boldsymbol{\sigma} \right) + \sum_{i,j \in V_{+} : i \neq j} \Cov_{\pi} \left( X_{i}, X_{j} \, \middle| \,  \boldsymbol{\sigma} \right).
\end{equation}
For the variance terms on the right hand side, we use the bound
\[
\Var_{\pi} \left( X_i \, \middle| \, \boldsymbol{\sigma} \right) 
\leq \p_{\pi} \left( i \in T^{\pi} \, \middle| \, \boldsymbol{\sigma} \right) 
\leq \exp \left( - s^{2} \left( \alpha \left( \left| V_{+} \right| - 1 \right) + \beta \left| V_{-} \right| \right) \frac{\log n}{n} \right).
\]
If $\boldsymbol{\sigma}$ is such that $\cF_\epsilon$ holds, then using the bounds $\left| V_{+} \right| - 1 \geq (1-\epsilon) n/2$ and $\left| V_{-} \right| \geq (1-\epsilon) n/2$ we thus have that 
\begin{equation}\label{eq:var_bound}
\Var_{\pi} \left( X_i \, \middle| \, \boldsymbol{\sigma} \right) 
\leq n^{\lambda - 1}. 
\end{equation}
The covariance terms can be computed as 
\begin{multline*}
\Cov_{\pi} \left( X_{i}, X_{j} \, \middle| \, \boldsymbol{\sigma} \right) 
= \E_{\pi} \left[ X_{i} X_{j} \, \middle| \, \boldsymbol{\sigma} \right] 
- \E_{\pi} \left[ X_{i} \, \middle| \, \boldsymbol{\sigma} \right] \E_{\pi} \left[ X_{j} \, \middle| \, \boldsymbol{\sigma} \right] \\
\begin{aligned}
&= \left( 1 - s^{2} \alpha \frac{\log n}{n} \right)^{2 \left| V_{+} \right| - 3} 
\left( 1 - s^{2} \beta \frac{\log n}{n} \right)^{2 \left| V_{-} \right|} 
- \left( 1 - s^{2} \alpha \frac{\log n}{n} \right)^{2 \left| V_{+} \right| - 2} 
\left( 1 - s^{2} \beta \frac{\log n}{n} \right)^{2 \left| V_{-} \right|}\\
&= s^{2} \alpha \frac{\log n}{n} \left( 1 - s^{2} \alpha \frac{\log n}{n} \right)^{2 \left| V_{+} \right| - 3} 
\left( 1 - s^{2} \beta \frac{\log n}{n} \right)^{2 \left| V_{-} \right|} \\
&\leq s^{2} \alpha \frac{\log n}{n} \exp \left( - s^{2} \left( \alpha \left( 2 \left| V_{+} \right| - 3 \right) + \beta \left( 2 \left| V_{-} \right| \right) \right) \frac{\log n}{n} \right).
\end{aligned}
\end{multline*}
If $\boldsymbol{\sigma}$ is such that $\cF_\epsilon$ holds, then using the bounds $2\left| V_{+} \right| - 3 \geq (1-\epsilon) n$ and $2 \left| V_{-} \right| \geq (1-\epsilon) n$ we thus have that 
\begin{equation}\label{eq:cov_bound}
\Cov_{\pi} \left( X_{i}, X_{j} \, \middle| \, \boldsymbol{\sigma} \right) 
\leq 
\left( s^{2} \alpha \log(n) \right) n^{-1 - \left( 1 - \epsilon \right) s^{2} \left( \alpha + \beta \right)} 
%= \left( s^{2} \alpha \log(n) \right) n^{-1 - 2(1-\lambda)}
= \left( s^{2} \alpha \log(n) \right) n^{2\lambda - 3}.
\end{equation}
Plugging~\eqref{eq:var_bound} and~\eqref{eq:cov_bound} back into~\eqref{eq:var_decomposition}, we have that 
\[
\Var_{\pi} \left( \left| T_{+}^{\pi} \right| \, \middle| \, \boldsymbol{\sigma} \right) 
\leq 
n \cdot n^{\lambda - 1} + n^{2} \cdot \left( s^{2} \alpha \log(n) \right) n^{2\lambda - 3} 
= n^{\lambda} + \left( s^{2} \alpha \log(n) \right) n^{2\lambda - 1}.
\]
whenever $\boldsymbol{\sigma}$ is such that $\cF_\epsilon$ holds. 
Since $\lambda < 1$, we have that $\lambda > 2 \lambda - 1$, and so the display above implies that 
\begin{equation}\label{eq:var_estimate}
\Var_{\pi} \left( \left| T_{+}^{\pi} \right| \, \middle| \, \boldsymbol{\sigma} \right) 
\leq 2 n^{\lambda}
\end{equation}
for all $n$ large enough, whenever $\boldsymbol{\sigma}$ is such that $\cF_\epsilon$ holds.

Next, we use Chebyshev's inequality to turn the first and second moment estimates into a probabilistic lower bound for $\left| T_{+}^{\pi} \right|$. 
If $\boldsymbol{\sigma}$ is such that $\cF_\epsilon$ holds, then, 
by~\eqref{eq:ET_lower_bound}, for all $n$ large enough such that $n^{\delta/2} \leq n^{\delta}/8$, we have that 
\[
\p_{\pi} \left( \left| T_{+}^{\pi} \right| \leq n^{\delta / 2} \, \middle| \, \boldsymbol{\sigma} \right) 
\leq 
\p_{\pi} \left( \left| \left| T_{+}^{\pi} \right| - \E_{\pi} \left[ \left| T_{+}^{\pi} \right| \, \middle| \, \boldsymbol{\sigma} \right] \right| \geq n^{\delta} / 8 \, \middle| \, \boldsymbol{\sigma} \right). 
%\leq 
%64 n^{-2\delta} \Var_{\pi} \left( \left| T_{+}^{\pi} \right| \, \middle| \, \boldsymbol{\sigma} \right).
\]
Thus by Chebyshev's inequality and~\eqref{eq:var_estimate} we have that 
\[
\p_{\pi} \left( \left| T_{+}^{\pi} \right| \leq n^{\delta / 2} \, \middle| \, \boldsymbol{\sigma} \right) 
\leq 
64 n^{-2\delta} \Var_{\pi} \left( \left| T_{+}^{\pi} \right| \, \middle| \, \boldsymbol{\sigma} \right) 
\leq 128 n^{\lambda - 2 \delta} 
\]
for all $n$ large enough, whenever $\boldsymbol{\sigma}$ is such that $\cF_\epsilon$ holds. 
Recall from~\eqref{eq:params} that $\lambda - 2\delta < 0$, so this bound decays to $0$ as $n \to \infty$. 

To remove the conditioning on $\boldsymbol{\sigma}$, we can write 
\[
\p_{\pi} \left( \left| T_{+}^{\pi} \right| \geq n^{\delta / 2} \right) 
\geq 
\E \left[ \p_{\pi} \left( \left| T_{+}^{\pi} \right| \geq n^{\delta / 2} \, \middle| \, \boldsymbol{\sigma} \right) \mathbf{1} \left( \cF_{\epsilon} \right) \right] 
\geq 
\left( 1 - 128 n^{\lambda - 2 \delta} \right) \p \left( \cF_{\epsilon} \right).
\]
Note in particular that this lower bound holds \emph{uniformly} over all $\pi \in \cS_{n}$. 
Hence, since $\p \left( \cF_{\epsilon} \right) \to 1$ as $n \to \infty$, we have that 
\[
\lim_{n \to \infty} \min_{\pi \in \cS_{n}} \p_{\pi} \left( \left| T_{+}^{\pi} \right| \geq n^{\delta / 2} \right) = 1.
\]

Finally, the same arguments also hold for $\left| T_{-}^{\pi} \right|$ by symmetry, 
so the conclusion follows by a union bound.

\section{Impossibility of community recovery from correlated SBMs}
\label{sec:impossibility_community_recovery}

%In this section we prove Theorem~\ref{thm:community_recovery_impossibility}. 

\begin{proof}[Proof of Theorem~\ref{thm:community_recovery_impossibility}] 
The key idea is to reduce the problem to that of exact community recovery in the (classical) single-graph SBM setting. 
Specifically, as observed in the proof of Theorem~\ref{thm:community_recovery}, 
the union graph 
$H_{*} := G_{1} \lor_{\pi_{*}} G_{2}$ 
satisfies 
\[
H_{*} \sim \mathrm{SBM} \left( n, \alpha (1 - (1 - s)^2) \frac{\log n}{n} , \beta (1 - (1 - s)^2 ) \frac{\log n}{n} \right), 
\]
and from $H_{*}$ it is possible to simulate $G_{1}$ and $G_{2}$. 
However, under the condition~\eqref{eq:union_graph_impossibility}, exact community recovery is impossible from an SBM with such parameters~\cite{abbe2016exact, mossel2016consistency,abbe2015community,Abbe_survey}. 

To make the argument formal, 
suppose by way of contradiction that there exists an estimator 
$\widetilde{\boldsymbol{\sigma}} = \widetilde{\boldsymbol{\sigma}}(G_1, G_2)$ 
such that 
\begin{equation}\label{eq:contradiction}
\limsup_{n \to \infty} \p \left( \overlap \left( \wt{\boldsymbol{\sigma}}(G_{1}, G_{2}), \boldsymbol{\sigma} \right) = 1 \right) > 0.
\end{equation}

Now let $H$ be a graph on the vertex set $[n]$ satisfying 
\[
H \sim \mathrm{SBM} \left( n, \alpha (1 - (1 - s)^2) \frac{\log n}{n} , \beta (1 - (1 - s)^2 ) \frac{\log n}{n} \right), 
\]
and let $\boldsymbol{\sigma}_{H}$ denote the underlying community labels of $H$. 
Given $H$, we now construct two edge-subsampled graphs $H_{1}$ and $H_{2}'$ as follows. 
First, define the parameters 
\[
(r_{01}, r_{10}, r_{11}) 
:= 
\left( \frac{s(1-s)}{1-(1-s)^{2}}, \frac{s(1-s)}{1-(1-s)^{2}}, \frac{s^{2}}{1-(1-s)^{2}} \right)
\]
and note that $r_{01} + r_{10} + r_{11} = 1$, so this triple defines a probability distribution.  
Now for every vertex pair $(i,j)$ independently: 
\begin{itemize}
    \item if $(i,j)$ is not an edge in $H$, then it is not an edge in $H_{1}$ and it is not an edge in $H_{2}'$; 
    \item if $(i,j)$ is an edge in $H$, then 
    \begin{itemize}
        \item with probability $r_{10}$, the pair $(i,j)$ is an edge in $H_{1}$ but not an edge in $H_{2}'$;
        \item with probability $r_{01}$, the pair $(i,j)$ is not an edge in $H_{1}$ but it is an edge in $H_{2}'$; and
        \item with probability $r_{11}$, the pair $(i,j)$ is an edge in both $H_{1}$ and $H_{2}'$. 
    \end{itemize}
\end{itemize}
The key observation is that, by construction, 
$(H_{1}, H_{2}', \boldsymbol{\sigma}_{H})$ has the same distribution as $(G_{1}, G_{2}', \boldsymbol{\sigma})$. 
Now let $\pi \in \cS_{n}$ be a uniformly random permutation which is independent of everything else. 
Finally, we generate $H_{2}$ by relabeling the vertices of $H_{2}'$ according to $\pi$ (i.e., vertex $i$ in $H_{2}'$ is relabeled to $\pi(i)$ in $H_{2}$). 
Again by construction, 
$(H_{1}, H_{2}, \boldsymbol{\sigma}_{H})$ has the same distribution as $(G_{1}, G_{2}, \boldsymbol{\sigma})$. 
In particular, 
$\overlap \left( \wt{\boldsymbol{\sigma}}(H_{1}, H_{2}), \boldsymbol{\sigma}_{H} \right)$ 
and 
$\overlap \left( \wt{\boldsymbol{\sigma}}(G_{1}, G_{2}), \boldsymbol{\sigma} \right)$ 
have the same distribution, 
and so 
\[
\p \left( \overlap \left( \wt{\boldsymbol{\sigma}}(H_{1}, H_{2}), \boldsymbol{\sigma}_{H} \right) = 1 \right) 
= 
\p \left( \overlap \left( \wt{\boldsymbol{\sigma}}(G_{1}, G_{2}), \boldsymbol{\sigma} \right) = 1 \right).
\]
Combining this with~\eqref{eq:contradiction}, we have that 
\begin{equation}\label{eq:H_recovery}
\limsup_{n \to \infty} \p \left( \overlap \left( \wt{\boldsymbol{\sigma}}(H_{1}, H_{2}), \boldsymbol{\sigma}_{H} \right) = 1 \right) > 0.
\end{equation}
However, it is known~\cite{abbe2016exact, mossel2016consistency,abbe2015community,Abbe_survey} that if~\eqref{eq:union_graph_impossibility} holds, 
then for every estimator $\boldsymbol{\sigma}' = \boldsymbol{\sigma}'(H)$ 
(including randomized estimators)
we have that 
\begin{equation}\label{eq:H_nonrecovery}
\lim_{n \to \infty} \p \left( \overlap \left( \boldsymbol{\sigma}'(H), \boldsymbol{\sigma}_{H} \right) = 1 \right) = 0. 
\end{equation}
Since $(H_{1}, H_{2})$ was constructed from $H$ using only additional randomness, $\wt{\boldsymbol{\sigma}}(H_{1}, H_{2})$ can be thought of as a randomized estimator of $\boldsymbol{\sigma}_{H}$ which takes $H$ as input. 
Therefore~\eqref{eq:H_recovery} and~\eqref{eq:H_nonrecovery} are in direct contradiction. 
Thus~\eqref{eq:contradiction} does not hold, which proves the claim. 
\end{proof}

\section{Proofs for many correlated SBMs}
\label{sec:proofs_many_corr_SBMs}

In this section we prove our results that concern $K \geq 3$ correlated SBMs, 
namely Theorems~\ref{thm:community_recovery_K} and~\ref{thm:K_community_recovery_impossibility}. 
These proofs are analogous to the proofs of Theorems~\ref{thm:community_recovery} and~\ref{thm:community_recovery_impossibility}, extending them to the setting of $K \geq 3$ correlated SBMs.

\begin{proof}[Proof of Theorem~\ref{thm:community_recovery_K}] 
Given permutations $\pi^{2}, \ldots, \pi^{K} \in \cS_{n}$, 
we define $G_{1} \lor_{\pi^{2}} G_{2} \ldots \lor_{\pi^{K}} G_{K}$, 
the \emph{union graph with respect to $\pi^{2}, \ldots, \pi^{K}$}, 
as follows: 
for distinct $i$ and $j$, 
the pair $(i,j)$ is an edge in $G_{1} \lor_{\pi^{2}} G_{2} \ldots \lor_{\pi^{K}} G_{K}$ 
if and only if 
$(i,j)$ is an edge in $G_{1}$ 
or $\left( \pi^{k}(i), \pi^{k}(j) \right)$ is an edge in $G_{k}$ for some $k \in \left\{ 2, \ldots, K \right\}$. 
In particular, let 
$H_{*} := G_{1} \lor_{\pi_{*}^{2}} G_{2} \ldots \lor_{\pi_{*}^{K}} G_{K}$. 
By construction, $H_{*}$ is the subgraph of the parent graph~$G$ consisting of exactly the edges that are in $G_{1}$ or in $G_{k}'$ for some $k \in \{2, \ldots, K\}$.  
Thus we have that 
\[
H_{*} \sim \SBM \left( n, \alpha \left( 1 - \left( 1 - s \right)^{K} \right) \frac{\log n}{n}, \beta \left( 1 - \left( 1 - s \right)^{K} \right) \frac{\log n}{n} \right). 
\]

The algorithm we study first computes, for every $k \in \{ 2, \ldots, K \}$, 
the permutation $\wh{\pi}^{k} := \wh{\pi} \left( G_{1}, G_{k} \right)$ according to Theorem~\ref{thm:graph_matching}. 
We then pick any community recovery algorithm that is known to succeed until the information-theoretic limit, and run it on 
$\wh{H} := G_{1} \lor_{\wh{\pi}^{2}} G_{2} \ldots \lor_{\wh{\pi}^{K}} G_{K}$; 
we denote the result of this algorithm by 
$\wh{\sigma} ( \wh{H} )$. 
We can then write 
\begin{align*}
\p ( \overlap(\wh{\boldsymbol{\sigma}}(\wh{H}), \boldsymbol{\sigma}) \neq 1 ) 
&\le \p ( \{ \overlap(\wh{\boldsymbol{\sigma}}(\wh{H}), \boldsymbol{\sigma}) \neq 1 \} \cap \{ \wh{H} = H_* \} ) + \p(\wh{H} \neq H_* ) \\
&\le \p( \overlap(\wh{\boldsymbol{\sigma}}(H_*), \boldsymbol{\sigma}) \neq 1) + \sum_{k=2}^{K} \p\left( \wh{\pi}^{k} \neq \pi_\star^{k} \right),
\end{align*}
where, to obtain the inequality in the second line, 
we have used that $\wh{\boldsymbol{\sigma}}(\wh{H}) = \wh{\boldsymbol{\sigma}}(H_*)$ on the event $\{\wh{H} = H_* \}$, 
and that $\wh{H} \neq H_{*}$ implies that $\wh{\pi}^{k} \neq \pi_{*}^{k}$ for some $k \in \{2, \ldots, K\}$. 
Since exact community recovery on $H_{*}$ is possible when~\eqref{eq:K_community_recovery} holds~\cite{abbe2016exact, mossel2016consistency,abbe2015community,Abbe_survey}, 
we know that 
$\p(\overlap(\wh{\boldsymbol{\sigma}}(H_*), \boldsymbol{\sigma}) \neq 1) \to 0$ as $n \to \infty$. 
In light of Theorem~\ref{thm:graph_matching} we also have, for every $k \in \{2, \ldots, K\}$, that $\p( \wh{\pi}^{k} \neq \pi_{*}^{k}) \to 0$ when $s^{2}(\alpha + \beta)/2 > 1$, concluding the proof. 
\end{proof}

\begin{proof}[Proof of Theorem \ref{thm:K_community_recovery_impossibility}] 
Suppose, by way of contradiction, that there exists 
an estimator 
$\widetilde{\boldsymbol{\sigma}} 
= \widetilde{\boldsymbol{\sigma}}\left(G_{1}, G_{2}, \ldots, G_{K} \right)$ 
such that 
\begin{equation}\label{eq:contradiction_K}
\limsup_{n \to \infty} \p \left( \overlap \left( \wt{\boldsymbol{\sigma}}(G_{1}, G_{2}, \ldots, G_{K}), \boldsymbol{\sigma} \right) = 1 \right) > 0.
\end{equation}

Now let $H$ be a graph on the vertex set $[n]$ satisfying 
\[
H \sim \mathrm{SBM} \left( n, \alpha \left(1 - \left(1 - s\right)^{K} \right) \frac{\log n}{n} , \beta \left(1 - \left(1 - s\right)^{K} \right) \frac{\log n}{n} \right), 
\]
and let $\boldsymbol{\sigma}_{H}$ denote the underlying community labels of $H$. 
Given $H$, we now construct $K$ edge-subsampled graphs, $H_{1}, H_{2}', \ldots, H_{K}'$, as follows. 
First, for $x \in \{0,1\}^{K}$ let $|x| := \sum_{k=1}^{K} x_{k}$. 
For every $x \in \{0,1\}^{K}$ let $r_{x} := s^{|x|} (1-s)^{K-|x|} / \left( 1 - \left( 1 - s \right)^{K} \right)$, and note that 
$\sum_{x \in \{0,1\}^{K} \setminus 0^{K}} r_{x} = 1$, 
so $\boldsymbol{r} := \left\{ r_{x} \right\}_{x \in \{0,1\}^{K} \setminus 0^{K}}$ 
defines a probability distribution. 
Now for every vertex pair $(i,j)$ independently: 
\begin{itemize}
    \item if $(i,j)$ is not an edge in $H$, 
    then it is not an edge in any of $H_{1}, H_{2}', \ldots, H_{K}'$; 
    \item if $(i,j)$ is an edge in $H$, then draw $x \in \{0,1\}^{K} \setminus 0^{K}$ from the distribution $\boldsymbol{r}$. Then $(i,j)$ is an edge in $H_{1}$ if and only if $x_{1} = 1$, and for every $k \in \{2, \ldots, K\}$, the pair $(i,j)$ is an edge in $H_{k}'$ if and only if $x_{k} = 1$. 
\end{itemize}
The key observation is that, by construction, 
$\left( H_{1}, H_{2}', \ldots, H_{K}', \boldsymbol{\sigma}_{H} \right)$ 
has the same distribution as 
$\left( G_{1}, G_{2}', \ldots, G_{K}', \boldsymbol{\sigma} \right)$. 
Now let $\pi^{2}, \ldots, \pi^{K} \in \cS_{n}$ 
be i.i.d.\ uniformly random permutations 
which are independent of everything else. 
Finally, for every $k \in \{2, \ldots, K\}$, 
we generate $H_{k}$ by relabeling the vertices of $H_{k}'$ according to $\pi^{k}$ 
(i.e., vertex $i$ in $H_{k}'$ is relabeled to $\pi^{k}(i)$ in $H_{k}$). 
Again by construction, 
$\left( H_{1}, H_{2}, \ldots, H_{K}, \boldsymbol{\sigma}_{H} \right)$ 
has the same distribution as 
$\left( G_{1}, G_{2}, \ldots, G_{K}, \boldsymbol{\sigma} \right)$. 
In particular, 
$\overlap \left( \wt{\boldsymbol{\sigma}}(H_{1}, H_{2}, \ldots, H_{K}), \boldsymbol{\sigma}_{H} \right)$ 
and 
$\overlap \left( \wt{\boldsymbol{\sigma}}(G_{1}, G_{2}, \ldots, G_{K}), \boldsymbol{\sigma} \right)$ 
have the same distribution, 
and so 
\[
\p \left( \overlap \left( \wt{\boldsymbol{\sigma}}(H_{1}, H_{2}, \ldots, H_{K}), \boldsymbol{\sigma}_{H} \right) = 1 \right) 
= 
\p \left( \overlap \left( \wt{\boldsymbol{\sigma}}(G_{1}, G_{2}, \ldots, G_{K}), \boldsymbol{\sigma} \right) = 1 \right).
\]
Combining this with~\eqref{eq:contradiction_K}, we have that 
\begin{equation}\label{eq:H_recovery_K}
\limsup_{n \to \infty} \p \left( \overlap \left( \wt{\boldsymbol{\sigma}}(H_{1}, H_{2}, \ldots, H_{K}), \boldsymbol{\sigma}_{H} \right) = 1 \right) > 0.
\end{equation}
However, it is known~\cite{abbe2016exact, mossel2016consistency,abbe2015community,Abbe_survey} that if~\eqref{eq:K_community_impossibility} holds, 
then for every estimator $\boldsymbol{\sigma}' = \boldsymbol{\sigma}'(H)$ 
(including randomized estimators)
we have that 
\begin{equation}\label{eq:H_nonrecovery_K}
\lim_{n \to \infty} \p \left( \overlap \left( \boldsymbol{\sigma}'(H), \boldsymbol{\sigma}_{H} \right) = 1 \right) = 0. 
\end{equation}
Since $(H_{1}, H_{2}, \ldots, H_{K})$ was constructed from $H$ using only additional randomness, the estimator $\wt{\boldsymbol{\sigma}}(H_{1}, H_{2}, \ldots, H_{K})$ is a randomized estimator of $\boldsymbol{\sigma}_{H}$ which takes $H$ as input. 
Therefore~\eqref{eq:H_recovery_K} and~\eqref{eq:H_nonrecovery_K} are in direct contradiction. 
Thus~\eqref{eq:contradiction_K} does not hold, proving the claim. 
\end{proof}

%%%%%%%%%%%%%%%%%%%%%%%%
%%% Acknowledgements %%%
%%%%%%%%%%%%%%%%%%%%%%%%

\section*{Acknowledgements}

We thank Jasmine Nirody for help with figures. 

%%%%%%%%%%%%%%%%%%
%%% References %%%
%%%%%%%%%%%%%%%%%%

{\small
\bibliographystyle{abbrv}
\bibliography{references}
}

%%%%%%%%%%%%%%%%
%%% Appendix %%%
%%%%%%%%%%%%%%%%

% \newpage

% \appendix

\end{document}